\documentclass[11pt,letter,reqno]{amsart}
\usepackage[top=2.5cm, bottom=2.5cm, inner=2.5cm, outer=2.5cm, marginparwidth=2cm]{geometry}

\usepackage{amssymb, latexsym, amsmath, amsxtra, mathrsfs, bm}

\usepackage[dvips]{graphics}
\usepackage{xypic,xcolor}
\usepackage{subcaption}

\usepackage{verbatim}
\usepackage[abs]{overpic}
\usepackage{hyperref}
\usepackage{mathtools}

\usepackage{enumitem}

\DeclareMathAlphabet{\mathbbold}{U}{bbold}{m}{n}

\allowdisplaybreaks[3]

\theoremstyle{plain}
        \newtheorem{theorem}{Theorem}[section]
        \newtheorem*{theorem*}{Theorem}
        \newtheorem*{conj*}{Conjecture}
        \newtheorem{lemma}[theorem]{Lemma}
        \newtheorem{prop}[theorem]{Proposition}
        \newtheorem*{prop*}{Proposition}
        \newtheorem{cor}[theorem]{Corollary}
        \newtheorem*{cor*}{Corollary}

        \newtheorem{thmx}{Theorem}

\theoremstyle{definition}
        \newtheorem{definition}[theorem]{Definition}
        \newtheorem*{definition*}{Definition}
        \newtheorem{rem}[theorem]{Remark}

\theoremstyle{remark}
        \newtheorem*{remark}{Remark}

\numberwithin{equation}{section}
\numberwithin{theorem}{section}
\numberwithin{table}{section}
\numberwithin{figure}{section}

\makeatletter
\def\th@plain{%
	\thm@notefont{}
	\itshape 
}
\def\th@definition{%
	\thm@notefont{}
	\normalfont 
}
\makeatother

\providecommand{\defn}[1]{\emph{#1}}

\newcommand{\dist}{\operatorname{dist}}

\newcommand{\length}{\operatorname{length}}

\newcommand{\diam}  {\operatorname{diam}}

\newcommand{\inter}  {\operatorname{int}}

\newcommand{\id} {\operatorname{id}}

\newcommand{\card} {\operatorname{card}}
\newcommand{\R}{\mathbb{R}}
\newcommand{\B}{\mathbb{B}}    
\newcommand{\C}{\mathbb{C}}

\newcommand{\N}{\mathbb{N}}

\newcommand{\D}{\mathbb{D}}

\renewcommand{\le}{\leqslant}
\renewcommand{\leq}{\leqslant}
\renewcommand{\ge}{\geqslant}
\renewcommand{\geq}{\geqslant}

\renewcommand{\:}{\colon}

\newcommand{\mesh}{\operatorname{mesh}}

\providecommand{\abs}[1]{\lvert#1\rvert}
\providecommand{\Absbig}[1]{\bigl\lvert#1\bigr\rvert}

\providecommand{\norm}[1]{\|#1\|}

\renewcommand{\=}{\coloneqq}

\newcommand{\cA}{\mathcal{A}}
\newcommand{\cB}{\mathcal{B}}
\newcommand{\cC}{\mathcal{C}}
\renewcommand{\cD}{\mathcal{D}}
\newcommand{\cE}{\mathcal{E}}
\newcommand{\cF}{\mathcal{F}}

\newcommand{\cI}{\mathcal{I}}
\newcommand{\cJ}{\mathcal{J}}

\newcommand{\cM}{\mathcal{M}}
\newcommand{\cN}{\mathcal{N}}

\newcommand{\cP}{\mathcal{P}}

\newcommand{\cS}{\mathcal{S}}

\newcommand{\cU}{\mathcal{U}}
\newcommand{\cV}{\mathcal{V}}

\newcommand{\cX}{\mathcal{X}}
\newcommand{\cY}{\mathcal{Y}}

\newcommand{\fC}{\mathfrak{C}}
\newcommand{\fD}{\mathfrak{D}}

\newcommand{\fF}{\mathfrak{F}}

\newcommand{\fX}{\mathfrak{X}}
\newcommand{\fY}{\mathfrak{Y}}

\newcommand{\sA}{\mathscr{A}}

\newcommand{\hX}{\widehat{X}}

\newcommand{\tC}{\widetilde{C}}

\newcommand{\tH}{\widetilde{H}}

\newcommand{\tU}{\widetilde{U}}

\newcommand{\tW}{\widetilde{W}}
\newcommand{\tX}{\widetilde{X}}
\newcommand{\tY}{\widetilde{Y}}

\newcommand{\ta}{\widetilde{a}}
\newcommand{\tb}{\widetilde{b}}
\newcommand{\tc}{\widetilde{c}}

\newcommand{\tg}{\widetilde{g}}

\newcommand{\tm}{\widetilde{m}}

\newcommand{\ty}{\widetilde{y}}

\newcommand{\tgamma}{\widetilde{\gamma}}

\newcommand{\tdelta}{\widetilde{\delta}}

\newcommand{\tlambda}{\widetilde{\lambda}}
\newcommand{\tLambda}{\widetilde{\Lambda}}

\newcommand{\tsigma}{\widetilde{\sigma}}

\newcommand{\ttau}{\widetilde{\tau}}

\newcommand{\Sp}{\mathbb{S}}
\newcommand{\cellint}{\operatorname{int_\circ}}
\newcommand{\cellbound}{\partial_{\circ}}

\newcommand{\cls}[1]{\overline{#1}}
\newcommand{\supp}{\operatorname{Car}}
\newcommand{\flower}{\cF}

\newcommand{\coverF}{\fF}

\newcommand{\modul}{\operatorname{Mod}}

\renewcommand{\top}{[n]}

\newcommand{\loc}{{\operatorname{loc}}}
\newcommand{\celPair}{cellular pair}
\newcommand{\celSeq}{cellular sequence}
\newcommand{\MkvSeq}{cellular Markov sequence}
\newcommand{\itCel}{iterated cellular}
\newcommand{\ItCel}{Iterated cellular}
\newcommand{\Round}{\operatorname{Round}}
\newcommand{\mfd}{\cM}
\newcommand{\dg}{d_g}
\newcommand{\Isom}{\mathrm{Isom}}
\newcommand{\ind}{i}
\newcommand{\NxD}[2]{M_{#1}(#2)}
\newcommand{\di}{\mathrm{d}}
\newcommand{\mk}{\mathrm{p}}
\newcommand{\stdInt}[1]{\inter #1}
\newcommand{\stdCellint}[1]{\cellint #1}
\newcommand{\brCellint}[1]{\cellint(#1)}
\newcommand{\sk}[2]{\bigcup_{i\leq #2}#1^{[i]}}
\newcommand{\clcnR}{connected closed Riemannian}
\newcommand{\clcntop}{connected closed topological}
\newcommand{\clcnmtr}{connected closed metric}
\newcommand{\orclcnR}{oriented, connected, and closed Riemannian}
\newcommand{\orclcntop}{oriented, connected, and closed topological}

\newcommand{\jsOpSd}{joins opposite sides}

\newif\ifneworder
\newordertrue

\begin{document}

\title[Thurston-type maps]{Rigidity and quasisymmetric uniformization of Thurston-type maps}

\author{Zhiqiang Li, Pekka Pankka, and Hanyun Zheng}

\address{Zhiqiang~Li, School of Mathematical Sciences \& Beijing International Center for Mathematical Research, Peking University, Beijing 100871, China.}
\email{zli@math.pku.edu.cn}

\address{Pekka~Pankka, Department of Mathematics and Statistics, P.O. Box 68 (Pietari Kalmin katu 5), FI-00014 University of Helsinki, Finland}
\email{pekka.pankka@helsinki.fi}

\address{Hanyun~Zheng, School of Mathematical Sciences, Peking University, Beijing 100871, China.}
\email{1900013001@pku.edu.cn}



\begin{abstract} 
    We prove the 
    No Invariant Line Fields conjecture for a class of generalized postcritically-finite branched covers on higher-dimensional Riemannian manifolds. 
    Moreover, we establish a quasisymmetric uniformization theorem for this class of generalized postcritically-finite maps.
\end{abstract}

\maketitle

{
    \hypersetup{linkcolor=black}
    \setcounter{tocdepth}{2}
    \tableofcontents
}


\section{Introduction}\label{sct: intro}

In this article, we prove the
No Invariant Line Fields conjecture for a class of branched covers on higher-dimensional Riemannian manifolds. 

\begin{thmx}[No Invariant Line Fields]
\label{tx: UQR = Lattes}
Let $f\colon \mfd^n \to \mfd^n$, $n \geq 3$, be an expanding Thurston-type map on an {\orclcnR} $n$-manifold. Then $f$ is uniformly quasiregular if and only if $f$ is a Latt\`es map.
\end{thmx}

We define Thurston-type maps later in this introduction. On the $2$-dimensional sphere, 
Thurston-type maps are Thurston maps, i.e., postcritically-finite branched covering maps of degree at least two;
in all dimensions, the expansion is defined analogously to the two-dimensional maps; see e.g.~recent monograph of Bonk \& Meyer \cite{BM17} for the $2$-dimensional theory of expanding Thurston maps.
Uniformly quasiregular mappings are (weakly) conformal with respect to an intrinsic measurable conformal structure and in this sense a higher-dimensional counterpart for complex dynamics on the Riemann sphere. We refer to a survey of Martin \cite{Ma14} for a detailed discussion on uniformly quasiregular mappings.

\smallskip
Expanding Thurston-type maps have a natural class of metrics called visual metrics, for which we establish a quasisymmetric uniformization theorem. 

\begin{thmx}[Uniformization]
\label{tx: QS uniformization of visual metric}
	Let $f\colon \mfd^n \to \mfd^n$, $n \geq 3$, be an expanding Thurston-type map of an {\orclcnR} $n$-manifold $(\mfd^n,g)$, and let $\varrho$ be a visual metric for $f$ on $\mfd^n$. Then the identity map $\id\:(\mfd^n,\varrho)\to(\mfd^n,\dg)$ is quasisymmetric if and only if $f$ is uniformly quasiregular.
\end{thmx}

Following the work of Bonk \& Kleiner \cite{BK02} on Cannon's conjecture, and the reformulation of Cannon's conjecture as a quasisymmetric uniformization problem of visual metrics \cite{Bo06, Kl06}, the corresponding two-dimensional quasisymmetric uniformization result was established by Ha\"issinsky \& Pilgrim \cite{HP09} and Bonk \& Meyer \cite{BM10, BM17}.

\smallskip

We discuss first in this introduction the background for these statements and the role of uniformly quasiregular maps in these statements. The discussion on Thurston-type maps and the methods in the proofs is postponed until after this discussion.

\medskip
\noindent {\bf No Invariant Line Fields and Rigidity.}
The No Invariant Line Fields (abbreviated as NILF) conjecture anticipates that a rational map on the Riemann sphere $\widehat{\C}$ having a Julia set of positive measure is either a Latt\`es map or carries no invariant line field on its Julia set  (see \cite{Mc94a,Mc94b,McS98}). 
 The NILF conjecture can be formulated as follows: 
 \emph{If a rational map $f\:\widehat{\C}\to\widehat{\C}$ has a Julia set of positive measure and an invariant conformal structure that is non-trivial on its Julia set, then $f$ is a Latt\`es map.}
A higher-dimensional counterpart of this conjecture was introduced by Martin and Mayer in~\cite{MM03}, formulated as the following quasiregular NILF conjecture (\cite[Conjecture~1.4]{MM03}):
\begin{quote}
	\emph{Suppose $f\:\Sp^n\to\Sp^n$, $n \geq 3$, is a non-injective uniformly quasiregular map on the $n$-sphere having a Julia set of positive measure.
    Then $f$ is a Latt\`es map.}
\end{quote}
The analogy between this conjecture and the NILF conjecture lies in the following two points.  
First, quasiregular maps, which admit a Fatou--Julia theory (see e.g.~\cite{Be13,BFN14,BN14,FN11}), provide an analog of the dynamics of rational maps. Second, uniformly quasiregular maps are precisely the quasiregular maps having invariant conformal structures (cf.~\cite{IM96,IM01,Ma10}). Note that due to the Liouville theorem, an invariant conformal structure of a uniformly quasiregular map $f\:\Sp^n\to\Sp^n$ cannot be trivial unless $f$ is trivial, i.e., a M\"obius transformation. 

Recall that a continuous map $f\colon \mfd^n \to \cN^n$ between oriented Riemannian $n$-manifolds is \defn{quasiregular} if $f$ is in the Sobolev space $W^{1,n}_{\loc}(\mfd^n,\cN^n)$ and there exists $K \geq 1$ for which the distortion inequality
\begin{equation}	\label{eq:K}
	\norm{Df}^n  \leq  K J_f \quad \text{a.e.\ }\mfd^n
\end{equation}
holds, where $\norm{Df}$ and $J_f$ are the operator norm and the Jacobian, respectively, of the differential $Df$ of $f$. In this case, we say that $f$ is \defn{$K$-quasiregular}. In particular, a holomorphic map of one complex variable is a $1$-quasiregular map. In this terminology, a map is \defn{quasiconformal} if it is a quasiregular homeomorphism. 
Note that quasiregular maps are branched covers (see~Subsection~\ref{subsct: prelim on BC}) that are orientation-preserving.
We refer to the monographs of Rickman \cite{Ri93}, Reshetnyak \cite{Re89}, and Iwaniec \& Martin \cite{IM01} for detailed expositions on quasiregular theory.
Following Iwaniec \& Martin \cite{IM96} (see also \cite{IM01}), a quasiregular map $f\colon \mfd^n\to\mfd^n$ of an oriented Riemannian $n$-manifold is \defn{uniformly quasiregular} (abbreviated as \defn{UQR}) if there exists $K \geq  1$ for which each iterate $f^i$, $i\geq 1$, of $f$ is $K$-quasiregular.\footnote{Note that the uniform quasiregularity of a map is independent of the choice of the Riemannian metric of $\mfd^n$, since, for Riemannian metrics $g$ and $g'$ on $\mfd^n$, the identity map $\id \colon (\mfd^n,g) \to (\mfd^n,g')$ of a closed Riemannian $n$-manifold is quasiconformal.}
See a survey of Martin \cite{Ma14} for a detailed discussion.

An important subclass of uniformly quasiregular maps are the (quasiregular) Latt\`es maps originating from the work of Mayer~\cite{May97,May98}, which include the well-known Latt\`es maps on $\widehat\C$. In this article, we say that a uniformly quasiregular map $f\colon \mfd^n\to\mfd^n$ is a \defn{(quasiregular) Latt\`es map} if there exist (i)~a discrete subgroup $\Gamma < \mathrm{Isom}_+(\R^n)$, (ii)~a strongly $\Gamma$-automorphic quasiregular map $\varphi\colon \R^n \to \mfd^n$, and (iii)~a conformal affine map $A\colon \R^n \to \R^n$, $x\mapsto \lambda Ux+v$, where $U\in \mathrm{SO}(n)$, $\lambda>0$, and $v\in \R^n$, satisfying $A \Gamma A^{-1} \subseteq\Gamma$, for which the following diagram commutes; see~Definition~\ref{d: Lattes triple} and~\cite{Ka22} for more details.
\begin{equation*}
	\xymatrix{
		\R^n \ar[r]^A \ar[d]_\varphi & \R^n \ar[d]^\varphi \\
		\mfd^n \ar[r]^f & \mfd^n
	}
\end{equation*}

The topological rigidity statement in Theorem~\ref{tx: UQR = Lattes} has a metric antecedent in a rigidity theorem of Ha\"issinsky and Pilgrim \cite[Theorems~4.4.3 and~4.4.4]{HP09}: \emph{coarse expanding conformal systems on a closed Riemannian $n$-manifold $\mfd^n$, $n \ge 3$, coincide with chaotic Latt\`es maps}.
We return to the relationship between Thurston-type maps and coarse expanding conformal systems later in this introduction.

\medskip
\noindent{\bf Uniformization.}
We move now to the topic of our second main theorem, uniformization.
The classical uniformization problem asks whether a given region in $\widehat{\C}$ or a given Riemann surface is conformally equivalent to a model region, such as the unit disk $\D$, or more generally, a Riemann surface. For metric spaces, this uniformization problem takes the following form:

\smallskip

\begin{quote}
	{\sc The Quasisymmetric Uniformization Problem.} \it Given a metric space $(\fX,d)$ homeomorphic to some ``standard" metric space $(\fX_0,d_0)$, under what condition is $(\fX,d)$ quasisymmetrically equivalent to $(\fX_0,d_0)$?
\end{quote}

\smallskip

Recall that a homeomorphism $\varphi \colon (\fX,d_\fX) \to (\fY,d_\fY)$ between metric spaces $(\fX,d_\fX)$ and $(\fY,d_\fY)$ is \emph{$\eta$-quasisymmetric} for a homeomorphism $\eta \colon [0,+\infty) \to [0,+\infty)$ if, for $x,y,z\in \fX$ with $x\ne z$, the inequality
\begin{equation*}
d_\fY(\varphi(x),\varphi(y))  \leq  \eta\biggl( \frac{d_\fX(x,y)}{d_\fX(x,z)} \biggr) d_\fY(\varphi(x),\varphi(z))
\end{equation*}
holds. Metric spaces $(\fX,d_\fX)$ and $(\fY,d_\fY)$ are \emph{quasisymmetrically equivalent} if there exists an $\eta$-quasisymmetric homeomorphism $(\fX,d_\fX) \to (\fY,d_\fY)$ for some homeomorphism $\eta\colon [0,+\infty) \to [0,+\infty)$.
This metric notion of quasisymmetric maps was coined by Tukia and V\"ais\"al\"a in \cite{TV80}.

Heuristically, the quasisymmetric uniformization problem asks for a list of quantitative conditions of the metric space $(\fX,d)$ that allow a promotion of a homeomorphism $\fX\to \fX_0$ to a homeomorphism $(\fX,d)\to (\fX_0,d_0)$ of metric spaces respecting ratios of distances.

This uniformization problem was settled for metric $1$-spheres, i.e., metric circles, in a seminal work of Tukia and V\"ais\"al\"a \cite{TV80}: \emph{A metric circle $\bigl(\cS^1,d\bigr)$ is quasisymmetrically equivalent to the standard circle $\Sp^1$ if and only if $\bigl( \cS^1,d \bigr)$ is doubling and of bounded turning}.

The corresponding result for metric $2$-spheres is due to Bonk and Kleiner \cite{BK02}. A special case of their result states that \emph{a linearly locally contractible and Ahlfors regular $2$-sphere is quasisymmetrically equivalent to the standard $2$-sphere}; see \cite{BK02} for the full characterization, and Wildrick \cite{Wi08,Wi10} for uniformization results for metric surfaces. Quasisymmetric and quasiconformal uniformization of metric spheres and surfaces has been under active research; see e.g.~
Rajala \cite{Ra17}, Lytchak \& Wenger \cite{LW20}, Ikonen \cite{Ik22}, Ntalampekos \& Romney \cite{NR23}, and Meier \cite{Me24} for uniformization results of this type. For the terminology, see e.g.~Heinonen~\cite{He01} or Bonk \& Kleiner~\cite{BK02}.

For metric spheres of dimension at least three, no such characterizations are known. In particular, by a result of Semmes \cite{Se96} in dimension $n=3$ and by a result of Heinonen and Wu \cite{HW10} in dimensions $n\ge 4$, there are linearly locally contractible Ahlfors regular metric $n$-spheres $(\cS^n,d)$ that are not quasisymmetrically equivalent to the standard $n$-sphere $\mathbb{S}^n$; see also \cite{PW14} for a discussion. In particular, we are not aware of any quantitative \emph{metric} conditions that would yield quasisymmetric uniformization of metric $n$-spheres.

The uniformization theorem for visual metrics (Theorem \ref{tx: QS uniformization of visual metric}) shows that the quasisymmetric uniformization problem is solvable for visual metrics of expanding Thurston-type maps. 
In dimension $n=2$, the quasisymmetric uniformization was established by Bonk \& Meyer~\cite[Theorem~18.1]{BM17} and Ha\"issinsky \& Pilgrim \cite[Theorem~4.2.11]{HP09}:	\emph{Let $f \colon \cS^2 \to \cS^2$ be an expanding Thurston map on a topological $2$-sphere and let $\varrho$ be a visual metric for $f$ on $\cS^2$. Then $\bigl(\cS^2, \varrho\bigr)$ is quasisymmetrically equivalent to the Riemann sphere $\widehat\C$ equipped with the spherical metric if and only if $f$ is topologically conjugate to a rational map $\widehat \C \to \widehat \C$.}
Recall that each rational map is uniformly quasiregular, and by Sullivan's UQR characterization theorem \cite[Theorem~9]{Su83} \emph{each uniformly quasiregular map on $\widehat\C$ is quasiconformally conjugate to a rational map}. 

The results of Bonk \& Meyer and Ha\"issinsky \& Pilgrim may be viewed as a dynamical counterpart of Cannon's conjecture \cite{Ca94} in geometric group theory. We refer to Kleiner~\cite{Kl06} and Bonk~\cite{Bo06} (see also Bonk \& Meyer~\cite{BM17}) for a discussion on Cannon's conjecture in this context.

The construction of visual metrics for Thurston-type maps is analogous to the construction of visual metrics for Thurston maps \cite{BM10,BM17}, visual metrics of coarse expanding conformal systems \cite{HP09}, and further to visual metrics of Gromov hyperbolic groups \cite{GD90}.
We refer to Subsection~\ref{subsct: visual metrics} and Appendix~\ref{Ap: visual metric} for a construction.

\medskip

The quasisymmetric uniformization theorem has two immediate consequences on the structure of the underlying manifold. When combined with a result of Kangasniemi \cite{Ka21}, we have the following corollary on the structure of the de Rham cohomology $H^*_{\mathrm{dR}}(\mfd^n)$ of $\mfd^n$. \emph{A fortiori}, a result of Manin and Prywes \cite{ManPr24} yields that $\mfd^n$ has a virtually abelian fundamental group if the identity map $\id\:(\mfd^n,\varrho)\to(\mfd^n,\dg)$ is quasisymmetric. We combine these consequences as a corollary.

\begin{cor*}
    Let $f\colon \mfd^n \to \mfd^n$, $n \geq 3$, be an expanding Thurston-type map of an {\orclcnR} $n$-manifold $(\mfd^n,g)$, and let $\rho$ be a visual metric for $f$ on $\mfd^n$. If the identity map $\id\:(\mfd^n,\varrho)\to(\mfd^n,\dg)$ is quasisymmetric, then there exists an embedding of algebras $H^*_{\mathrm{dR}}(\mfd^n) \to \wedge^* \R^n$ and the fundamental group $\pi_1(M)$ is virtually abelian.
\end{cor*}

\begin{remark}
We note in passing that, in all dimensions $n\ge 2$, a closed manifold $\mfd^n$ admitting a non-injective uniformly quasiregular map is quasiregularly elliptic, that is, a manifold admitting a non-constant quasiregular map $\R^n \to \mfd^n$ from the Euclidean space $\R^n$. Whereas closed quasiregularly elliptic manifolds are classified up to dimension $n \le 4$ (see \cite{HeP25, ManPr24}), manifolds admitting non-injective UQR dynamics are classified only in dimensions $n=2,3$. In particular, by a result of Kangasniemi \cite{Ka21}, a quasiregularly elliptic $4$-manifold $\Sp^2\times \Sp^2 \# \Sp^2 \times \Sp^2$ does not admit a non-injective uniformly quasiregular map.
\end{remark}

\medskip

\noindent{\bf Thurston-type maps.}
Now we proceed to the definition of Thurston-type maps. Recall that a postcritically-finite (abbreviated as~PCF) branched cover $f\colon \cS^2 \to \cS^2$ having topological degree at least $2$ is called a \emph{Thurston map}. By postcritically-finite we mean that the postcritical set $P_f = \bigcup_{m\ge 1} f^m(B_f)$ is a finite set, where $B_f$ is the critical set of $f$. Recall that, by a simple Euler characteristic computation, branched covers on $2$-torus are covering maps and branched covers on higher-genus surfaces are homeomorphisms.
A Thurston map $f\colon \cS^2\to \cS^2$ is \emph{expanding} if there exists a finite cover $\cU$ of $\cS^2$ for which the mesh of the cover $(f^m)^*\cU$ tends to zero as $m\to \infty$, where $(f^m)^*\cU$ is the cover of $\cS^2$ given by connected components of the sets $f^{-m}(U)$ for $U\in \cU$, and \textit{mesh} of a finite cover is the maximal diameter of the sets in the cover. Note that Latt\`es maps on $\widehat\C$ are rational expanding Thurston maps with parabolic orbifolds. A rational Thurston map is expanding if and only if the Julia set is the whole Riemann sphere. In our higher-dimensional case, expansion also guarantees a full Julia set and one may anticipate the converse result, but this is beyond the scope of this article. We refer the reader to the monograph of Bonk and Meyer \cite{BM17} for a comprehensive exposition on expanding Thurston maps.

More generally, we follow \cite{HR02} and call
a continuous map $f\colon \fX\to \fY$ between topological spaces a \defn{(generalized) branched cover} if $f$ is discrete and open. Recall that $f$ is \defn{discrete} if the fibers $f^{-1}(y)$, $y\in \fY$, are discrete sets, and $f$ is \defn{open} if the images of open sets under $f$ are open sets. The classes of maps called \emph{branched covering maps} in \cite{BM17} and \emph{finite branched coverings} in \cite{HP09} are indeed discrete and open. In the context of orientation-preserving maps on {\orclcntop} manifolds or oriented surfaces, our definition of a branched cover coincides with these notions from \cite{HP09} and \cite{BM17} (cf.~\cite{LuP20}).

We denote the \emph{branch set} (or \emph{critical set}) and the \emph{post-branch set} (or \emph{postcritical set}) of $f$, respectively, by
\begin{equation*}
	B_f \= \{ x\in \fX : f \text{ is not a local homeomorphism at }x\}\quad\text{ and }\quad P_f \= \bigcup_{m \geq 1} f^m(B_f).
\end{equation*}

Readers familiar with higher-dimensional branched covers immediately observe that the PCF condition---in the definition of Thurston maps---does not lead to a fruitful theory in higher dimensions, since in dimensions $n\ge 3$ branched covers between $n$-manifolds do not have isolated critical points. In this article, we generalize the PCF condition using cell complex structures.

Here we give some terminology that will be discussed in detail later (see Section~\ref{sct: prelim}). A map $f\colon \fX'\to \fX$ is \emph{$(\cD',\cD)$-cellular} for cell decompositions (see Definition~\ref{d: cell decomposition}) $\cD'$ and $\cD$ of $\fX'$ and $\fX$, respectively, if, for each cell $c\in \cD'$, $f(c)\in\cD$ and the restriction $f|_c \colon c\to f(c)$ is a homeomorphism. In this case, $f$ defines an induced map by $f_*\:\cD'\to\cD,\,c\mapsto f(c)$.

We are now ready to introduce the following notion of cellularly-postcritical-finiteness as a generalization of postcritical-finiteness.

\begin{definition}[Cellularly-postcritically-finite]\label{def:CPCF}
	A branched cover $f\colon \mfd^n\to\mfd^n$ on an {\orclcntop} $n$-manifold is \defn{cellularly-postcritically-finite} (abbreviated as~\defn{CPCF}) if there exists a cell decomposition $\cP$ of the postcritical set $P_f$ having the following properties:
	\begin{enumerate}[label=(\alph*),font=\rm]
		\smallskip
		\item There exists a refinement $\cP_0$ of $\cP$ for which the restriction $f|_{P_f} \colon P_f \to P_f$ is a $(\cP_0,\cP)$-cellular map, and a cell decomposition $\cD_0$ of $\mfd^n$ for which $\cP_0 \subseteq \cD_0$. \label{item:CPCF-i}
		\smallskip
		\item There exists a cell decomposition $\cB$ of $B_f$ for which the restriction $f|_{B_f} \colon B_f \to P_f$ is a $(\cB,\cP)$-cellular map and, for each $c\in \cB$, the local index $\ind(\cdot,f)$ is constant on its cell-interior $\stdCellint{c}$, that is, the function $\ind(\cdot,f)|_{\stdCellint{c}} \colon \stdCellint{c} \to \N$, $x \mapsto \ind(x,f)$, is constant. \label{item:CPCF-ii}
	\end{enumerate}
	We call the quadruple $(\cB,\cP;\cP_0,\cD_0)$ \defn{(CPCF) data of $f$}.
\end{definition}

The role of Condition~\ref{item:CPCF-i} is two-fold. First, it postulates that the postcritical set $P_f$ is cellularly tame in the ambient space $\mfd^n$, i.e.,~that $P_f$ admits a cell decomposition that extends to a cell decomposition of the ambient space $\mfd^n$. This is a non-trivial topological requirement on the postcritical set $P_f$, which is trivially satisfied in dimension two. The second part of Condition~\ref{item:CPCF-i} is to postulate
the compatibility of the two cell decompositions on $P_f$, that is, the restriction $f|_{P_f} \colon P_f \to P_f$ is a cellular Markov map (see Definition~\ref{d: cellular Markov}) with respect to the given structures.
Condition~\ref{item:CPCF-ii}
requires that the branching of the map is encoded by a cellular structure of the branch set.

The definition of Thurston-type maps is analogous to the definition of Thurston maps.
\begin{definition}[Thurston-type maps] \label{def: Thurston type maps}
	A continuous map $f\colon \mfd^n\to\mfd^n$ on an {\orclcntop} $n$-manifold is of \defn{Thurston-type} if $f$ is a CPCF 
	branched cover of topological degree at least $2$.
\end{definition}

Clearly two-dimensional PCF branched covers are CPCF. Hence Thurston maps are Thurston-type maps.
It is also easy to observe that the idea of \emph{two-tile subdivision rules} on $2$-spheres (see \cite[Chapter~12]{BM17}) can be extended to higher-dimensional cases and yield
a concrete subclass of Thurston-type maps in all dimensions $n\ge 3$.
Similarly, it can be verified that orthotopic Latt\`es maps are Thurston-type maps. We call a Latt\`es map $f\colon \mfd^n \to \mfd^n$ \emph{orthotopic} if a fundamental domain of the group $\Gamma$ can be realized as an orthotope $\prod_{i=1}^n [0,a_i] \subseteq \R^n$ and $f$ is semi-conjugated to a scaling $\R^n \to \R^n$ by a strongly $\Gamma$-automorphic quasiregular map $h\colon \R^n \to \mfd^n$ (cf.\ Definition~\ref{d: Lattes triple}). We have decided not to discuss orthotopic Latt\'es maps further in this paper and instead return to the topic elsewhere. See \cite{LZ25b} for a detailed discussion; see also Astola, Kangaslampi \& Peltonen \cite{AKP10} for concrete examples of orthotopic Latt\`es maps.

\begin{remark}
Although the classes of uniformly quasiregular maps and Thurston-type maps intersect, neither of these classes is contained in the other---even after conjugation. Indeed, for $n \geq 3$, 
the branch set $B_f$ of a quasiregular map $\Sp^n \to \Sp^n$ need not have a cellular structure by a construction of Heinonen and Rickman \cite{HR98}.
Further, by a result of Martin and Peltonen \cite{MarPe10}, the branch set of a quasiregular map on $\Sp^n$ is realizable as the branch set of a uniformly quasiregular map on $\Sp^n$. This shows that there are UQR maps that are not Thurston-type maps. Similarly, not all expanding Thurston-type maps are UQR. For instance, one may construct examples where the finite-forward-index condition (see (\ref{e: ffi})) fails.
\end{remark}

\medskip
\noindent{\bf Thurston-type maps as {\itCel} branched covers.}
As a key property of Thurston maps, a Jordan curve passing through the postcritical points of a Thurston map on $\cS^2$ induces a natural backward iterated sequence of tile structures of the sphere; see \cite[Chapter~5]{BM17}. In higher dimensions, the strong version of the Jordan curve theorem is not at our disposal. We may use the CPCF data of a Thurston-type map $\mfd^n\to \mfd^n$ to obtain an analogous backward iterated sequence of cell decompositions over the decomposition $\cD_0$ of the manifold $\mfd^n$. 

\begin{thmx}[Iterated cellular decompositions] \label{tx: cellular sequence}
	Let $f\colon \mfd^n \to \mfd^n$, $n \geq 3$, be a Thurston-type map of an {\orclcntop} $n$-manifold $\mfd^n$. Then there exists a sequence $\{\cD_m\}_{m\in\N_0}$ of cell decompositions of $\mfd^n$ such that $f$ is $(\cD_{m+1},\cD_m)$-cellular for each $m\in\N_0$.
\end{thmx}
We call a sequence $\{\cD_m\}_{m\in\N_0}$ as in Theorem \ref{tx: cellular sequence} a \defn{\celSeq} of $f$, and a map $\mfd^n \to \mfd^n$ having a {\celSeq} an \emph{{\itCel} map}.
For Thurston maps, the aforementioned sequence of cell decompositions associated to a Jordan curve containing the postcritical points is a {\celSeq} in this terminology.

Similar to Thurston maps, a Thurston-type map need not be cellular Markov, i.e., cell decomposition $\cD_1$ need not refine cell decomposition $\cD_0$.
However, the converse holds: A branched cover $\mfd^n \to \mfd^n$ having the cellular Markov property is a Thurston-type map; see Proposition~\ref{p: Mkv => CPCF}.

Having {\celSeq}s at our disposal, we may characterize the expansion of a Thurston-type map via backward contraction along a {\celSeq}. More precisely, an {\itCel} branched cover $f\colon \mfd^n \to \mfd^n$ is expanding if and only if $f$ has a {\celSeq} $\{ \cD_m \}_{m\in \N}$ for which the meshes of cell decompositions $\cD_m$ tend to zero as $m\to +\infty$; see Subsection~\ref{subsct: cellular Expansion}.

\medskip
\noindent{\bf Metric theory of expanding Thurston-type maps.}
As in \cite{BM17}, we say that a metric $\varrho$ on $\mfd^n$ is a \emph{visual metric of $f$} if there exist a {\celSeq} $\{\cD_m\}_{m\in \N_0}$ of $f$ and constants $C>0$ and $\Lambda>1$ having the property that, for all $x,\,y\in \mfd^n$,
\begin{equation*}
\frac{1}{C} \Lambda^{-m(x,y)} \le \varrho(x,y) \le C \Lambda^{-m(x,y)},
\end{equation*}
where $m(\cdot, \cdot)$ is the separation level of $x$ and $x'$ with respect to $\{\cD_m\}_{m\in \N_0}$; see \eqref{e: the constant m(x,y)} for the definition.

It is almost immediate from the definition of the visual metric that an expanding {\itCel} branched cover $f\colon \mfd^n \to \mfd^n$ has good metric properties with respect to its visual metric. Indeed, $f\colon (\mfd^n,\varrho) \to (\mfd^n,\varrho)$ is a \emph{uniform expanding branched quasisymmetric map}.

Recall that a map $f\colon \fX \to \fY$ between metric spaces $(\fX,d_\fX)$ and $(\fY,d_\fY)$ is an \defn{$\eta$-branched quasisymmetric} (abbreviated as~\defn{$\eta$-BQS}) \defn{map} for a homeomorphism $\eta \colon [0,+\infty) \to [0,+\infty)$ if, for all pairs of intersecting \defn{continua}, i.e.,~non-singleton connected compact subsets $E$ and $F$ in $\fX$,
\begin{equation*}
	\diam_\fY(f(E))  \leq  \eta\biggl( \frac{\diam_\fX E}{\diam_\fX F}\biggr) \diam_\fY(f(F)).
\end{equation*}
The name---introduced by Guo and Williams in \cite{GW16}---for these maps stems from the observation that homeomorphic branched quasisymmetric maps between bounded turning spaces coincide with quasisymmetries.

In this article, we introduce a dynamical version of branched quasisymmetric maps called uniform expanding branched quasisymmetric maps.

\begin{definition}[Uniform expanding BQS map]\label{d: UEBQS}
	Let $\fX$ be a compact locally-connected metric space, and $f\:\fX\to \fX$ be a branched cover. We call $f$ a \defn{uniform expanding branched quasisymmetric} (abbreviated as~\defn{uniform expanding BQS}) \defn{map} if there exists a finite cover $\cU$ of $\fX$ by connected open subsets of $\fX$ and a homeomorphism $\eta\:[0,+\infty)\to[0,+\infty)$ such that the following conditions are satisfied:
	\begin{enumerate}[label=(\roman*),font=\rm]
		\smallskip
		\item (Expansion.) $\max\{\diam U: U\in(f^m)^*\cU\}\to 0$ as $m\to+\infty$.
		\smallskip
		\item (Uniformity.) For each $m\in\N$ and each $U\in (f^m)^*\cU$, the map $f^m|_{U}\:U\to f^m(U)$ is $\eta$-BQS.
	\end{enumerate}
	Here $(f^m)^*\cU\=\{U\subseteq\mfd^n:U\text{ is a connected component of }f^{-m}(V)\text{ for some }V\in\cU\}$ is the pullback of $\cU$. We call $(\cU,\eta)$ \defn{(uniform expanding BQS) data} of $f$.
\end{definition}

Notwithstanding the terms ``uniform'' and ``expanding'', uniform expanding BQS maps, being branched covers, are not necessarily \emph{uniformly expanding} (i.e., the distance between each pair of sufficiently close points is stretched by at least a uniform factor $\lambda>1$ under an iteration). Having the notion of uniform expanding branched quasisymmetries at our disposal, we may view 
Theorems~\ref{tx: UQR = Lattes} and \ref{tx: QS uniformization of visual metric}
as instances in the following list of equivalent conditions for expanding {\itCel} branched covers.

\begin{thmx}
\label{tx: QS-UEBQS-Lattes-UQR}
	Let $f\colon \mfd^n\to\mfd^n$, $n\geq 3$, be an orientation-preserving expanding {\itCel} branched cover
	on an {\orclcnR} $n$-manifold $\mfd^n$, and $\varrho$ be a visual metric for $f$. Then the following statements are equivalent:
	\begin{enumerate}[label=(\roman*),font=\rm]
		\smallskip
		\item The identity map $\id \colon (\mfd^n,\varrho) \to (\mfd^n,\dg)$ is a quasisymmetry. \label{item:UQR-1}
		\smallskip
		\item $f \colon (\mfd^n,\dg) \to (\mfd^n,\dg)$ is a uniform expanding BQS map. \label{item:UQR-2}
		\smallskip
		\item $f\colon \mfd^n\to\mfd^n$ is UQR. \label{item:UQR-3}
		\smallskip
		\item $f \colon \mfd^n\to\mfd^n$ is a (chaotic) Latt\`es map.\label{item:UQR-4}
		\smallskip
		\item $f \colon (\mfd^n,\dg) \to (\mfd^n,\dg)$ is metric CXC. \label{item:UQR-5}
	\end{enumerate}
\end{thmx}
Following~\cite{HMM04}, we say that a Latt\`es map $f\colon \mfd^n \to \mfd^n$ is \defn{chaotic} if the group $\Gamma$ of $f$ is cocompact and $f$ is expanding as a Latt\`es map, i.e.,~$\lambda>1$. In this case, the Julia set of $f$ is the whole manifold $\mfd^n$. Note that expanding Latt\`es maps are chaotic (see Lemma~\ref{l: Lattes: top exact => chaotic}).

\medskip
Theorem~\ref{tx: QS-UEBQS-Lattes-UQR} is established in three parts.
First, we show that conditions~\ref{item:UQR-1} and \ref{item:UQR-2} are equivalent, independently from the other conditions. Second, we show the equivalence of \ref{item:UQR-2}, \ref{item:UQR-4}, and \ref{item:UQR-5} in the following order:
\begin{equation}\label{e: implications: UBQS-Lattes-CXC}
	\xymatrix{
		\text{unif.~exp.~BQS} \ar@{=>}[r] & \text{conical UQR} \ar@{=>}[r] & \text{Latt\`es} \ar@{=>}[r] & \text{CXC} \ar@{=>}[r]
		& \text{unif.~exp.~BQS},
	}
\end{equation}
where \emph{conical UQR} refers to uniformly quasiregular maps $\mfd^n \to \mfd^n$ having the property that every point of $\mfd^n$ is a \emph{conical point}. We recall related notions and statements in Appendix~\ref{Ap: conical pt of UQR}. The implication ``$\text{conical UQR}\Rightarrow\text{Latt\`es}$" follows from a theorem of Martin and Mayer \cite[Theorem~1.3]{MM03}, which shows that a uniformly quasiregular map having a positive-measure conical set is a Latt\`es map. Note that it is an open question whether a uniformly quasiregular map having a positive-measure Julia set also has a positive-measure set of conical points. The implication ``$\text{Latt\`es}\Rightarrow\text{CXC}$" is due to Ha\"issinsky and Pilgrim \cite[Theorem~4.4.3]{HP09}.
The third part is the equivalence of \ref{item:UQR-3} and \ref{item:UQR-4}, which is also independent of the other conditions.
For this, we first show that---under the expanding {\itCel} branched cover assumption---UQR maps have \emph{finite forward index} (abbreviated as~\emph{ffi}), that is,
\begin{equation}\label{e: ffi}
	\sup\{ i(x,f^m) : m\in \N_0,\,x\in\mfd^n\} <+\infty;
\end{equation}
see Proposition~\ref{p: CPCF: UQR => bd multi}. Using this property and the combinatorial properties of the map, we conclude in the following order:
\begin{equation}
	\label{e: implications: UQR-Lattes-UQR}
	\xymatrix{
		\text{UQR} \ar@{=>}[r] & \text{ffi UQR} \ar@{=>}[r] & \text{conical UQR} \ar@{=>}[r] & \text{Latt\`es}\ar@{=>}[r] 
		& \text{UQR}.
	}
\end{equation}

The implications in (\ref{e: implications: UBQS-Lattes-CXC}) hold without additional conditions on the cellular structure of the map. We formulate this as the following theorem, which can be viewed as an analog of \cite[Theorems~4.4.3 and~4.4.4]{HP09}. The implication ``$\text{unif.\ exp.\ BQS}\Rightarrow\text{Latt\`es}$" parallels \cite[Theorem~4.4.4]{HP09} in the sense that, in both cases, the maps are UQR and each point on $\mfd^n$ is proved to be a conical point.
\begin{thmx}\label{tx: UeBQS-Lattes}
	Let $f\colon \mfd^n\to\mfd^n$, $n\geq3$, be an orientation-preserving branched cover on an {\orclcnR} $n$-manifold. Then $f\colon (\mfd^n,\dg) \to (\mfd^n,\dg)$ is a uniform expanding BQS map if and only if $f$ is a chaotic Latt\`es map.
\end{thmx}

Theorem~\ref{tx: UeBQS-Lattes} implies that all uniform expanding BQS maps on $\mfd^n$ satisfy the ffi property (\ref{e: ffi}). Note that the ffi condition plays an important role from the perspective of ergodic theory. For example, in the two-dimensional case, for expanding Thurston maps, this condition is equivalent to the asymptotic $h$-expansiveness and the upper semi-continuity of the entropy map; cf.~\cite{Li15,LS24}. In addition, for CXC systems, a related condition (the axiom [Deg]) helps to establish the asymptotic $h$-expansiveness in~\cite{LZ25a}, from which we obtain:
\begin{cor*}
	An orientation-preserving uniform expanding BQS on an {\orclcnR} $n$-manifold of dimension $n \geq 3$ is asymptotically $h$-expansive and thus, admits an equilibrium state for each continuous potential function.
\end{cor*}

Ergodic theory of Thurston maps and related dynamical systems has been studied in the literature \cite{BM17, DPTUZ21, HP09, HRL19, Li15, Li18, LS24}.
We expect that the cellular structure of Thurston-type maps facilitates a similar ergodic theory in higher dimensions.

\medskip
\noindent{\bf Organization of the article.}
This article is organized as follows. Section~\ref{sct: prelim} is dedicated to preliminaries on branched covers, cell decompositions, and related notions. In Section~\ref{sct: cellular maps}, we discuss notions related to expanding Thurston-type maps. In particular, we establish that Thurston-type maps are iterated cellular branched covers through an inductive construction of a {\celSeq}.

In Section~\ref{sct: QS uniformization of visual}, we consider an aspect of quasisymmetric uniformization of visual metrics, giving a direct proof of (i)$\Longleftrightarrow$(ii) in Theorem~\ref{tx: QS-UEBQS-Lattes-UQR}. There are two key ingredients: first, a class of open sets called \emph{cellular neighborhoods} constructed from cells in the {\celSeq} (Subsection~\ref{subsct: cellular neighborhood}), which serves as a basic tool in this section and Section~\ref{sct: Rigidity of UQR}; then a list of metric properties of cells in the {\celSeq}, which is in the spirit of the quasi-visual approximation by Bonk and Meyer~\cite{BM22}.

In Section~\ref{sct: Rigidity of UQR}, we prove the implications in \eqref{e: implications: UQR-Lattes-UQR}, which also give a direct proof for Theorem~\ref{tx: UQR = Lattes}. The proof is divided into three parts. First, in Subsection~\ref{subsct: marking}, we introduce a notion called \emph{markings for {\celSeq}s}, which is a key combinatorial tool that allows us to give metric estimates in the coming subsection. Then, in Subsection~\ref{subsct: cellular UQR => Lattes}, we establish two key lemmas: one is a lemma on the finite-forward-index (ffi) property (\ref{e: ffi}); the other is a metric estimate on the cells in the {\celSeq}, necessary for the implication ``ffi UQR $\Rightarrow$ conical UQR". These two lemmas crucially rely on the {\celSeq} and properties of quasiregular maps, and the combinatorial tools constructed in Subsections~\ref{subsct: cellular neighborhood} and~\ref{subsct: marking} play an important role in the proof. We also note that the metric estimate is given via moduli of curve families, while in the two-dimensional theory (\cite{BM17}), Koebe's distortion theorem of rational maps is used.
Finally, in the same subsection, we prove ``ffi UQR $\Rightarrow$ conical UQR" and then establish Theorem~\ref{tx: UQR = Lattes}.

Finally, in Section~\ref{sct: UEBQS <=> Lattes}, we prove Theorem~\ref{tx: UeBQS-Lattes} (which includes the equivalence of (ii), (iv), and (v) in Theorem~\ref{tx: QS-UEBQS-Lattes-UQR}). We utilize the metric properties of the uniform expanding BQS maps to estimate the distortion of the map, which leads to the conclusion that uniform expanding BQS maps on Riemannian manifolds are UQR maps with full conical sets. This, combined with the discussions in previous sections, establishes Theorems~\ref{tx: QS uniformization of visual metric} and \ref{tx: QS-UEBQS-Lattes-UQR}.

We conclude with an appendix recalling notions and results related to quasiregular mappings, branched quasisymmetries, conical points, and visual metrics.

\medskip
\noindent{\bf Acknowledgments.} 
The authors would like to express their sincere gratitude to
Mario~Bonk, Romain~Dujardin, Peter~Ha\"issinsky, Daniel~Meyer, Kevin~M.~Pilgrim, and Ruicen~Qiu for their helpful comments. The authors are also grateful for the hospitality of Urgench State University, where the authors had fruitful discussions and initiated their collaboration on this article.
Z.~Li and H.~Zheng were partially supported by Beijing Natural Science Foundation (JQ25001, 1214021) and National Natural Science Foundation of China (12471083, 12101017, 12090010, and 12090015). P.~Pankka was partially supported by the Research Council of Finland project~\#332671 and Centre of Excellence FiRST.

\medskip
\noindent{\bf Notation.}
Let $\N\=\{1,\,2,\,3,\,\dots\}$ be the set of positive integers, and $\N_0\=\{0\}\cup\N$. As usual, we denote by $\log_a$ the logarithm to the base $a>0$, and by $\log$ the natural logarithm function.

For a metric space $(X,d)$, the distance between $a$ and $b$ is denoted by $\abs{a-b}$ when the metric is clear from the context. The diameter of a subset $A$ is $\diam_d A\=\sup\{d(a,b) : a,b\in A\}$. The distance between subsets $A$ and $B$ is $\dist_d(A,B)\=\inf\{d(a,b):a\in A,\,b\in B\}$. Denote by $B_d(a,r)\=\{x\in X:d(a,x)<r\}$ the open metric ball and by $S_d(a,r)\=\{x\in X:d(a,x)=r\}$ the metric sphere centered at $a\in X$ of radius $r>0$. For a cover $\cC$ of $(X,d)$, we denote 
\begin{equation}\label{e: mesh(C)}
	\mesh_d(\cC)\=\sup\{\diam_d C:C\in\cC\}.
\end{equation}
We omit $d$ in the above notations if the metric is clear from the context.

We denote by $\B^n$ the open unit ball in $\R^n$, and by $\B^n(a,r)$ the open ball in $\R^n$ at $a\in\R^n$ with radius $r>0$. Denote by 
$\Sp^n\=\bigl\{(x_1,\,\dots,\, x_{n+1})\in\R^{n+1}:x_1^2+\cdots+x_{n+1}^2=1\bigr\}$
the $n$-dimensional unit sphere in $\R^{n+1}$, and by $\cS^n$ a topological $n$-sphere, i.e., a topological space homeomorphic to $\Sp^n$.

The spherical metric $\sigma$ on $\Sp^n$ is defined as usual. We call a metric $\tsigma$ on $\cS^n$ a spherical metric if there exists an isometry $\varrho\:(\cS^n,\tsigma)\to(\Sp^n,\sigma)$.

Throughout this article, unless otherwise stated, we use $\mfd^n$ to denote $n$-manifolds. For a Riemannian $n$-manifold $(\mfd^n,g)$, we denote by $\dg$ the Riemannian distance on $\mfd^n$.

We use $\id$ to denote the identity map on a subset $D$ of a set $\Omega$, i.e., the map $\id\:D\to\Omega$, $x\mapsto x$. We use $I_n$ (or simply $I$ when there is no ambiguity) to denote the $n$-dimensional identity matrix.

\section{Preliminaries}\label{sct: prelim}
\subsection{Branched covers}\label{subsct: prelim on BC}
Our discussions in this article concentrate on branched covers on manifolds. In the literature, there are various definitions of branched covers between topological spaces, most of which coincide in the context of topological manifolds; see e.g.,~\cite{GW16,HR02,HP09}.
In this article, we adopt the formulation in \cite{HR02}: A discrete, open, and continuous map $f\colon \fX\to \fY$ between topological spaces $\fX$ and $\fY$ is called a \defn{branched cover}. 

For a map $f\:\fX\to \fY$ between topological spaces and a subset $A\subseteq \fX$, we denote
\begin{equation}\label{e: N(f,A)}
	N(f,A)\=\sup\bigl\{\card\bigl(f^{-1}(y)\cap A\bigr):y\in f(A)\bigr\},
\end{equation}
and $N(f)\=N(f,\fY)$.
The \defn{local multiplicity}\footnote{This notion is also called \emph{local index}, \emph{local degree}, etc., in different contexts.} of $f$ at a point $x\in \fX$ is defined as
\begin{equation}\label{e: N_loc(f,x)}
	\ind(x,f)\=\inf\{N(f,U):U\subseteq \fX\text{ is an open neighborhood of }x\}.
\end{equation}

Recall that, given a branched cover $f\:\fX\to \fY$, a point $x\in \fX$ is a \defn{branch point} if $f$ is not locally homeomorphic at $x$. In terms of local multiplicity, $x$ is a branch point if and only if $\ind(x,f) \geq 2$. The set of all branch points is denoted by $B_f$. It is easy to see that the branch set is closed.

Let $\fX$ be a compact locally-connected Hausdorff space, $f\:\fX\to\fX$ be a branched cover, and $\cU$ be a cover of $\fX$ by connected open subsets. The \defn{pullback} of $\cU$ by $f$ is the open cover
\begin{equation}\label{e: pullback of cover}
	\begin{aligned}
		f^*\cU\=\bigl\{\tU : \text{there exists }U\in\cU
		\text{ for which }\tU\text{ is a connected component of }f^{-1}(\cU)\bigr\}.
	\end{aligned}
\end{equation}
It can be checked that $f^*g^*\cU=(g\circ f)^*\cU$, and thus $(f^*)^n\cU=(f^n)^*\cU$ for each $n\in\N$. 

\begin{lemma}\label{l: open+closed => f(tU)=U}
    Let $\fX$ be a locally-connected topological space, $f\:\fX\to\fX$ be continuous, open, and closed. Then for each connected open subset $U$ of $\fX$ and each connected component $\tU$ of $f^{-1}(U)$, we have $f\bigl(\tU\bigr)=U$.
\end{lemma} 
\begin{proof}
Set
$F\=\tU\cup\bigl(\fX\smallsetminus f^{-1}(U)\bigr)=\fX\smallsetminus\bigcup\bigl(\cC\smallsetminus\bigl\{\tU\bigr\}\bigr)$, 
where $\cC$ is the collection of all connected components of $f^{-1}(U)$. Then $F$ is a closed subset of $\fX$.
Since $f$ is open and closed, the subset $f\bigl(\tU\bigr)=f(F)\cap U$ is an open and closed subset of $U$. The connectedness of $U$ implies $f\bigl(\tU\bigr)=U$.
\end{proof}
\begin{rem}\label{r: forward inv of pullback}
    Since a continuous map between compact Hausdorff spaces is closed (cf.~\cite[Section~26]{Mu00}), by Lemma~\ref{l: open+closed => f(tU)=U} the pullback defined in (\ref{e: pullback of cover}) satisfies $f\bigl(\tU\bigr)\in\cU$ for each $\tU\in f^*\cU$.
\end{rem}

We now recall the notion of expansion defined in terms of the pullback of covers.
\begin{definition}[Expansion]\label{d: expansion}
	Let $\fX$ be a compact locally-connected metric space and $f\:\fX\to\fX$ be a branched cover. We call $f$ \defn{expanding} if there exists a finite cover $\cU$, consisting of connected open subsets, that satisfies $\mesh((f^m)^*\cU)\to 0$ as $m\to+\infty$. Here $\mesh(\cC)\=\sup\{\diam C:C\in\cC\}$ for a cover $\cC$ of $\fX$.
	In this case, we also say that $f$ is \defn{expanding with respect to $\cU$}.
\end{definition}

The above definition of expansion yields a natural idea to estimate the size of a connected set.
\begin{definition}[Approximation using covers]\label{d: approximate}
    Let $\fX$ be a compact locally-connected metric space, $\cU$ be a finite cover $\cU$ of $\fX$ by connected open sets, and $f\:\fX\to\fX$ be a branched cover expanding with respect to $\cU$. For $m\in\N_0$ and $U\in (f^m)^*\cU$, we call \defn{$U$ a $(f,\cU)$-approximation of a connected set $A\subseteq\fX$} if $A\subseteq U$ and, for all $n> m$, no element in $(f^n)^*\cU$ contains $A$. 
\end{definition}
\begin{rem}\label{r: approximate: iterate}
    Clearly, for $m\in\N_0$, $U\in(f^m)^*\cU$, and a connected set $A$, if $U$ is an $(f,\cU)$-approximation of $A$, then for each $n\leq m$, $f^n(U)\in(f^{m-n})^*\cU$ is an $(f,\cU)$-approximation of $f^n(A)$.
\end{rem}

Now we discuss some basic properties of branched covers on {\clcntop} manifolds. Note that quasiregular maps between Riemannian manifolds are orientation-preserving branched covers.

\begin{rem}\label{r: bc on mfd is cls fbc}
A branched cover $f\:\mfd^n\to\mfd^n$ on a {\clcntop} $n$-manifold is open and closed, and hence surjective. It follows that $N(f)$ is finite, and equal to $\abs{\deg(f)}$ when $\mfd^n$ is oriented (cf.~\cite[Theorem~5.5]{Va66}).     
\end{rem}

Given a branched cover $f\:\mfd^n\to\mfd^n$ on a {\clcntop} $n$-manifold, for each $x\in\mfd^n$, we can find a normal neighborhood $U$ of $x$ such that for each $y\in f(U)$, 
\begin{equation}\label{e: i(x,f)=Sum i(f,z)}
	N(f,U)=\ind(x,f)=\sum_{z\in U\cap f^{-1}(y)}\ind(z,f).
\end{equation}
Recall that a domain $D$ is a \defn{normal neighborhood of $x$ (with respect to $f$)} if $f(\partial D)=\partial f(D)$ and $f^{-1}(f(x))\cap D=\{x\}$.

More precisely, we use \cite[Lemmas~9.14 and~9.15]{Vu88} to find a normal neighborhood $U$ containing $x$, which is so small that $U$ and $f(U)$ are contained in charts and that $N(f,U)=i(x,f)$. Then $f|_U\:U\to f(U)$ is closed, since for a relatively closed subset $E=F\cap U$, where $F$ is a closed subset of $\mfd^n$, we have that $f\bigl(F\cap\cls{U}\bigr)=f(E)\cup f(F\cap\partial U)$ is a closed subset of $\mfd^n$, and thus $f(E)=f\bigl(F\cap\cls{U}\bigr)\cap f(U)$ is a relatively closed subset of $f(U)$. Then (\ref{e: i(x,f)=Sum i(f,z)}) follows from \cite[Lemma~9.22]{Vu88}.

\subsection{Cell decompositions and flowers}\label{subsct: cell decomp}
Here we introduce cell decompositions and cellular neighborhoods of vertices called flowers. Our definitions of cells and cell decompositions follow \cite[Chapter~5]{BM17}. Throughout the remainder of this section, $\fX$ will always be a locally-compact Hausdorff space. 

\subsubsection{Cell decompositions}

Recall that, for each $n\in\N$, a subset $c$ of $\fX$ homeomorphic to $[0,1]^n$ is called an \defn{$n$-dimensional cell}, and $\dim(c)\=n$ is called the dimension of $c$. We denote by $\cellbound c$ the set of points corresponding to $[0,1]^n\smallsetminus (0,1)^n$ under a homeomorphism between $c$ and $[0,1]^n$. We call $\cellbound c$ the \defn{cell-boundary} and $\stdCellint{c}\=c\smallsetminus\cellbound c$ the \defn{cell-interior} of $c$. The sets $\cellbound c$ and $\stdCellint{c}$ are independent of the choice of the homeomorphism. Note that the cell-boundary and cell-interior generally do not agree with the boundary $\partial c$ and interior $\stdInt{c}$ of $c$ regarded as a subset of the topological space $\fX$. A $0$-dimensional cell is a subset consisting of a single point in $\fX$. For a $0$-dimensional cell $c$, we set $\cellbound c\=\emptyset$ and $\stdCellint{c}\=c$.

We recall some basic properties of cells for further discussion; we omit the standard proofs. 
\begin{lemma}\label{l: properties of cells}
	Let $\fX$ be a locally-compact Hausdorff space. Then the following statements are true:
	\begin{enumerate}[label=(\roman*),font=\rm]
    	\smallskip
		\item Each cell $c\subseteq\fX$ is closed. Moreover, $c=\cls{\stdCellint{c}}$.
		\smallskip
		\item Assume that $\fX$ is a topological $n$-manifold. Then for each $n$-dimensional cell $X$ in $\fX$, $\stdCellint{X}$ and $\cellbound X$ agree with the interior and boundary of $X$ regarded as a subset of the topological manifold $\fX$, respectively.
	\end{enumerate}
\end{lemma}

\begin{definition}[Cell decompositions]\label{d: cell decomposition}
A collection $\cD$ of cells in a locally-compact Hausdorff space $\fX$ is a \defn{cell decomposition of $\fX$} if the following conditions are satisfied:
    \begin{enumerate}[label=(\roman*),font=\rm]
    	\smallskip
        \item The union of all cells in $\cD$ is equal to $\fX$.
        \smallskip
        \item $\brCellint{\sigma}\cap\brCellint{\tau}=\emptyset$ for all distinct $\sigma,\,\tau\in\cD$.
        \smallskip
        \item For each $\tau\in\cD$, the cell-boundary $\cellbound\tau$ is a union of cells in $\cD$.
        \smallskip
        \item Every point in $\fX$ has a neighborhood that meets only finitely many cells in $\cD$.
    \end{enumerate}
\end{definition}

For a collection $\cC$ of cells in an ambient space $\fX$, we denote by $\abs{\cC}\=\bigcup\cC$ the \defn{space} of $\cC$. If $\cC$ is a cell decomposition of $\abs{\cC}$, then we call $\cC$ a \defn{cell complex}. A cell decomposition is a cell complex.
A subset $\cC'\subseteq\cC$ is a \defn{subcomplex} of a cell complex $\cC$ if $\cC'$ is a cell complex.

Let $\cD$ be a cell decomposition of $\fX$.
For each $S \subseteq \fX$, we denote 
\begin{equation*}
\cD|_{S}\=\{c \in\cD : c \subseteq S \},
\end{equation*}
and call it the \defn{restriction of $\cD$ on $S$}. 
For each $k\in\N_0$, we denote
 \begin{equation*}
 \cD^{[k]}\=\{c\in\cD:\dim(c)=k\},
 \end{equation*} 
 and call $\Absbig{\sk{\cD}{k}}$ the \defn{$k$-skeleton of $\cD$}.

We record some elementary properties of cell decompositions, and refer the reader to \cite[Section~5.1]{BM17} for details.

\begin{lemma}[{\cite[Lemmas~5.2 and~5.3]{BM17}}]\label{l: properties of cell decompositions}
	Let $\cD$ be a cell decomposition of $\fX$. Then the following statements are true:
	\begin{enumerate}[label=(\roman*),font=\rm]
    	\smallskip
		\item For each $k\in\N_0$, the $k$-skeleton of $\cD$ is equal to $\bigcup\{\stdCellint{c}:c\in\cD,\,\dim(c) \leq  k\}$.
		\smallskip
		\item $\fX=\bigcup\{\stdCellint{c}:c\in\cD\}$.
		\smallskip
		\item For each $\tau\in\cD$, we have $\tau=\bigcup\{\stdCellint{c}:c\in\cD,\,c\subseteq\tau\}$.
		\smallskip
		\item If $\sigma$ and $\tau$ are two distinct cells in $\cD$ with $\sigma\cap\tau\neq\emptyset$, then one of the following three statements is true: $\sigma\subseteq\cellbound\tau$, $\tau\subseteq\cellbound\sigma$, or $\sigma\cap\tau=\cellbound\sigma\cap\cellbound\tau$ and the intersection consists of cells in $\cD$ of dimension strictly less than $\min\{\dim(\sigma),\dim(\tau)\}$.
		\smallskip
		\item If $\sigma,\,\tau_1,\,\dots,\,\tau_k$ are cells in $\cD$ and $\brCellint{\sigma}\cap\bigcup_{j=1}^k\tau_j\neq\emptyset$, then $\sigma\subseteq\tau_i$ for some $i\in\{1,\,\dots,\,k\}$.
	\end{enumerate}
\end{lemma}

The following lemma is a direct consequence of Definition~\ref{d: cell decomposition} and Lemma~\ref{l: properties of cell decompositions}.
\begin{lemma}\label{l: cell decomp: restrict}
     Let $\cD$ be a cell decomposition of a locally-compact Hausdorff space $\fX$ and $S\subseteq\fX$. Then the following statements are true:
     \begin{enumerate}[label=(\roman*),font=\rm]
        \smallskip
         \item  $\cD|_S$ is a subcomplex of $\cD$.
         \smallskip
         \item If $S$ is a union of cells in $\cD$, then $\cD|_S$ is a cell decomposition of $S$. In particular, for each $c\in\cD$, $\cD|_c$ and $\cD|_{\cellbound c}$ are cell decompositions of $c$ and $\cellbound c$, respectively.
         \smallskip
         \item If $S=\abs{\cC}$ for a subcomplex $\cC$ of $\cD$, then $\cD|_S=\cC$. 
     \end{enumerate}
 \end{lemma}

As another direct consequence of Definition~\ref{d: cell decomposition} and Lemma~\ref{l: properties of cell decompositions}, given a cell decomposition $\cD$ of $\fX$, for each $x\in\fX$, there is a unique cell $c\in\cD$ such that $x\in\stdCellint{c}$, which we call the \emph{carrier}.

\begin{definition}[Carrier]\label{d: carrier}
	The \defn{carrier} of $x\in\fX$ in a cell decomposition $\cD$ is the unique cell $\supp_{\cD}(x)\in\cD$ for which $x\in\brCellint{\supp_\cD(x)}$.
\end{definition}
It is easy to observe that for $d\in\N_0$, $x\in\Absbig{\sk{\cD}{d}}\smallsetminus\Absbig{\sk{\cD}{d-1}}$ if and only if $\dim(\supp_\cD(x))=d$.

In this article, we mostly focus on the special case where the topological space $\fX$ is a \defn{closed topological manifold}, i.e.,~a compact topological manifold without boundary.
For a topological $n$-manifold $\mfd^n$, by the invariance of domain theorem, the dimension of each cell in a cell decomposition of $\mfd^n$ is no more than $n$. This observation, together with Baire's theorem, yields the following properties.

\begin{lemma}\label{l: cell decomp: compact, and Mfd}
    Let $\cD$ be a cell decomposition of a locally-compact Hausdorff space $\fX$. Then the following statements are true:
    \begin{enumerate}[label=(\roman*),font=\rm]
    	\smallskip
		\item If $\fX$ is compact, then $\cD$ is a finite set.
		\smallskip
		\item If $\fX$ is a closed topological $n$-manifold, then $\fX=\Absbig{\cD^{[n]}}$ and each cell in $\cD$ is contained in an $n$-dimensional cell in $\cD$.
        \smallskip
        \item For each $c\in\cD$ and each integer $0\leqslant k\leqslant \dim(c)$, there exists $\sigma\in\cD|_c$ for which $\dim(\sigma)=k$.
	\end{enumerate}  
\end{lemma}
\begin{proof}
    (i) For each $x\in\fX$, by Definition~\ref{d: cell decomposition}~(iv), we may choose a neighborhood $U_x$ in $\mfd^n$ that meets only finitely many cells in $\cD$.
	Since $\fX$ is compact, suppose $\fX=U_{x_1}\cup\dots\cup U_{x_m}$ for some $x_1,\,\dots,\, x_m\in\mfd^n$. Then for each $c\in\cD$, there exists $i\in\{1,\,\dots,\,m\}$ such that $U_{x_i}\cap c\neq\emptyset$. It follows that 
	$\cD=\bigcup_{i=1}^m\{c\in\cD:c\cap U_i\neq\emptyset\}$
	is a finite set.

    \smallskip

    (ii) By (i) $\cD$ is a finite set. Since for each $c\in\sk{\cD}{n-1}$ the interior $\stdInt{c}$ (regarded as a subset of the topological $n$-manifold $\fX$) is empty, we have
	$\stdInt{\Absbig{\sk{\cD}{n-1}}}=\emptyset$. Combining this with $|\cD|=\fX$, we have that
	$\Absbig{\cD^{[n]}}$ is dense in $\fX$. Since each $c\in\cD$ is a closed subset of $\fX$, $\Absbig{\cD^{[n]}}$ is closed. Thus, $\fX=\Absbig{\cD^{[n]}}$. 
    As a consequence, for each $c\in\cD$, there exists $X\in\cD^{[n]}$ that meets $\stdCellint{c}$, and it follows from Lemma~\ref{l: properties of cell decompositions}~(v) that $c\subseteq X$. 

    \smallskip

    (iii) We argue by induction. 
	For each $c\in\cD$ with $\dim(c)=0$, the claim is clear. It suffices to show that for each $c\in\cD$ with $\dim(c)>0$, there exists $\sigma\in\cD|_c$ with $\dim(\sigma)=\dim(c)-1$.
	
	Let $c\in\cD$ with $\dim(c)=d>0$ be arbitrary.
	By Lemma~\ref{l: cell decomp: restrict}, $\cD|_{\cellbound c}$ is a cell decomposition of $\cD$.
	Since $\cellbound c$ is a topological $(d-1)$-sphere, by (ii),  $\cellbound c$ is the union of $(d-1)$-dimensional cells in $\cD|_{\cellbound c}$, which is a subset of $\cD|_c$. Thus, there exists $\sigma\in\cD|_c$ with $\dim(\sigma)=d-1$.
\end{proof}

Finally, we recall the notion of \emph{refinements} of cell decompositions.
\begin{definition}[Refinements]\label{d: refinement}
A cell decomposition $\cD_1$ of a locally-compact Hausdorff space $\fX$ is a \defn{refinement} of a cell decomposition $\cD_0$ of $\fX$ if the following conditions are satisfied:
\begin{enumerate}[label=(\roman*),font=\rm]
	\smallskip
    \item For each $\sigma\in\cD_1$, there exists $\tau\in\cD_0$ satisfying $\sigma\subseteq\tau$.
    \smallskip
    \item Each cell $\tau\in\cD_0$ is the union of cells $\sigma\in\cD_1$ satisfying $\sigma\subseteq\tau$.
\end{enumerate}
In this case, we also say that $\cD_1$ \defn{refines $\cD_0$}.
\end{definition}

\begin{lemma}[{\cite[Lemma~5.7]{BM17}}]\label{l: cell decomp: refinement inte() contained in inte()}
	Let $\cD$ be a cell decomposition of a locally-compact Hausdorff space $\fX$ and let $\cD'$ be a refinement of $\cD$. Then for each $\sigma\in\cD'$, there exists a minimal cell $\tau\in\cD$ with $\sigma\subseteq\tau$, i.e., if $\sigma\subseteq\ttau$ for some $\ttau\in\cD$, then $\tau\subseteq\ttau$. Moreover, $\tau$ is the unique cell in $\cD$ with $\stdCellint{\sigma}\subseteq\stdCellint{\tau}$.
\end{lemma}

\begin{cor}\label{c: refinement: inte(c1) meet inte(c0)}
    Let $\cD$ be a cell decomposition of a locally-compact Hausdorff space $\fX$ and let $\cD'$ be a refinement of $\cD$. If $c'\in\cD'$ and $c\in\cD$ satisfy $\brCellint{c'}\cap\brCellint{c}\neq\emptyset$, then $c'\subseteq c$ and $\stdCellint{c'}\subseteq\stdCellint{c}$.
\end{cor}
\begin{proof}
    By Lemma~\ref{l: cell decomp: refinement inte() contained in inte()}, there exists $\sigma\in\cD$ for which $\stdCellint{c'}\subseteq\stdCellint{\sigma}$. Since $\brCellint{c'}\cap\brCellint{c}\neq\emptyset$, $\brCellint{c}\cap\brCellint{\sigma}\neq\emptyset$, and thus $c=\sigma$.
\end{proof}

\begin{cor}\label{c: refinement: restriction}
    Let $\fX$ be a locally-compact Hausdorff space, $\cD$ be a cell decomposition of $\fX$, $\cD'$ be a refinement of $\cD$, and $S\subseteq\fX$ be a union of cells in $\cD$. Then $\cD'|_S$ is a cell decomposition of $S$ that refines $\cD|_S$.
\end{cor}
\begin{proof}
    By Definition~\ref{d: refinement}, $S$ is also a union of cells in $\cD'$.
    Thus, by Lemma~\ref{l: cell decomp: restrict}~(ii), $\cD|_S$ and $\cD'|_S$ are indeed cell decompositions of $S$. For each $c'\in\cD'|_S$, by Lemma~\ref{l: properties of cell decompositions}~(ii), there exists $c\in\cD|_S$ with $\brCellint{c'}\cap\brCellint{c}\neq\emptyset$. Since $\cD'$ refines $\cD$, by Corollary~\ref{c: refinement: inte(c1) meet inte(c0)}, $c'\subseteq c$. On the other hand, by Definition~\ref{d: refinement}, each cell in $\cD|_S$ is a union of cells in $\cD'|_S$. Thus, $\cD'|_S$ is a refinement of $\cD|_S$.
\end{proof}

We give a lemma that allows us to obtain a refinement of a cell decomposition by subdividing each cell.
\begin{lemma}\label{l: construct refinement}
	Let $\cD$ be a cell decomposition of a locally-compact Hausdorff space $\fX$. For each $c\in\cD$, let $\cD'(c)$ be a cell decomposition of $c$ that refines $\cD|_c$. Then $\cD'\=\bigcup_{c\in\cD}\cD'(c)$ is a cell decomposition that refines $\cD$ if and only if for all $c,\,\sigma\in\cD$ with $c\subseteq\sigma$, we have $\cD'(\sigma)|_c=\cD'(c)$.
\end{lemma}
\begin{proof}
    First, we show the ``only if" part. Assume that $\cD'$ is a cell decomposition. Let $c,\,\sigma\in\cD$ satisfy $c\subseteq\sigma$. For each $c'\in\cD'(\sigma)|_c$, since $\cD'(c)$ is a cell decomposition of $c$, by Lemma~\ref{l: properties of cell decompositions}~(ii), there exists $c''\in\cD'(c)$ such that $\brCellint{c'}\cap\brCellint{c''}\neq\emptyset$. Then since $\cD'$ is a cell decomposition, $c'=c''\in\cD'(c)$. Thus $\cD'(\sigma)|_c\subseteq\cD'(c)$. The same arguments show that $\cD'(c)\subseteq\cD'(\sigma)|_c$.
    
    In what follows, we show the ``if" part. Assume that for all $c,\,\sigma\in\cD$ with $c\subseteq\sigma$, we have $\cD'(\sigma)|_c=\cD'(c)$. We verify conditions~(i)--(iii) in Definition~\ref{d: cell decomposition}.
    
    Conditions~(i) and~(iii) are direct consequences of the construction of $\cD'$. 
    
    Now we verify (ii).
	Let $\sigma',\,\tau'\in\cD'$ be such that $\brCellint{\sigma'}\cap\brCellint{\tau'}\neq\emptyset$. Suppose $\sigma,\,\tau\in\cD$ satisfy $\sigma'\in\cD'(\sigma)$ and $\tau'\in\cD'(\tau)$. 
	Since $\cD'(\sigma)$ refines $\cD|_\sigma$, by Lemma~\ref{l: cell decomp: refinement inte() contained in inte()}, there is $\sigma_1\in\cD|_\sigma$ such that $\stdCellint{\sigma'}\subseteq\stdCellint{\sigma_1}$ and $\sigma'\subseteq\sigma_1$. Likewise, there is $\tau_1\in\cD|_\tau$ such that $\stdCellint{\tau'}\subseteq\stdCellint{\tau_1}$ and $\tau'\subseteq\tau_1$. If follows from $\brCellint{\sigma'}\cap\brCellint{\tau'}\neq\emptyset$ that  $\stdCellint{\sigma_1}$ meets $\stdCellint{\tau_1}$, and thus $\sigma_1=\tau_1$. Denote $c\=\sigma_1=\tau_1\in\cD$. Then 
    since $\sigma'\subseteq\sigma_1$ and $\tau'\subseteq\tau_1$,
    $\sigma'\in\cD'(\sigma)|_c$ and $\tau'\in\cD'(\tau)|_c$. By the assumption $\cD'(\sigma)|_c=\cD'(c)=\cD'(\tau)|_c$, we have $\sigma',\,\tau'\in\cD'(c)$. Since $\cD'(c)$ is a cell decomposition and $\sigma'\cap\tau'\neq\emptyset$, we have $\sigma'=\tau'$. 

    Finally, we verify~(iv).
    Fix arbitrary $x\in\fX$ and consider a neighborhood $U$ of $x$ that meets only finitely many cells in $\cD$. 
    For each $c\in\cD$, since $c$ is compact, by Lemma~\ref{l: cell decomp: compact, and Mfd}~(i), $\cD'(c)$ is finite. It follows that $U$ meets only finitely many cells in $\cD'$. 
    
    Hence, $\cD'$ is a cell decomposition. It is clear from the construction of $\cD'$ that $\cD'$ refines $\cD$.
\end{proof}

\subsubsection{Flowers}
Let $\cD$ be a cell decomposition of a closed topological $n$-manifold $\mfd^n$, $n\geq 2$. We have by Lemma~\ref{l: cell decomp: compact, and Mfd}~(ii) that $\cD^{[k]}$ is non-empty for each $0 \leq  k \leq  n$. For a $0$-dimensional cell $\{v\}\in\cD$, we call the unique point $v\in\mfd^n$ a \defn{vertex}. We call the $n$-dimensional cells of $\cD$ \defn{chambers} and $(n-1)$-dimensional cells \defn{facets}.

The chambers of $\cD$ form a cover of $\mfd^n$, but are not open. To construct from $\cD$ an open cover of $\mfd^n$, we introduce the notion of \emph{flowers}. Note that this notion is very similar to the \emph{stars} in the theory of simplicial complex.

Let $\cD$ be a cell decomposition of a closed topological $n$-manifold $\mfd^n$. For $c\in\cD$, we denote
\begin{equation}\label{e: cF(c)}
	\flower_\cD(c)\=\bigcup\{\stdCellint{\sigma}:\sigma\in\cD,\,c\subseteq\sigma\}=\bigcup\{\stdCellint{\sigma}:\sigma\in\cD,\,\sigma\cap\stdCellint{c}\neq\emptyset\}.
\end{equation}
Similarly, for a point $x\in\mfd^n$, we denote 
\begin{equation}\label{e: Fl(x)}
	\begin{aligned}
			\flower_\cD(x)\=\flower_\cD(\supp_\cD(x))=\bigcup\{\stdCellint{\sigma}:x\in\sigma\}.
	\end{aligned}
\end{equation}
It is easy to check by Lemma~\ref{l: properties of cell decompositions}~(v) that the equalities in (\ref{e: cF(c)}) and (\ref{e: Fl(x)}) hold. It is also clear that $\stdCellint{c}\subseteq\flower_\cD(c)$ for each $c\in\cD$, and $x\in\flower_\cD(x)$ for each $x\in\mfd^n$.
 
Clearly $\{\flower_\cD(x):x\in\mfd^n\}$ is a cover of $\mfd^n$. We choose a finite subcover, namely,
\begin{equation*}
\coverF(\cD)\=\{\flower_\cD(p):p\text{ is a vertex of }\cD\}.
\end{equation*}
For each vertex $p$ in $\cD$, we call $\flower_\cD(p)$ the \defn{flower} of the vertex $p$ in $\cD$.
Indeed, $\coverF(\cD)$ is a cover of $\mfd^n$, since for each $x\in\mfd^n$, Lemmas~\ref{l: properties of cell decompositions}~(ii) and~\ref{l: cell decomp: compact, and Mfd}~(iii) yield that there exist $c\in\cD$ with $x\in\stdCellint{c}$ and a vertex $p$ with $p\in c$.

With $\cD$ understood, we often omit $\cD$ in the above notations, that is, we denote $\flower(\cdot)\=\flower_\cD(\cdot)$.
The following properties are elementary.
\begin{prop}\label{p: c-flower: structure}
	Let $\cD$ be a cell decomposition of a closed topological $n$-manifold $\mfd^n$ and $c\in\cD$. Then the following statements are true:
	\begin{enumerate}[font=\rm,label=(\roman*)]
    	\smallskip
		\item $\mfd^n\smallsetminus\flower(c)=\bigcup\{\sigma\in\cD:c\nsubseteq \sigma\}$, and $\flower(c)$ is a path-connected open set.
		\smallskip
		\item $\cls{\flower(c)}=\bigcup\bigl\{X\in\cD^{[n]}:c\subseteq X\bigr\}$.
		\smallskip
		\item For each $x\in\mfd^n$, $x\in\partial\flower(c)$ if and only if there exist $X\in\cD^{[n]}$ and $\sigma\in\cD|_X$ such that $x\in\sigma$, $c\nsubseteq\sigma$, and $c\subseteq X$.
	\end{enumerate}
\end{prop}
\begin{proof}
	(i) First, if $x\in\mfd^n\smallsetminus\flower(c)$, then $c\nsubseteq\supp_\cD(x)$ by (\ref{e: cF(c)}) and Definition~\ref{d: carrier}. On the other hand, if $\sigma\in\cD$ satisfies $c\nsubseteq\sigma$, then for each $x\in\sigma$, we have $c\nsubseteq\supp_\cD(x)\subseteq\sigma$ (cf.~Lemma~\ref{l: properties of cell decompositions}~(v)), and thus $x\notin\flower(c)$ by (\ref{e: cF(c)}) and Definition~\ref{d: carrier}.
    Thus $\mfd^n\smallsetminus\flower(c)=\bigcup\{\sigma\in\cD:c\nsubseteq\sigma\}$. So Lemma~\ref{l: cell decomp: compact, and Mfd}~(i) implies that $\mfd^n\smallsetminus\flower(c)$ is closed and $\flower(c)$ is open. For each $c\subseteq\sigma$, it is easy to see that $\brCellint{\sigma}\cup\brCellint{c}$ is path-connected. Hence, $\flower(c)$ is path-connected.

	\smallskip
	
	(ii) For each $X\in\cD^{[n]}$ with $c\subseteq X$, we have $\stdCellint{X}\subseteq \flower(c)$, and thus $X\subseteq\cls{\flower(c)}$. It follows that $\bigcup\{X\in\cD^{[n]}:c\subseteq X\}\subseteq\cls{\flower(c)}$. On the other hand, by Lemma~\ref{l: cell decomp: compact, and Mfd}~(ii) and (\ref{e: cF(c)}), we have $\flower(c)\subseteq\bigcup\bigl\{X\in\cD^{[n]}:c\subseteq X\bigr\}$. Since the right-hand side is a closed set contained in $\cls{\flower(c)}$, we have $\cls{\flower(c)}=\bigcup\bigl\{X\in\cD^{[n]}:c\subseteq X\bigr\}$.
	
	\smallskip
	
	(iii) Fix $x\in\mfd^n$. First suppose $x\in\partial\flower(c)=\cls{\flower(c)}\smallsetminus\flower(c)$. Then $c\nsubseteq\supp_\cD(x)$ by (\ref{e: cF(c)}). By (ii), we can choose $X\in\cD^{[n]}$ with $x\in X$ and $c\subseteq X$. By Lemma~\ref{l: properties of cell decompositions}~(v), we have $\supp_\cD(x)\subseteq X$.
	On the other hand, suppose $X\in\cD^{[n]}$ and $\sigma\in\cD|_X$ satisfy $x\in\sigma$, $c\nsubseteq\sigma$, and $c\subseteq X$. By (i) and (ii) we have $x\in\cls{\flower(c)}\smallsetminus\flower(c)=\partial\flower(c)$.
\end{proof}

As a direct consequence of Proposition~\ref{p: c-flower: structure}, the structure of $\flower(x)$, $x\in\mfd^n$, are clear.
\begin{cor}\label{c: x-flower structure}
Let $\cD$ be a cell decomposition of a closed topological $n$-manifold $\mfd^n$ and $x\in\mfd^n$. Then the following statements are true:
	\begin{enumerate}[font=\rm,label=(\roman*)]
    	\smallskip
		\item $\mfd^n\smallsetminus\flower(x)=\bigcup\{\sigma\in\cD:x\notin\sigma\}$, and $\flower(x)$ is a path-connected open neighborhood of $x$.
		\smallskip
		\item $\cls{\flower(x)}=\bigcup\bigl\{X\in\cD^{[n]}:x\in X\bigr\}$.
		\smallskip
		\item For each $z\in\mfd^n$, $z\in\partial\flower(x)$ if and only if there exist $X\in\cD^{[n]}$ and $\sigma\in\cD|_X$ such that $z\in\sigma$, $x\notin\sigma$, and $x\in X$.
	\end{enumerate}
\end{cor}
\begin{proof}
    By Lemma~\ref{l: properties of cell decompositions}~(v), for each $\sigma\in\cD$, $x\in\sigma$ if and only if $\supp_{\cD}(x)\subseteq\sigma$. Thus, each statement in (i)--(iii) follows from (\ref{e: Fl(x)}) and the corresponding statement in Proposition~\ref{p: c-flower: structure}~(i)--(iii).
\end{proof}

Moreover, we will give a necessary and sufficient condition for a set to be contained in some flower. First, we introduce the notion of \emph{joining opposite sides}, which generalizes a similar notion of Bonk and Meyer for Jordan curves containing the postcritical set on $\cS^2$; see \cite[Section~5.7]{BM17}.

\begin{definition}[Joining opposite sides]   \label{d: join_opposite_sides}
	Let $\cD$ be a cell decomposition of a closed topological $n$-manifold $\mfd^n$. A subset $A\subseteq\mfd^n$ \defn{{\jsOpSd}} of $\cD$ if $\bigcap\{c\in\cD:A\cap c\neq\emptyset\}=\emptyset$.
\end{definition}

\begin{lemma}\label{l: decomp of S^n: join opp <=> not in flower}
	Let $\cD$ be a cell decomposition of a closed topological $n$-manifold $\mfd^n$. Then for each subset $A\subseteq\mfd^n$, the following are equivalent:
	\begin{enumerate}[label=(\roman*),font=\rm]
    	\smallskip
		\item $A$ {\jsOpSd} of $\cD$.
		\smallskip
		\item $A$ is not contained in $\flower(c)$ for each $c\in\cD$.
		\smallskip
		\item $A$ is not contained in any flower in $\coverF(\cD)$.
	\end{enumerate}
\end{lemma}
\begin{proof}
	(ii)$\Rightarrow$(iii) is clear. It suffices to prove 	
	(i)$\Rightarrow$(ii) and 	
	(iii)$\Rightarrow$(i).
	
	First, we verify (i)$\Rightarrow$(ii). Suppose $A\subseteq\flower(c)$ for some $c\in\cD$. Let $\sigma\in\cD$ be an arbitrary cell with $\sigma\cap A\neq\emptyset$. Then $\sigma\cap\flower(c)\neq\emptyset$,
	and, by Lemma~\ref{p: c-flower: structure}~(i), we have $c\subseteq\sigma$. It follows that $c\subseteq\bigcap\{\sigma\in\cD:A\cap\sigma\neq\emptyset\}\neq\emptyset$.
	
	For (iii)$\Rightarrow$(i), suppose $\bigcap\{\sigma\in\cD:A\cap\sigma\neq\emptyset\}\neq\emptyset$. By Lemma~\ref{l: properties of cell decompositions}~(iv), $\bigcap\{\sigma\in\cD:A\cap\sigma\neq\emptyset\}$ is a union of cells in $\cD$. Then by Lemma~\ref{l: cell decomp: compact, and Mfd}~(iii), there exists a vertex $p$ for which $p\in\bigcap\{\sigma\in\cD:A\cap\sigma\neq\emptyset\}$. For each $x\in A$, since $\supp_{\cD}(x)\cap A\neq\emptyset$, we have $p\in\supp_{\cD}(x)$ and $x\in\flower(p)$. Then $A\subseteq\flower(p)$.
\end{proof}

By Lemma~\ref{l: decomp of S^n: join opp <=> not in flower}, a flower cannot join opposite sides of $\cD$, and in particular, cannot meet two disjoint cells (cf.~Definition~\ref{d: join_opposite_sides}).

\begin{cor}\label{c: decomp of S^n: disjoint chambers=>not in a flower}
	Let $\cD$ be a cell decomposition of a closed topological $n$-manifold $\mfd^n$, and let $\sigma,\,\tau\in\cD$ be disjoint. Then for each $c\in\cD$, either $\flower(c)\cap\sigma=\emptyset$ or $\flower(c)\cap\tau=\emptyset$.
\end{cor}

\subsection{Cellular maps}
We recall cellular maps and related notions, following \cite[Chapter~5]{BM17}.
In this subsection, $\fX$ and $\fX'$ stand for locally-compact Hausdorff spaces.

\begin{definition}[Cellular maps]\label{d: cellular map}
A continuous map $f\:\fX'\to\fX$ is a \defn{cellular map} if there exist cell decompositions $\cD'$ and $\cD$ of $\fX'$ and $\fX$, respectively, such that for each $c\in\cD'$, $f(c)$ is a cell in $\cD$ and the restriction $f|_{c}\:c\to f(c)$ is a homeomorphism. In this case, we call $f$ a \defn{$(\cD',\cD)$-cellular map}, $(\cD',\cD)$ a \defn{{\celPair}} of $f$, and $\cD'$ a \defn{pullback} of $\cD$ by $f$.
\end{definition}

To study dynamics from the perspective of cellular structures, we consider a sequence $\{\cD_m\}_{m\in\N_0}$ of cell decompositions such that all pairs $(\cD_{m+1},\cD_m)$ of consecutive cell decompositions are cellular. 

\begin{definition}[{\ItCel} maps]\label{d: iter cellular map}
	A sequence $\{\cD_m\}_{m\in\N_0}$ of cell decompositions is a \defn{{\celSeq}} of a continuous map $f\:\fX\to\fX$ if, for each $m\in\N_0$, $f$ is $(\cD_{m+1},\cD_m)$-cellular. We call a continuous map $f\:\fX\to\fX$ an \defn{{\itCel} map} if there exists a {\celSeq} of $f$.
\end{definition}

\begin{rem}\label{r: basic cellular property}
It follows from the definition that the composition of two cellular maps is cellular:
if $g\:\fX''\to\fX'$ is $(\cD'',\cD')$-cellular and $f\:\fX'\to\fX$ is $(\cD',\cD)$-cellular, then $f\circ g$ is $(\cD'',\cD)$-cellular. In particular, given a {\celSeq} $\{\cD_m\}_{m\in\N_0}$ of an {\itCel} map $f\:\fX\to\fX$, for all $m,\,k\in\N_0$, $f^k$ is $(\cD_{m+k},\cD_m)$-cellular.
\end{rem}

Although neither each cell decomposition admits a pullback, nor {\celPair}s in Definition~\ref{d: cellular map} are unique, we show that the pullback $\cD'$ of cell decomposition $\cD$, if exists, is uniquely determined by $\cD$.

\begin{lemma}\label{l: cellular: pull back ess decomp}
	Let $f\:\fX\to\fX$ be a cellular map and $(\cD_1,\cD_0)$ be a {\celPair} of $f$. 
	Then the following statements are true:
	\begin{enumerate}[label=(\roman*),font=\rm]
    	\smallskip
		\item For each $\tau\in\cD_0$, we have $f^{-1}(\stdCellint{\tau})=\bigcup\{\stdCellint{\sigma}:\sigma\in\cD_{1},\,f(\sigma)=\tau\}$.
		\smallskip
		\item For each $\sigma\in\cD_{1}$, $\stdCellint{\sigma}$ is a connected component of $f^{-1}(\stdCellint{f(\sigma)})$.
		\smallskip
		\item 
		$\cD_{1}=\bigcup_{\tau\in \cD_0}\bigl\{\cls{c}:c\text{ is a connected component of }f^{-1}(\stdCellint{\tau})\bigr\}$.
		\smallskip
		\item If $(\cD_1',\cD_0)$ is a cellular pair of $f$, then $\cD_1'=\cD_1$.
	\end{enumerate}
\end{lemma}
\begin{proof}
	(i) Let $\tau\in\cD_0$ be arbitrary. For each $x\in f^{-1}(\stdCellint{\tau})$, $f(x)\in \brCellint{\tau}\cap\brCellint{f (\supp_{\cD_1}(x))}$, and thus $f(\supp_{\cD_1}(x))=\tau$ (cf.~Definition~\ref{d: cell decomposition}~(ii)). Then $f^{-1}(\stdCellint{\tau})\subseteq\bigcup\{\stdCellint{\sigma}:\sigma\in\cD_{1},\,f(\sigma)=\tau\}$.
	On the other hand, it is clear that $\bigcup\{\stdCellint{\sigma}:\sigma\in\cD_{1},\,f(\sigma)=\tau\}\subseteq f^{-1}(\stdCellint{\tau})$.
	
	\smallskip
	
	(ii) Let $\tau\in\cD_0$ be arbitrary.  We show that if $\sigma,\,\sigma'\in\cD_{1}$ satisfy $\sigma\neq\sigma'$, $\tau=f(\sigma)=f(\sigma')$, then $\sigma'\cap\stdCellint{\sigma}=\emptyset$. Suppose $\sigma'\cap\stdCellint{\sigma}\neq\emptyset$. Then by Lemma~\ref{l: properties of cell decompositions}~(v), we have that $\sigma\subseteq\sigma'$. If $\sigma\cap\stdCellint{\sigma'}\neq\emptyset$, then $\sigma'\subseteq\sigma$, which contradicts $\sigma\neq\sigma'$. Hence, $\sigma\subseteq\cellbound\sigma'$, which contradicts $\dim(\sigma)=\dim(\tau)=\dim(\sigma')$ (cf.~Definition~\ref{d: cellular map}). Therefore, $\sigma'\cap\stdCellint{\sigma}=\emptyset$.
	
	Fix arbitrary $\sigma\in\cD_{1}$ and let $\tau\=f(\sigma)$. Since $\stdCellint{\sigma'}\cap\sigma=\emptyset$ for each $\sigma'\in\cD_{1}$ that satisfies $\tau=f(\sigma')$ and $\sigma'\neq\sigma$, it follows from (i) that $\stdCellint{\sigma}=\sigma\cap f^{-1}(\stdCellint{\tau})$. By Lemma~\ref{l: properties of cells}~(i), $\sigma$ is a closed subset of $\fX$, then $\stdCellint{\sigma}=\sigma\cap f^{-1}(\stdCellint{\tau})$ is a closed subset of $f^{-1}(\stdCellint{\tau})$.
	
	Then $\{\stdCellint{\sigma}:\sigma\in\cD_{1},\,f(\sigma)=\tau\}$ is a finite collection of closed connected subsets of $f^{-1}(\stdCellint{\tau})$ that are pairwise disjoint. Hence, for each $c\in\cD_{1}$ with $f(c)=\tau$, $\stdCellint{c}$ is a connected component of $f^{-1}(\stdCellint{\tau})$.
	
	\smallskip
	
	(iii) For each $\sigma\in\cD_{1}$, we have by (ii) that $\stdCellint{\sigma}$ is a connected component of $f^{-1}(\brCellint{f(\sigma)})$, and by Lemma~\ref{l: properties of cells}~(i) that $\sigma=\cls{\stdCellint{\sigma}}$.
	
	On the other hand, for each $\tau\in\cD_0$ and each connected component $c$ of $f^{-1}(\stdCellint{\tau})$, by (i) and~(ii) we have that $c=\stdCellint{\sigma}$ for some $\sigma\in\cD_{1}$. Thus, $\cls{c}=\sigma\in\cD_{1}$.
	
	\smallskip
	
	(iv) Statement~(iv) follows from (iii).
\end{proof}

    Homeomorphisms are cellular maps: under a homeomorphism, each cell decomposition admits a unique pullback. For a homeomorphism  $\phi\:\fX'\to\fX$ and a cell decomposition $\cD$ of $\fX$, denote
    \begin{equation}\label{e: homeo pullback}
        \phi^*(\cD)\=\bigl\{\phi^{-1}(c):c\in\cD\bigr\}.
    \end{equation}
	\begin{cor}\label{c: cellular: cellular homeo}
    Let $\phi\:\fX'\to\fX$ be a homeomorphism and $\cD$ a cell decomposition of $\fX$, then $\phi^*(\cD)\=\bigl\{\phi^{-1}(c):c\in\cD\bigr\}$ is a cell decomposition, and the unique pullback of $\cD$.
	In particular, given a $(\cD',\cD)$-cellular map $f\colon\fX'\to\fX$, for each $c'\in \cD'$, we have $\cD'|_{c'}=(f|_{c'})^*\bigl(\cD|_{f(c')}\bigr)$. 
	\end{cor}

\section{Cellular sequence and expansion of Thurston-type maps}
\label{sct: cellular maps}
This section contains a discussion on notions related to Thurston-type maps.
First, we discuss some elementary properties of cellular maps. Then we recall the notion of expansion and discuss its relation to cellular structures. Finally, in Subsection~\ref{subsct: CPCF and CTM}, we establish a key result of this section---the existence of {\celSeq}s, as discussed in the introduction. 

We begin by fixing some notations and terminology.
With an {\itCel} map $f\colon \mfd^n \to \mfd^n$, $n\geq 2$, on a closed topological $n$-manifold and a {\celSeq} $\{\cD_m\}_{m\in \N_0}$ of $f$ understood, we call, for each $i\in\N_0$, a vertex, a facet, and a chamber in $\cD_i$ a \defn{level-$i$~vertex} (or \defn{$i$-vertex}), a \defn{level-$i$~facet} (or \defn{$i$-facet}), and a \defn{level-$i$~chamber} (or \defn{$i$-chamber}), respectively. We also denote $\flower_i(\cdot)\=\flower_{\cD_i}(\cdot)$ for each $i\in\N_0$, and call an element in $\coverF(\cD_i)$ a \defn{level-$i$~flower} (or \defn{$i$-flower}).
For a cellular pair $(\cD_1,\cD_0)$, we also use the above notations and terminology with $i\in\{0,\,1\}$.

\subsection{Invariance of flowers}\label{subsct: inv flower}
We begin with a statement that, for all cellular maps, flowers are forward invariant in the following sense.

\begin{prop}\label{p: c-flower: fwd invariance}
	Let $f\:\mfd^n\to\mfd^n$ be a $(\cD_1,\cD_0)$-cellular map on a {\clcntop} $n$-manifold. Then
    \begin{enumerate}[label=(\roman*),font=\rm]
        \smallskip
        \item for each $c\in\cD_1$, we have $f(\flower_1(c))\subseteq\flower_0(f(c))$ and $f(\partial\flower_1(c))\subseteq\partial\flower_0(f(c))$;
        \smallskip
        \item consequently, for each $x\in\mfd^n$, $f(\flower_1(x))\subseteq\flower_0(f(x))$ and $f(\partial\flower_1(x)))\subseteq\partial\flower_0(f(x))$.
    \end{enumerate}
\end{prop}
\begin{proof}
	(i) Fix $c\in\cD$. First consider $x\in \flower_1(c)$. By (\ref{e: cF(c)}), there is $\sigma\in\cD_1$ with $c\subseteq\sigma$ and $x\in\stdCellint{\sigma}$. Then $f(c)\subseteq f(\sigma)$ and thus $ f(\stdCellint{\sigma})=\brCellint{f(\sigma)}\subseteq\flower_0(f(c))$. Hence, $f(x)\in f(\stdCellint{\sigma})\subseteq\flower_0(f(c))$.
	Now consider $x\in\partial\flower_1(c)$. By Proposition~\ref{p: c-flower: structure}~(iii), we can choose $X\in\cD_1^{[n]}$ and $\sigma\in\cD_1|_X$ satisfying $x\in\sigma$, $c\nsubseteq\sigma$, and $c\subseteq X$. Then since $f|_X\:X\to f(X)$ is a homeomorphism, we have $f(x)\in f(\sigma)$, $f(c)\subseteq f(X)$, and $f(c)\nsubseteq f(\sigma)$. Then (again, by Proposition~\ref{p: c-flower: structure}~(iii)) we have $f(x)\in\partial\flower_0(f(c))$.

    \smallskip

    (ii) Fix $x\in\mfd^n$. By the definition of cellular maps, $f(\supp_{\cD_1}(x))=\supp_{\cD_0}(f(x))$. Thus, by (i) and (\ref{e: Fl(x)}), $f(\flower_1(x))=f(\flower_1(\supp_{\cD_1}(x)))\subseteq\flower_0(f(\supp_{\cD_1}(x)))=\flower_0(\supp_{\cD_0}(f(x)))=\flower_0(f(x))$. Likewise, $f(\partial\flower_1(x)))\subseteq\partial\flower_0(f(x))$.
\end{proof}

In Proposition~\ref{p: c-flower: fwd invariance}, when $f$ is not open, $f(\flower_1(c))$ may not equal $\flower_0(f(c))$ (in fact, one can construct an example where all $1$-chambers are mapped to an $0$-chamber). However, for cellular branched covers (i.e., branched covers that are cellular maps), we have the following total invariance properties of flowers analogous to the corresponding properties in dimension two in Bonk--Meyer~\cite[Lemma~5.29]{BM17}. The main difference between the statements is that, in dimensions $n>2$, flowers do not have ``cyclic petals" (cf.~\cite[Lemma~5.29]{BM17}).

\begin{prop}\label{p: p-flowers: invariance}
	Let $f\colon \mfd^n \to \mfd^n$ be a cellular branched cover on a {\clcntop} $n$-manifold and $(\cD_1,\cD_0)$ be a {\celPair} of $f$.
	Then the following statements are true:
	\begin{enumerate}[label=(\roman*),font=\rm]
    	\smallskip
		\item If $p$ is a $1$-vertex, then $f(\flower_{1}(p))=\flower_0(f(p))$.
		\smallskip
		\item If $q$ is a $0$-vertex, then the connected components of $f^{-1}(\flower_0(q))$ are exactly all the flowers $\flower_{1}(p),\,p\in f^{-1}(q)$. Consequently, $\coverF(\cD_1)=f^*\coverF(\cD_0)$.
	\end{enumerate}
\end{prop}

Recall that $f^*\flower(\cD_0)$ is defined by (\ref{e: pullback of cover}). 
Before the proof of Proposition~\ref{p: p-flowers: invariance}, we state the following corollary, which is a straightforward consequence of Proposition~\ref{p: p-flowers: invariance}, Lemma~\ref{l: decomp of S^n: join opp <=> not in flower}, and Corollary~\ref{c: decomp of S^n: disjoint chambers=>not in a flower}; cf.~\cite[Lemma~5.29]{BM17}.

\begin{cor}\label{c: CBC: connected set in flower}
	Let $f\:\mfd^n\to\mfd^n$ be a cellular branched cover on a {\clcntop} $n$-manifold and $(\cD_1,\cD_0)$ be a {\celPair} of $f$.
	Let $A\subseteq\mfd^n$ be a connected set. Then
    \begin{enumerate}[label=(\roman*),font=\rm]
        \item $A$ is contained in a $1$-flower if and only if $f(A)$ is contained in a $0$-flower;
        \smallskip
        \item $A$ {\jsOpSd} of $\cD_1$ if and only if $f(A)$ {\jsOpSd} of $\cD_0$; and
        \smallskip
        \item if $A$ meets both $\sigma,\,\tau$ for some disjoint $\sigma,\,\tau\in\cD_1$, then $f(A)$ {\jsOpSd} of $\cD_0$ and cannot be contained in any $0$-flower.
    \end{enumerate}
\end{cor}

In order to prove Proposition~\ref{p: p-flowers: invariance}, we recall that a cellular branched cover is surjective on the level of chambers. We formulate this as follows.

\begin{lemma}
	\label{l: CBC: Y meet f^k(Fl)=> Y=f(Y')}
	Let $f\colon \mfd^n \to \mfd^n$ be a cellular branched cover on a {\clcntop} $n$-manifold and $(\cD_1,\cD_0)$ be a {\celPair} of $f$. Let $p$ be a vertex in $\cD_1$, and $Y\in \cD_0^{\top}$ satisfy $Y \cap f(\flower_1(p))\neq\emptyset$. Then there exists $\tY\in\cD_{1}^{[n]}$ with $p\in \tY$ such that $Y=f\bigl(\tY\bigr)$.
\end{lemma}
\begin{proof}
	By Proposition~\ref{p: c-flower: structure}~(i), the subset $\flower_{1}(p)$ is open. Since $f$ is an open map, $f(\flower_{1}(p))$ is open. Then since $Y=\cls{\stdCellint{Y}}$ (cf.~Lemma~\ref{l: properties of cells}), we have $ f(\flower_1(p))\cap\stdCellint{Y}\neq\emptyset$.
	So by Proposition~\ref{p: c-flower: structure}~(ii), there exists $\tY\in\cD_{1}^{\top}$ satisfying $p\in\tY$ and $f\bigl(\tY\bigr)\cap\stdCellint{Y}\neq \emptyset$. Since $\dim(Y)=\dim\bigl(f\bigl(\tY\bigr)\bigr)=n$, it follows from Lemma~\ref{l: properties of cell decompositions}~(iv) that $Y=f\bigl(\tY\bigr)$.
\end{proof}

We also recall a topological fact; see e.g.,~\cite[Lemma~5.4]{BM17}.

\begin{lemma}\label{l: BM-Lemma-5.4}
	Let $A\subseteq \fX$ be a closed subset of a locally-compact Hausdorff space $\fX$, and $U\subseteq \fX\smallsetminus A$ be a non-empty connected open set with $\partial U\subseteq A$. Then $U$ is a connected component of $\fX\smallsetminus A$.
\end{lemma}

\begin{proof}[Proof of Proposition~\ref{p: p-flowers: invariance}]

	(i) Let $p$ be a $1$-vertex and denote $q\=f(p)$. Let $c\in\cD_0$ be a cell containing $q$. By Lemma~\ref{l: cell decomp: compact, and Mfd}~(ii) there exists $X\in\cD_0^{\top}$ containing $c$. 
	Then by Lemma~\ref{l: CBC: Y meet f^k(Fl)=> Y=f(Y')}, there exists $\tX\in\cD_{1}^{\top}$ satisfying $p\in\tX$ and $X=f\bigl(\tX\bigr)$.
	By Corollary~\ref{c: cellular: cellular homeo}, $\cD_{1}|_{\tX}=\bigl(f|_{\tX}\bigr)^*(\cD_0|_X)$. Denote $\tc\=\bigl(f|_{\tX}\bigr)^{-1}(c)\in\cD_{1}|_{\tX}$, then $p\in\tc$ and $\stdCellint{c}=f(\stdCellint{\tc})\subseteq f(\flower_{1}(p))$. Since $c\in\cD_0$ containing $q$ is arbitrary, we have $\flower_0(q)\subseteq f(\flower_{1}(p))$. On the other hand, by Proposition~\ref{p: c-flower: fwd invariance}~(ii), $f(\flower_{1}(p))\subseteq\flower_0(q)$.

	\smallskip
	
	(ii) It is clear that $f^{-1}(q)$ is a collection of $1$-vertices.
	Assuming $p\in f^{-1}(q)$, we show that $\flower_1(p)$ is a connected component of $f^{-1}(\flower_0(q))$. With Lemma~\ref{l: BM-Lemma-5.4} and the fact that $\flower_{1}(p)$ is open and connected cf.~Proposition~\ref{p: c-flower: structure}~(i)), it suffices to show $\partial\flower_{1}(p)\subseteq \mfd^n\smallsetminus f^{-1}(\flower_0(q))$, because
    By Proposition~\ref{p: c-flower: fwd invariance}~(ii) and the fact that $\flower_0(q)$ is open (cf.~Proposition~\ref{p: c-flower: structure}~(i)),  we have $f(\partial\flower_1(p))\subseteq\partial\flower_0(q)\subseteq\mfd^n\smallsetminus\flower_0(q)$. Thus, $\partial\flower_{1}(p)\subseteq \mfd^n\smallsetminus f^{-1}(\flower_0(q))$.
	
    Conversely, we show that each connected component of $f^{-1}(\flower_0(q))$ is a flower. Let $U$ be a connected component of $f^{-1}(\flower_0(q))$. Then $f(U)=\flower_0(q)$ (cf.~Remark~\ref{r: forward inv of pullback} and Lemma~\ref{l: open+closed => f(tU)=U}).
    Let $p\in U$ satisfy $f(p)=q$. Then $p$ is a $1$-vertex and $\flower_1(p)$ is a connected component of $f^{-1}(\flower_0(q))$ as shown above. Thus, we have $U=\flower_{1}(p)$.
    It follows immediately that $\coverF(\cD_1)=f^*\coverF(\cD_0)$.
\end{proof}

\subsection{Local multiplicity and branch set}
Now we consider the local multiplicity and the structure of the branch set of a cellular branched cover, in terms of its combinatorial data. As a consequence, we obtain a description of the branch set of a cellular branched cover.

\begin{lemma}\label{l: multi: i(x,f)=max card U(Y)}
	Let $f\:\mfd^n\to\mfd^n$ be a cellular map on a {\clcntop} $n$-manifold and $(\cD_1,\cD_0)$ be a {\celPair} of $f$. Then for each $x\in\mfd^n$,
	\begin{equation*}
	\ind(x,f)=\max_{\tau\in\cD_0,\,f(x)\in\tau}\card\{\sigma\in\cD_1:x\in\sigma,\,f(\sigma)=\tau\}.
\end{equation*}
\end{lemma}
\begin{proof}
	Fix arbitrary $x\in\mfd^n$.
	For each $\tau\in\cD_0$ with $f(x)\in\tau$, denote
	\begin{equation*}
		\cU(\tau)\=\{\sigma\in\cD_1:x\in\sigma,\,f(\sigma)=\tau\}.
	\end{equation*}

    First, we show 
	$\ind(x,f) \geq  \max\{\card (\cU(\tau)):\tau\in\cD_0,\, f(x)\in\tau\}$.
	Fix $\tau\in\cD_0$ with $f(x)\in\tau$ and enumerate $\cU(\tau)=\{\sigma_1,\,\dots,\,\sigma_m\}$. Let $\{y_i\}_{i\in\N}$ be a sequence in $\stdCellint{\tau}$ such that $y_i\to f(x)$ as $i\to+\infty$. By the definition of cellular maps, for each $1 \leq  j \leq  m$, the map $f|_{\sigma_j}\:\sigma_j\to\tau$ is homeomorphic. Thus, we can choose $x^j_i\=\bigl(f|_{\sigma_j}\bigr)^{-1}(y_i)$ for each $i\in\N$ and each $1 \leq  j \leq  m$.
    Combining this construction with the fact that $\stdCellint{\sigma_1},\,\dots,\,\stdCellint{\sigma_m}$ are pairwise disjoint, we have
    \begin{enumerate}[label=(\roman*),font=\rm]
    \smallskip
        \item $x^1_i,\,\dots,\,x^m_i$ are pairwise distinct points for each $i\in\N$;
        \smallskip
        \item $f\bigl(x^1_i\bigr)=\cdots=f(x^m_i)=y_i$ for each $i\in\N$; and
        \smallskip
        \item $x^j_i\to x$ as $i\to+\infty$ for each $1\leq j\leq m$.
    \end{enumerate}
    Then it follows from the definition of $i(x,f)$ (cf.~(\ref{e: N_loc(f,x)})) that $\ind(x,f) \geq  m$.
	
	To show $\ind(x,f) \leq  \max\{\card (\cU(\tau)):\tau\in\cD_0,\,f(x)\in\tau\}$, we argue by contradiction and assume
	$\ind(x,f)> m\=\max\{\card(\cU(\tau)):\tau\in\cD_0,\,f(x)\in\tau\}$. 
	By Proposition~\ref{p: c-flower: structure}~(i), $\flower_1(x)$ is an open neighborhood. Then by the definition of $\ind(x,f)$ (cf.~\ref{e: N_loc(f,x)}),
	$
	N(f,\flower_1(x)) \geq  \ind(x,f)> m
	$,
	and thus (by (\ref{e: N(f,A)})) there exist distinct $x_1,\,\dots,x_{m+1}\in\flower_1(x)$ and $y\in f(\flower_1(x))$ such that $f(x_1)=\cdots=f(x_{m+1})=y$. By~(\ref{e: Fl(x)}), we can choose $\sigma_i\in\cD_1$ for each $1 \leq  i \leq  m+1$ satisfying $x\in\sigma_i$ and $x_i\in\stdCellint{\sigma_i}$. By Proposition~\ref{p: c-flower: fwd invariance}~(ii), $y\in\flower_0(f(x))$, and thus there is $\tau\in\cD_0$ satisfying $f(x)\in\tau$ and $y\in\stdCellint{\tau}$. For each $1 \leq  i \leq  m+1$, since $y\in\brCellint{f(\sigma_i)}\cap\brCellint{\tau}\neq\emptyset$, we have $f(\sigma_i)=\tau$ and $\sigma_i\in\cU(\tau)$. 
    For all $1 \leq  i<j \leq  m+1$, since $x_i\neq x_j$, $f(x_i)=y=f(x_j)$, and $f|_{c}$ is injective for each $c\in\cD_1$, we have $\sigma_i\neq\sigma_j$.
    Hence, $\card(\cU(\tau)) \geq  m+1$, which contradicts the choice of $m$.
    
	Therefore, $\ind(x,f)=\max\{\card(\cU(\tau)):\tau\in\cD_0,\,f(x)\in\tau\}$.
\end{proof}

The above lemma yields a two-sided bound for the local multiplicity of a cellular map in terms of the combinatorial data. Given a cell decomposition $\cD$ of a {\clcntop} $n$-manifold $\mfd^n$, we denote
\begin{equation}\label{e: N(D,x)}
	\NxD{x}{\cD}\=\card\bigl\{X\in\cD^{[n]}:x\in X\bigr\}.
\end{equation}

\begin{lemma}\label{l: multiplicity: N_loc(f,x) and N(D,x)}
	Let $f\:\mfd^n\to\mfd^n$ be a cellular map on a {\clcntop} $n$-manifold, and $(\cD_1,\cD_0)$ be a {\celPair} of $f$. Then $\ind(x,f) \leq  \NxD{x}{\cD_1} \leq \NxD{f(x)}{\cD_0}\cdot \ind(x,f)$ for each $x\in\mfd^n$. In particular, $\NxD{x}{\cD_1}\leq \ind(x,f)\cdot\card\cD_0$.
\end{lemma}
\begin{proof}
	Fix arbitrary $x\in\mfd^n$. Denote $\cX\=\bigl\{X\in\cD_1^{[n]}:x\in X\bigr\}$ and $V\=\bigcup\cX$. By Proposition~\ref{p: c-flower: structure}~(i) and~(ii), $V=\cls{\flower_1(x)}$ and $x\in\flower_1(x)\subseteq\stdInt{V}$. By the definition of cellular maps, $f|_X$ is injective for each $X\in\cX$ which, combined with (\ref{e: N(f,A)}), yields $N(f,V) \leq \card\cX=\NxD{x}{\cD_1}$. Consequently, (\ref{e: N_loc(f,x)}) yields $\ind(x,f) \leq  N(f,V) \leq  \NxD{x}{\cD_1}$. The first inequality is now verified.
	
	Now we verify the second inequality. Denote $\cY\=\bigl\{Y\in\cD_0^{[n]}:f(x)\in Y\bigr\}$ and set $\cU(Y)\=\{X\in\cX:f(X)=Y\}$ for each $Y\in\cY$. It is clear that $\{f(X):X\in\cX\}\subseteq\cY$, and consequently
	$\card\cX=\sum_{Y\in\cY}\card(\cU(Y)) \leq \card\cY\cdot\max_{Y\in\cY}\card(\cU(Y))$.
	Then by Lemma~\ref{l: multi: i(x,f)=max card U(Y)}, we have $\card\cX \leq \ind(x,f)\cdot\card\cY=\ind(x,f)\cdot \NxD{f(x)}{\cD_0}$.
\end{proof}

The function $x \mapsto \NxD{x}{\cD}$ attains its local maxima at vertices of $\cD$.
\begin{lemma}\label{l: multiplicity: N(D,x)<=N(D,p)}
	Let $\cD$ be a cell decomposition of a closed topological $n$-manifold $\mfd^n$. Then there exists a vertex $q$ with $\NxD{q}{\cD}=\max\{\NxD{x}{\cD}:x\in\mfd^n\}$.
\end{lemma}
\begin{proof}
	Let $x\in\mfd^n$ be arbitrary and denote $c\=\supp_\cD(x)$ (cf.~Definition~\ref{d: carrier}). By Lemma~\ref{l: cell decomp: compact, and Mfd}~(iii), there exists a vertex $p\in c$. For each $X\in\cD^{[n]}$, if $x\in X$, then $X\cap\stdCellint{c}\neq\emptyset$ and we have by Lemma~\ref{l: properties of cell decompositions}~(v) that $p\in c\subseteq X$. It follows that $\NxD{x}{\cD} \leq \NxD{p}{\cD}$. Since the set of all vertices is finite (cf.~Lemma~\ref{l: cell decomp: compact, and Mfd}~(i)), the lemma follows.
\end{proof}

Lemmas~\ref{l: multiplicity: N_loc(f,x) and N(D,x)} and \ref{l: multiplicity: N(D,x)<=N(D,p)} yield the following corollary.
\begin{cor}\label{c: multiplilcity: characterization of bounded N_loc}
	Let $f\:\mfd^n\to\mfd^n$ be an {\itCel} map on a {\clcntop} $n$-manifold and $\{\cD_m\}_{m\in\N_0}$ be a {\celSeq} of $f$. Then the following statements are equivalent:
	\begin{enumerate}[label=(\roman*),font=\rm]
    	\smallskip
		\item $\sup\{\ind(x,f^m):m\in\N,\,x\in\mfd^n\}<+\infty$.
		\smallskip
		\item $\sup\{\NxD{x}{\cD_m}:m\in\N,\,x\in\mfd^n\}<+\infty$.
		\smallskip
		\item $\sup\{\NxD{p}{\cD_m}:m\in\N,\,p\text{ is a vertex of }\cD_m\}<+\infty$.
	\end{enumerate}
\end{cor}

Lemma~\ref{l: multi: i(x,f)=max card U(Y)} yields the following description of the local multiplicity function.
\begin{cor}\label{c: multi: N(f,x)=const on int(c)}
	Let $f\:\mfd^n\to\mfd^n$ be a cellular branched cover on a {\clcntop} $n$-manifold and $(\cD_1,\cD_0)$ be a {\celPair} of $f$. Then {\rm(i)} $\ind(\cdot,f)\:\stdCellint{c}\to\N$ is constant for each $c\in\cD_1$ and {\rm(ii)} for each $c\in\cD_1$, $c\subseteq B_f$ if and only if $B_f\cap\stdCellint{c}\neq\emptyset$.
\end{cor}
\begin{proof}
	(i) Consider $c\in\cD_1$ and $x,\,x'\in\stdCellint{c}$. Then for all $\sigma\in\cD_1$ and $\tau\in\cD_0$, it follows from Lemma~\ref{l: properties of cell decompositions}~(v) that
 $x\in \sigma\Leftrightarrow c\subseteq\sigma\Leftrightarrow x'\in\sigma$ and
 $f(x)\in\tau\Leftrightarrow f(c)\subseteq\tau \Leftrightarrow f(x')\in\tau$.
 Then (i) follows from Lemma~\ref{l: multi: i(x,f)=max card U(Y)}.
	
	\smallskip
	
	(ii) By the definition of the branch set, for each point $x\in\mfd^n$, $x\in B_f$ if and only if $\ind(x,f)>1$. This, together with (i), implies that if $c\in\cD_1$ satisfies $B_f\cap\stdCellint{c}\neq\emptyset$, then 
    $\stdCellint{c}\subseteq B_f$, and thus $c\subseteq B_f$ since $c=\cls{\stdCellint{c}}$ (cf.~Lemma~\ref{l: properties of cells}~(i)) and $B_f$ is closed. The converse implication is clear.
\end{proof}

As a consequence, the branch set of a cellular branched cover admits a cell complex structure.
\begin{prop}\label{p: B_f: is a sub complex}
	Let $f\:\mfd^n\to\mfd^n$ be a cellular branched cover on a {\clcntop} $n$-manifold and $(\cD_1,\cD_0)$ be a {\celPair} of $f$. Then there exists a collection $\cB_f\subseteq\cD_1$ for which $B_f=\abs{\cB_f}$ and $\cB_f$ is a cell decomposition of $B_f$.
\end{prop}
\begin{proof}
	Set $\cB_f\=\{c\in\cD_1:c\subseteq B_f\}$. Now we show that $\cB_f$ is a cell decomposition of $B_f$ by verifying conditions (i) through (iv) in Definition~\ref{d: cell decomposition}.
	
	(i) follows from Lemma~\ref{l: properties of cell decompositions}~(ii) and Corollary~\ref{c: multi: N(f,x)=const on int(c)}. 
	(ii) follows from the fact that $\cB_f\subseteq\cD_1$ and $\cD_1$ is a cell decomposition. To show (iii), consider arbitrary $c\in\cB_f$. If $\sigma\in\cD_1$ satisfies $\sigma\subseteq\cellbound c$, then since $\cellbound c\subseteq B_f$, $\sigma\in\cB_f$. It follows that $\cellbound c=\bigcup\{\sigma\in\cD_1:\sigma\subseteq\cellbound c\}=\bigcup\{\sigma\in\cB_f:\sigma\subseteq\cellbound c\}$. (iv) follows from the fact that $\cD_1$ is finite.
\end{proof}

\subsection{Expansion of {\itCel} branched covers}\label{subsct: cellular Expansion}
First, for branched covers on connected spaces, we show a consequence of expansion (Definition~\ref{d: expansion}): the \emph{topological exactness}.

\begin{prop}\label{p: expans=>LEO}
	Let $\fX$ be a compact connected locally-connected metric space, and let $f\:\fX\to\fX$ be an expanding branched cover. Then $f$ is topologically exact, i.e., for each non-empty open subset $U$ of $\fX$, there exists $m\in\N$ such that $f^m(U)=\fX$.
\end{prop}
\begin{proof}
	Let $U$ be a non-empty open subset of $\fX$. Suppose $x\in U$ and $B(x,r)\subseteq U$ with some $r>0$. Let $\cU$ be a finite cover of $\fX$ by connected open subsets for which $\mesh((f^m)^*\cU)\to 0$ as $m\to+\infty$. Then there exists $m\in\N$ for which $\mesh((f^m)^*\cU)<r/\card\cU$.
	
	For each subset $\cU'\subseteq (f^m)^*\cU$, we denote $f^m(\cU')\=\{f^m(U'):U'\in\cU'\}\subseteq\cU$ (cf.~Remark~\ref{r: forward inv of pullback}).
	Let $\tU_1\in(f^m)^*\cU$ contain $x$, and denote $\cU_1\=\bigl\{\tU_1\bigr\}$. We recursively define a sequence $\{\cU_i\}_{i\in\N}$ of subsets of $(f^m)^*\cU$. Suppose that $\cU_i$ is defined for $i\in\N$. If $f^m(\cU_i)\subseteq\cU$ covers $\fX$, then define $\cU_{i+1}\=\cU_i$. Otherwise, suppose $f^m(\cU_i)\neq\cU$. Then since $\cU$ is a finite open cover of the connected space $\fX$, there exists $U_{i+1}\in\cU\smallsetminus f^m(\cU_i)$ that intersects $\bigcup f^m(\cU_i)$. Thus, there exists $\tU_{i+1}\in(f^m)^*\cU$ such that $\tU_{i+1}$ is a connected component of $f^{-m}(U_{i+1})$ that intersects $\bigcup\cU_i$, and we define $\cU_{i+1}\=\cU_i\cup\bigl\{\tU_{i+1}\bigr\}$. 
	
	By the above construction and the finiteness of $\cU$, there exists $1 \leq  t \leq \card\cU$ for which $f^m(\bigcup\cU_t)=\bigcup f^m(\cU_t)=\fX$. The construction of the sequence $\{\cU_i\}_{i\in\N}$ implies
$\diam(\bigcup\cU_t) \leq \card\cU\cdot\mesh((f^m)^*\cU)<r$. 
	Then $\bigcup\cU_t\subseteq U$, and thus $f^m(U)=\fX$. 
\end{proof}

For {\itCel} branched covers, expansion can be characterized in terms of expanding {\celSeq}s. In what follows, we consider \emph{metric manifolds}, i.e.,~topological manifolds with metrics compatible with the given topology.

\begin{definition}
	A {\celSeq} $\{\cD_m\}_{m\in\N_0}$ of
	an {\itCel} map on a {\clcnmtr} $n$-manifold $\mfd^n$ is \defn{expanding} if $\mesh(\cD_m)\to0$ as $m\to+\infty$.
\end{definition}

\begin{prop}\label{p: characterization of expans}
	Let $f\:\mfd^n\to\mfd^n$ be an {\itCel} branched cover on a {\clcnmtr} $n$-manifold. Then the following statements are equivalent:
	\begin{enumerate}[label=(\roman*),font=\rm]
    	\smallskip
		\item The map $f$ is expanding.
		\smallskip
		\item There exists an expanding {\celSeq} of $f$.
		\smallskip
		\item Each {\celSeq} of $f$ is expanding.
	\end{enumerate}
\end{prop}

To prove Proposition~\ref{p: characterization of expans} we state the following lemma. When applying this lemma, we usually choose $\cU\=\coverF(\cC_0)$ where $\{\cC_m\}_{m\in\N_0}$ is a {\celSeq} of $f$.

\begin{lemma}\label{l: CBC: cover cell by M flowers}
	Let $f\:\mfd^n\to\mfd^n$ be an {\itCel} branched cover on a {\clcntop} $n$-manifold, $\{\cD_m\}_{m\in\N_0}$ be a {\celSeq} of $f$, and $\cU$ be a finite cover of $\mfd^n$ by connected open subsets. Then there exists $M=M(\cD_0,\cU)\in\N$ having the following property: for each $m\in\N_0$ and each $X\in\cD_m$, there exists a cover $\cU_X$ of $X$ by elements in $(f^m)^*\cU$ for which $\card\cU_X \leq  M$.
\end{lemma}
\begin{proof}
	For each $Y\in\cD_0$, we fix a finite cover $\cA_Y$ of $Y$ by connected subsets of $Y$ such that each element of $\cA_Y$ is contained in some element in $\cU$. For example, one may consider 
	\begin{equation*}
	\cA_Y\=\bigl\{\phi_Y^{-1}(I):I\in\{[0,1/h],\,[1/h,2/h],\,\dots,\,[(h-1)/h,1]\}^d\bigr\},
\end{equation*} 
	where $\phi_Y\:Y\to[0,1]^d$, $d\=\dim(Y)$, is a homeomorphism and $h\in\N$ is sufficiently large. Denote $M\=\max\{\card\cA_Y:Y\in\cD_0\}$.
	
	Fix arbitrary $m\in\N_0$ and $X\in\cD_m$. Denote $Y\=f^m(X)\in\cD_0$.
	Since $f^m|_X\:X\to Y$ is a homeomorphism (cf.~Remark~\ref{r: basic cellular property}), $\cA_X\=\bigl\{(f^m|_X)^{-1}(A):A\in\cA_Y\bigr\}$ is a finite cover of $X$ by connected subsets of $X$ with $\card\cA_X=\card\cA_Y$.
	Then for each element $A\in\cA_X$, $f^m(A)$ is contained in an element in $\cU$, and it follows from the definition of $(f^m)^*\cU$ (cf.~(\ref{e: pullback of cover})) that each element in $\cA_X$ is contained in an element in $(f^m)^*\cU$. Hence, $X$ can be covered by no more than $M$ elements in $(f^m)^*\cU$.
\end{proof}

\begin{proof}[Proof of Proposition~\ref{p: characterization of expans}]
	First, we verify (i)$\Rightarrow$(iii). Suppose $f$ is expanding and let $\cU$ be a finite cover of $\mfd^n$ by connected open sets for which $\mesh((f^m)^*\cU)\to 0$ as $m\to+\infty$. Let $\{\cD_m\}_{m\in\N_0}$ be a {\celSeq} of $f$. Then by Lemma~\ref{l: CBC: cover cell by M flowers}, we can fix a number $M\in\N$ such that for each $m\in\N_0$, each $X\in\cD_m$ is covered by at most $M$ elements of $(f^m)^*\cU$. This implies $\mesh(\cD_m) \leq  M\mesh((f^m)^*\cU)$. Therefore, $\{\cD_m\}_{m\in\N}$ is an expanding {\celSeq} of $f$.

    (iii)$\Rightarrow$(ii) is clear, and it suffices to show (ii)$\Rightarrow$(i). If $\{\cD_m\}_{m\in\N_0}$ is an expanding {\celSeq} of $f$, then $\coverF(\cD_0)$ is a cover by connected open sets for which the mesh of $\coverF(\cD_m)=(f^m)^*\coverF(\cD_0)$  converges to $0$ as $m\to+\infty$ (cf.~Propositions~\ref{p: c-flower: structure}~(ii) and~\ref{p: p-flowers: invariance}). Then $f$ is expanding.
\end{proof}

Given a cell decomposition $\cD$ of a {\clcntop} $n$-manifold $\mfd^n$, and a collection $\fC$ of subsets of $\mfd^n$, we denote by $J(\fC,\cD)$ the minimal number of elements in $\fC$ required to form a ``chain" that {\jsOpSd} (cf.~Definition~\ref{d: join_opposite_sides}) of $\cD$. More precisely, set 
\begin{equation}\label{e: the joining number J(m>k)}
	\begin{aligned}
		J(\fC,\cD)\=\min\Bigl\{\card\cA:\cA\subseteq\fC,\,\bigcup\cA\text{ is connected and {\jsOpSd} of }\cD\Bigr\}.
	\end{aligned}
\end{equation}
Note that a similar construction is given in \cite[Section~8.1]{BM17}.

\begin{lemma}\label{l: expansion: J_m--->+infty}
	Let $\{\cD_m\}_{m\in\N_0}$ be an expanding {\celSeq} of an {\itCel} map $f\:\mfd^n\to\mfd^n$ on a {\clcnmtr} $n$-manifold. 
	Then $J\bigl(\bigcup_{i \geq  m}\cD_i,\cD_0\bigr)\to +\infty$ 
	and $J(\cD_m,\cD_0)\to+\infty$ as $m\to+\infty$.
\end{lemma}
\begin{proof}
	Let $d$ be a metric on $\mfd^n$ compatible with the topology. Denote by $\delta_0$ the Lebesgue number of $\coverF(\cD_0)$.
    For each $m\in\N_0$, if $\cA\subseteq\bigcup_{i \geq  m}\cD_i$ is a finite subset such that $\bigcup\cA$ is connected and {\jsOpSd} of $\cD_0$, then the connectedness of $\bigcup\cA$ implies $\diam(\bigcup\cA)\leq\sup\{\mesh(\cD_i):i \geq  m\}\cdot\card\cA $, and thus, by Lemma~\ref{l: decomp of S^n: join opp <=> not in flower}, $\sup\{\mesh(\cD_i):i \geq  m\}\cdot\card\cA\geq \delta_0$.
    Thus, by (\ref{e: the joining number J(m>k)}), $J\bigl(\bigcup_{i \geq  m}\cD_i,\cD_0\bigr)\geq\delta_0/\sup\{\mesh(\cD_i):i \geq  m\}$.

    Since $f$ is expanding, $\sup\{\mesh(\cD_i):i \geq  m\}\to0$, and thus $J\bigl(\bigcup_{i \geq  m}\cD_i,\cD_0\bigr)\to +\infty$ as $m\to+\infty$. Since $\cD_m\subseteq\bigcup_{i \geq  m}\cD_i$, it follows from (\ref{e: the joining number J(m>k)}) that $J(\cD_m,\cD_0) \geq  J\bigl(\bigcup_{i \geq  m}\cD_i,\cD_0\bigr)$, and thus $J(\cD_m,\cD_0)\to+\infty$ as $m\to+\infty$.
\end{proof}

\subsection{Construction of {\celSeq}s}\label{subsct: CPCF and CTM}

Recall the CPCF condition from Definition~\ref{def:CPCF}, and recall that a Thurston-type map is a CPCF branched cover with topological degree at least $2$.
In this subsection, we prove the key property discussed in the introduction
that CPCF branched covers admit {\celSeq}s. More precisely, we establish the following main theorem of this subsection.

\begin{theorem}\label{t: CPCF=>encoding sequence}
	Let $f\:\mfd^n\to\mfd^n$ be a CPCF branched cover on a {\clcntop} $n$-manifold. Then $f$ admits a {\celSeq} $\{\cD_m\}_{m\in\N}$ satisfying the following conditions:
	\begin{enumerate}[label=(\roman*),font=\rm]
    	\smallskip
		\item For each $m\in\N_0$, there is a cell complex $\cP_m\subseteq\cD_m$ such that $P_f=\abs{\cP_m}$.
		\smallskip
		\item $\{\cP_m\}_{m\in\N}$ is a {\celSeq} of $f|_{P_f}$ such that $\cP_{m+1}$ refines $\cP_m$ for each $m\in\N_0$.
	\end{enumerate}
\end{theorem}

    Theorem \ref{t: CPCF=>encoding sequence} refines Theorem~\ref{tx: cellular sequence} in the introduction and we obtain it by pulling back the CPCF data $(\cB,\cP;\cP_0,\cD_0)$ to produce a diagram
	\begin{equation*}
		\xymatrix{
			\cdots&\cB_{m+1} \ar[dr]^{\!\!\!\!\!(f|_{B_f})_*} & & \cdots & \cB_1 \ar[dr]^{\!\!\!\!\!(f|_{B_f})_*} &  \cB \ar[dr]^{\!\!\!\!\!(f|_{B_f})_*} &\\
			\cdots \ar[r]_{(f|_{P_f})_*}  & \cP_{m+1} \ar[r]_{(f|_{P_f})_*}  \ar@{^(->}[d] & \cP_m \ar[r]_{(f|_{P_f})_*}  \ar@{^(->}[d] & \cdots \ar[r]_{(f|_{P_f})_*} & \cP_1 \ar[r]_{(f|_{P_f})_*}  \ar@{^(->}[d] & \cP_0 \ar[r]_{(f|_{P_f})_*} \ar@{^(->}[d] & \cP \\
			\cdots \ar[r]^{f_*} & \cD_{m+1} \ar[r]^{f_*} & \cD_m \ar[r]^{f_*} & \cdots  \ar[r]^{f_*} & \cD_1 \ar[r]^{f_*} & \cD_0 & 
		}
	\end{equation*}
	which is interpreted as follows:
	\begin{enumerate}[label=(\roman*),font=\rm]
		\item $\{\cD_m\}_{m\in\N_0}$, $\{\cB_m\}_{m\in\N_0}$, and $\{\cP_m\}_{m\in\N_0}$ are sequences of cell decompositions of $\mfd^n$, $B_f$, and $P_f$, respectively. Moreover, for each $m\in\N_0$, $\cB_m\subseteq\cD_m$, $\cP_m\subseteq\cD_m$, and $\cP_{m+1}$ is a refinement of $\cP_m$.
		\smallskip
		\item For each $m\in\N_0$, $f \colon \mfd^n\to\mfd^n$ is $(\cD_{m+1},\cD_m)$-cellular, $f|_{B_f}\:B_f\to B_f$ is $(\cB_{m+1},\cP_m)$-cellular, and $f|_{P_f} \colon P_f \to P_f$ is $(\cP_{m+1},\cP_m)$-cellular.
	\end{enumerate}

We begin by considering a natural class of cellular maps, the \emph{cellular Markov maps}, which are {\itCel} maps and, in the context of branch covers, give a subclass of maps that are CPCF. In this subsection, $\fX$ stands for a locally-compact Hausdorff space.

\begin{definition}[Cellular Markov partitions, cellular Markov maps]\label{d: cellular Markov}
	A \defn{cellular Markov partition} of a continuous map $f\:\fX\to\fX$ is a {\celPair} $(\cD',\cD)$ of $f$ where $\cD'$ is a refinement of $\cD$. We call $f$ a \defn{cellular Markov map} if there exists a cellular Markov partition of $f$.
\end{definition}

A sequence $\{\cD_m\}_{m\in \N}$ of cell decompositions of $\fX$ is \defn{a {\MkvSeq}} of $f$ if, for each $m\in \N$, the pair $(\cD_{m},\cD_{m-1})$ is a cellular Markov partition of $f$. We show that a cellular Markov partition induces a {\MkvSeq}.

\begin{prop}
\label{p: cellular Markov: induced ess-seq}
Let $f\colon \fX\to \fX$ be a cellular Markov map and $(\cD',\cD)$ be a cellular Markov partition of $f$. Then there exists a unique {\MkvSeq} $\{\cD_m\}_{m\in \N}$ of $f$ with $\cD_0 = \cD$.
\end{prop}

Proposition~\ref{p: cellular Markov: induced ess-seq} is obtained by repeatedly applying the following lemma. Recall that $\phi^*\cD$ denotes the pullback of a cell decomposition $\cD$ under a homeomorphism $\phi$, given by (\ref{e: homeo pullback}).

\begin{lemma}\label{l: cellular Markov: pullbcak of ess also ess}
	Let $f\:\fX\to\fX$ be a cellular Markov map, and $(\cD_1,\cD_0)$ be a cellular Markov partition of $f$. Then
	\begin{equation*}
	\cD_2\=\bigcup_{\sigma\in\cD_1}(f|_{\sigma})^*\bigl(\cD_1|_{f(\sigma)}\bigr)
	\end{equation*}
	is a cell decomposition of $\fX$ and $(\cD_2,\cD_1)$ is a cellular Markov partition of $f$.
\end{lemma}
\begin{proof}
	For each $\sigma\in\cD_1$, since $\cD_1|_{f(\sigma)}$ is a cell decomposition that refines $\cD_0|_{f(\sigma)}$ (cf.~Corollary~\ref{c: refinement: restriction}), and $f|_{\sigma}\:\sigma\to f(\sigma)$ is a homeomorphism, it follows from Corollary~\ref{c: cellular: cellular homeo} that $(f|_\sigma)^*\bigl(\cD_1|_{f(\sigma)}\bigr)$ is a well-defined cell decomposition of $\sigma$ that refines $\cD_1|_\sigma=(f|_\sigma)^*\bigl(\cD_0|_{f(\sigma)}\bigr)$.
    
    In what follows, we apply Lemma~\ref{l: construct refinement}. Denote $\cD_2(\sigma)\=(f|_{\sigma})^*\bigl(\cD_1|_{f(\sigma)}\bigr)$ for each $\sigma\in\cD_1$.
    Let $c_1,\,\sigma_1\in\cD_1$ satisfy $c_1\subseteq\sigma_1$. By the construction of $\cD_2(\sigma_1)$, $f|_{\sigma_1}$ is a homeomorphism that is also $\bigl(\cD_2(\sigma_1),\cD_1|_{f(\sigma_1)}\bigr)$-cellular, and thus, by Corollary~\ref{c: cellular: cellular homeo},
    \begin{equation*}
        \begin{aligned}
            \cD_2(\sigma_1)|_{c_1}=((f|_{\sigma_1})|_{c_1})^*\bigl(\bigl(\cD_1|_{f(\sigma_1)}\bigr)|_{f(c_1)}\bigr)=(f|_{c_1})^*\bigl(\cD_1|_{f(c_1)}\bigr)=\cD_2(c_1).
        \end{aligned}
    \end{equation*}
 
    Thus, by Lemma~\ref{l: construct refinement}, $\cD_2$ is a cell decomposition that refines $\cD_1$.
	By the construction of $\cD_2$, $f$ is $(\cD_2,\cD_1)$-cellular. Hence, $(\cD_2,\cD_1)$ is a cellular Markov partition of $f$.
\end{proof}

It is easy to check that cellular Markov branched covers are CPCF.
\begin{prop}\label{p: Mkv => CPCF}
	A cellular Markov branched cover $f\:\mfd^n\to\mfd^n$ on a {\clcntop} $n$-manifold is a CPCF branched cover.
\end{prop}
\begin{proof}
	Let $\{\cD_m\}_{m\in\N_0}$ be a {\MkvSeq} of $f$ from Proposition~\ref{p: cellular Markov: induced ess-seq}. By Proposition~\ref{p: B_f: is a sub complex}, for each $m\in\N$, there exists a subcomplex $\cB_m$ of $\cD_m$ such that $B_f=\abs{\cB_m}$. Denote $\cP_0\=\bigcup_{m\in\N}\{f^m(c):c\in\cB_m\}\subseteq\cD_0$ and $\cP_1\=\bigcup_{m\in\N}\{f^m(c):c\in\cB_{m+1}\}\subseteq\cD_1$. By the construction of $\cP_1$ and $\cP_0$, it is easy to see from Definition~\ref{d: cell decomposition} that $\cP_1$ and $\cP_0$ are cell decompositions of $P_f$ for which $f|_{P_f}$ is $(\cP_1,\cP_0)$-cellular and $f|_{B_f}$ is $(\cB_1,\cP_0)$-cellular. By Lemma~\ref{c: multi: N(f,x)=const on int(c)}, $i(\cdot,f)|_c$ is constant for each $c\in\cB_1$.
    Since $\cP_1$ and $\cP_0$ are subcomplexes of $\cD_1$ and $\cD_0$, respectively, by Lemma~\ref{l: cell decomp: restrict}~(iii), $\cP_1=\cD_1|_{P_f}$ and $\cP_0=\cD_0|_{P_f}$.
    Thus, by Corollary~\ref{c: refinement: restriction}, $\cP_1$ is a refinement of $\cP_0$.
    Therefore, $f$ is CPCF with data $(\cB_1,\cP_0;\cP_1,\cD_1)$. 
\end{proof}

Now we are ready to establish Theorem~\ref{t: CPCF=>encoding sequence}, which is obtained by repeatedly applying the following Proposition~\ref{p: CPCF: CPCF=>cellular CPCF}.

\begin{prop}\label{p: CPCF: CPCF=>cellular CPCF}
	If $f\:\mfd^n\to\mfd^n$ is a CPCF branched cover on a {\clcntop} $n$-manifold with data $(\cB,\cP;\cP_0,\cD_0)$, then there exist a cell decomposition $\cD_1$ of $\mfd^n$ and cell complexes $\cB_1,\,\cP_1\subseteq\cD_1$ such that
	\begin{enumerate}[font=\rm,label=(\roman*)]
    	\smallskip
		\item $f$ is $(\cD_1,\cD_0)$-cellular and
		\smallskip
		\item $f$ is CPCF with data $(\cB_1,\cP_0;\cP_1,\cD_1)$.
	\end{enumerate}
\end{prop}

Before proving Proposition~\ref{p: CPCF: CPCF=>cellular CPCF}, we formulate several technical lemmas. First, we use the following lemma to get a ``branch of inverse of $f$" on the interior of each chamber in $\cD_0$.

\begin{lemma}\label{l: constructing covering map U'-> U}
    Let $f\:\mfd^n\to\mfd^n$ be a branched cover on a {\clcntop} $n$-manifold $\mfd^n$. Let $V\subseteq\mfd^n\smallsetminus f(B_f)$ be a connected open set and $U$ be a connected component of $f^{-1}(V)$. Then $f|_U\:U\to V$ is a covering map. Moreover, if $V$ is simply-connected, then $f|_U\:U\to V$ is a homeomorphism.
\end{lemma}
\begin{proof}
    By Remark~\ref{r: bc on mfd is cls fbc} and Lemma~\ref{l: open+closed => f(tU)=U}, $f|_U$ is surjective and finite-to-one.
    Let $y\in V$ be an arbitrary point. We enumerate $U\cap f^{-1}(y)$ as $\{x_i\}_{i\in\cI}$, where $\cI$ is a finite index set.

    From $V\subseteq\mfd^n\smallsetminus f(B_f)$ we get $U\subseteq\mfd^n\smallsetminus B_f$. Thus, we can choose, for each $i\in\cI$, a small open neighborhood $U_i\subseteq U$ of $x_i$ for which $f|_{U_i}$ is a homeomorphism between $U_i$ and the open neighborhood $f(U_i)$ of $y$. Since $\cI$ is finite, we may assume $U_i$, $i\in\cI$, to be so small that they are pairwise disjoint. 

     Let $D\subseteq V$ be a connected open neighborhood of $y$ for which $\cls{D}\subseteq\bigcap_{i\in\cI}f(U_i)$. For each $i\in\cI$, denote $D_i\=(f|_{U_i})^{-1}(D)$. Then for each $i\in\cI$, $f|_{\cls{D_i}}\:\cls{D_i}\to \cls{D}$ is a homeomorphism, and thus, since $f(\partial D_i)=\partial D$, $D_i$ is the connected component of $f^{-1}(D)$ containing $x_i$ (cf.~Lemma~\ref{l: BM-Lemma-5.4}).
     
     Since $U$ is a connected component of $f^{-1}(V)$, $(f|_U)^{-1}(D)$ is the union of the connected components of $f^{-1}(D)$ that are contained in $U$. Combining this with the fact that each connected component of $f^{-1}(D)$ contains at least one point in $f^{-1}(y)$ (cf.~Remark~\ref{r: bc on mfd is cls fbc} and Lemma~\ref{l: open+closed => f(tU)=U}), we have $(f|_U)^{-1}(D)=\bigcup_{i\in\cI}D_i$. Now we conclude that $D$ is evenly covered by $f|_U$.
 
    Therefore, since $y\in V$ is arbitrary, $f|_U$ is a covering map.
\end{proof}

Then we record a lemma in \cite{BM17}, which allows us to extend the ``branch of inverse" given by the lemma above to the boundary of $V$, leading to the construction of $\cD_1$. 
\begin{lemma}[{\cite[Lemma~5.15]{BM17}}]\label{l: BM: Lemma 5.15}
	Let $\fX,\,\fY$ be metric spaces and $f\:\fX\to \fY$ be a continuous map. Let $\fX$ be compact and $A\subseteq \fY$ be a closed subset. Then for each $\epsilon>0$, there exists $\delta>0$ such that $f^{-1}(\cN_\delta(A))\subseteq \cN_\epsilon\bigl(f^{-1}(A)\bigr)$.	 
\end{lemma}

We also give a lemma that will be used to verify that $\cD_1$ is a cell decomposition.
\begin{lemma}\label{l: constant deg on int(c)=>neighborhood not map to f(c)}
	Let $f\:\mfd^n\to\mfd^n$ be a branched cover on a {\clcntop} $n$-manifold. Let $c\subseteq\mfd^n$ be a cell for which $i(\cdot,f)|_{\stdCellint{c}}$ is constant and $f|_c\:c\to f(c)$ is a homeomorphism. Then there exists an open neighborhood $U_c\supseteq \stdCellint{c}$ with 
	$
	U_c\cap f^{-1}(f(\stdCellint{c}))=\stdCellint{c}.
	$
\end{lemma}
\begin{remark}
    In particular, we may let $f\:\mfd^n\to\mfd^n$ be CPCF with data $(\cB,\cP;\cP_0,\cD_0)$ and $c\in\cB$ (see condition~\ref{item:CPCF-ii} in Definition~\ref{def:CPCF}).
\end{remark}
\begin{proof}
	We show that for each $x\in\stdCellint{c}$, there exists an open neighborhood $U_x$ of $x$ such that $U_x\cap f^{-1}(f(\stdCellint{c}))\subseteq\stdCellint{c}$. Fix $x\in\stdCellint{c}$. We argue by contradiction and assume that $\{x_k\}_{k\in\N}$ is a sequence in $\mfd^n\smallsetminus\stdCellint{c}$ satisfying $x_k\to x$ as $k\to+\infty$ and $y_k\=f(x_k)\in f(\stdCellint{c})$ for all $k\in\N$. Then $y_k\to f(x)$ as $k\to+\infty$. Choose a small open neighborhood $U$ of $x$ such that (\ref{e: i(x,f)=Sum i(f,z)}) holds for each $y\in f(U)$.
	Since $f|_{c}$ is a homeomorphism, we can pick a sufficiently large $N\in\N$ such that $x'\=(f|_{c})^{-1}(y_N)\in U\cap\stdCellint{c}$. Then we have $f(x')=y_N=f(x_N)$. However, since $i(x,f)=i(x',f)$,
	\begin{equation*}
		\ind(x,f)=N(f,U) \geq  \ind(x',f)+\ind(x_N,f)= \ind(x,f)+\ind(x_N,f),	
	\end{equation*}
	which is a contradiction.
	
	Consider the open neighborhood $U_c\=\bigcup_{x\in \stdCellint{c}}U_x$ of $\stdCellint{c}$. Given $z\in U_c$, suppose $z\in U_x$ for some $x\in\stdCellint{c}$. If $f(z)\in f(\stdCellint{c})$, then by the choice of $U_x$ we have $z\in\stdCellint{c}$. It follows that
	$U_c\cap f^{-1}(f(\stdCellint{c}))\subseteq\stdCellint{c}$. The other direction is clear, and thus $U_c\cap f^{-1}(f(\stdCellint{c}))=\stdCellint{c}$.
\end{proof}

Now we are ready to prove Proposition~\ref{p: CPCF: CPCF=>cellular CPCF}. 
\begin{proof}[Proof of Proposition~\ref{p: CPCF: CPCF=>cellular CPCF}]
	Let $f\:\mfd^n\to\mfd^n$ be a CPCF branched cover with data $(\cB,\cP;\cP_0,\cD_0)$.
	By the Chernavskii--V\"ais\"al\"a theorem (see~\cite[Theorem~5.4]{Va66}, cf.~\cite{Ch64, Vu88}), the topological dimension of $B_f$ is no more than $n-2$, which, combined with the fact that $f|_{B_f}$ is $(\cB,\cP)$-cellular and $\cP_0$ refines $\cP$, implies that $f(B_f)$ is contained in the $(n-2)$-skeleton of $\cD_0$.

    \smallskip
    
    {\it Step 1:} We aim to construct $\cD_1$.
    Consider $Y\in\cD_0^{[n]}$, denote $U\=\stdCellint{Y}$, and let $U'$ be a connected component of $f^{-1}(U)$. 

    Since $f(B_f)$ is contained in the $(n-2)$-skeleton of $\cD_0$, $U\subseteq\mfd^n\smallsetminus f(B_f)$. Thus, since $U$ is simply-connected, Lemma~\ref{l: constructing covering map U'-> U} asserts that $f|_{U'}\:U'\to U$ is a homeomorphism. 
    Denote $g\=(f|_{U'})^{-1}\:U\to U'$.
	
    In what follows, we extend $g$ to a homeomorphism $\tg\:Y\to\cls{U'}$ satisfying $\tg\=\bigl(f|_{\cls{U'}}\bigr)^{-1}$. 
    Fix $w\in\partial U$ and let $\{w_i\}_{i\in\N}$ be a sequence in $U$ satisfying $w_i\to w$ as $i\to+\infty$.  
	Enumerate $f^{-1}(w)\=\{z_1,\,\dots,\,z_m\}$ and     
    let $\epsilon>0$ be so small that the closed balls $\cls{B(z_i,\epsilon)}$, $1 \leq  i \leq  m$, are pairwise disjoint. Then, by Lemma~\ref{l: BM: Lemma 5.15}, we may choose $\delta>0$ satisfying $f^{-1}(B(w,\delta))\subseteq\bigcup_{i=1}^{m}B(z_i,\epsilon)$. Let $\phi\:\cls{U}\to\cls{\B^n}$ be a homeomorphism and denote, for each $r\in(0,1)$, $V_r\=\phi^{-1}(\B^n\cap\B^n(\phi(w),r))$. Choose a sufficiently small $r\in(0,1)$ such that $\cls{V_r}\subseteq B(w,\delta)$
    and denote $V\=V_r$.
    Then since $f^{-1}(B(w,\delta))\subseteq\bigcup_{i=1}^{m}B(z_i,\epsilon)$ and $\cls{B(z_i,\epsilon)}$, $1\le i\le m$, are pairwise disjoint, there exists a unique $1 \leq  k \leq  m$ such that $g(V)\subseteq B(z_k,\epsilon)$. 
    By the construction of $V$, for each sufficiently large $i\in\N$, we have $w_i\in V$ and thus $g(w_i)\in g(V)\subseteq B(z_k,\epsilon)$.
    Thus, the limit points of the sequence $\{g(w_i)\}_{i\in\N}$ are in $\cls{B(z_k,\epsilon)}$. Consequently, $z_k$ is the only possible limit point of $g(w_i)$, since it is easy to verify that each limit point of $\{g(w_i)\}_{i\in\N}$ is in $f^{-1}(w)$. Hence, $g(w_i)\to z_k\in\cls{U'}$ as $i\to+\infty$. We define $\tg(w)\=\lim_{i\to+\infty} g(w_i)=z_k$. Note that, by the above construction, $\tg(w)$ does not depend on the choice of the sequence $\{w_i\}_{i\in\N}$.
    
    It is easy to obtain from the construction of $\tg$ that $\tg$ is continuous, $f(\tg(y))=y$ for each $y\in Y$, and $\tg(f(x))=x$ for each $x\in \cls{U'}$. Hence, $\tg\:Y\to \cls{U'}$ is a homeomorphism for which $\tg=\bigl(f|_{\cls{U'}}\bigr)^{-1}$.
	
	Denote 
    \begin{equation*}
     \fD\=\bigl\{\cls{U'}:U'\text{ is a connected component of }f^{-1}(\stdCellint{X_0})\text{ for some }X_0\in\cD_0^{[n]}\bigr\}.   
    \end{equation*}
    By the construction of $\fD$, for each $X\in\fD$, $f|_X\:X\to f(X)$ is homeomorphic and $f(X)\in\cD_0^{[n]}$, and thus by Corollary~\ref{c: cellular: cellular homeo}, $(f|_X)^*\bigl(\cD_0|_{f(X)}\bigr)$ is a well-defined cell decomposition of $X$.
	Define
	\begin{equation*}
	\cD_1\=\bigcup_{X\in\fD}(f|_{X})^*\bigl(\cD_0|_{f(X)}\bigr).
	\end{equation*}

    \smallskip

    {\it Step 2:}
	We verify that $\cD_1$ is a cell decomposition of $\mfd^n$ for which $f$ is $(\cD_1,\cD_0)$-cellular. 

    Since, for each $Y\in\cD_0^{[n]}$, the set of connected components of $f^{-1}(\stdCellint{Y})$ is finite (cf.~Lemma~\ref{l: open+closed => f(tU)=U} and Remark~\ref{r: bc on mfd is cls fbc}), $\fD$ is finite. Then
	\begin{enumerate}
    \smallskip
		\item $\cD_1$ is a finite collection of cells, since $\fD$ is finite and, for each $X\in\fD$, $\cD_0|_{f(X)}$ is finite (cf.~Lemma~\ref{l: cell decomp: compact, and Mfd}~(i)); and 
		\smallskip
		\item for each $c\in\cD_1$, the cell-boundary $\cellbound c$ is a union of cells in $\cD_1$, since for each $X\in\fD$, $\cD_0|_{f(X)}$ is a cell decomposition.
	\end{enumerate}
	
	It is also straightforward to show $\bigcup\cD_1=\bigcup\fD=\mfd^n$: Suppose $E\=\mfd^n\smallsetminus\bigcup\fD\neq\emptyset$. Since $\fD$ is a finite collection of closed sets, $E$ is open, and thus $f(E)$ is open since $f$ is open. Then we can find $Y\in\cD_0^{[n]}$ such that $f(E)\cap\stdCellint{Y}\neq\emptyset$ (cf.~Lemma~\ref{l: cell decomp: compact, and Mfd}~(ii) and Lemma~\ref{l: properties of cells}~(i)), and thus $E\cap f^{-1}(\stdCellint{Y})$, which contradicts $E=\mfd^n\smallsetminus\bigcup\fD$. 
	
	The above discussion verifies conditions~(i), (iii), and~(iv) in Definition~\ref{d: cell decomposition}. Now it suffices to verify condition (ii), that is, if $c,\,c'\in\cD_1$ and $\brCellint{c}\cap\brCellint{c'}\neq\emptyset$, then $c=c'$.
	
	Consider the following collection of cells in $\mfd^n$:
	\begin{equation*}
	\begin{aligned}
		\cD'_1=\{\tc\subseteq\mfd^n:f(\tc)\in\cD_0, \, f|_{\tc}\:\tc\to f(\tc)\text{ is a homeomorphism}\}.
	\end{aligned}
	\end{equation*}
	Then $\cD_1\subseteq\cD'_1$ by the construction of $\cD_1$. It suffices to prove the following claim:
	
	\smallskip
	
	{\it Claim.} If $c,\,c'\in\cD'_1$ and $\brCellint{c}\cap\brCellint{c'}\neq\emptyset$, then $c=c'$.
	
	\smallskip
	
	{\it Proof of Claim.} Let $c,\,c'\in\cD'_1$ satisfy $\brCellint{c}\cap\brCellint{c'}\neq\emptyset$. By the definition of $\cD'_1$, $f(c)$ and $f(c')$ are cells in $\cD_0$ with $\brCellint{f(c)}\cap\brCellint{f(c')}\neq\emptyset$. Thus, $f(c)=f(c')$, and we denote $\tau\=f(c)=f(c')\in\cD_0$.
    In what follows, we show that $\brCellint{c}\cap\brCellint{c'}$ is a (relatively) open and closed subset of $\stdCellint{c}$.
        
	First, we show that $\brCellint{c}\cap\brCellint{c'}$ is relatively closed. Consider $x\in\brCellint{c}\smallsetminus\brCellint{c'}$. If $x\in\cellbound c'$, then $f(x)\in\brCellint{\tau}\cap\cellbound \tau$, which cannot be true. Hence, $x\in\mfd^n\smallsetminus c'$. By Lemma~\ref{l: properties of cells}, $c'$ is closed, 
    and there exists an open neighborhood $U$ of $x$ such that $U\cap c'=\emptyset$. It follows that $\brCellint{c}\smallsetminus\brCellint{c'}$ is an (relatively) open subset of $\stdCellint{c}$, and $\brCellint{c}\cap\brCellint{c'}$ is a (relatively) closed subset of $\stdCellint{c}$. 
	
	Next, we show that $\brCellint{c}\cap\brCellint{c'}$ is relatively open. More precisely, for each $x\in\brCellint{c}\cap\brCellint{c'}$, we give an open neighborhood $V_x$ of $x$ such that $V_x\cap\stdCellint{c}\subseteq\stdCellint{c'}$. We argue by contradiction and suppose that there exist $x\in\brCellint{c}\cap\brCellint{c'}$ and a sequence $\{x_i\}_{i\in\N}$ in $\brCellint{c}\smallsetminus\brCellint{c'}$ such that $x_i\to x$ as $i\to+\infty$. Denote $y_i\=f(x_i)\in\stdCellint{\tau}$ and $x'_i\=(f|_{c'})^{-1}(y_i)\in\stdCellint{c'}$ for each $i\in\N$. Then
    \begin{enumerate}
        \item $y_i\to f(x)$ and $x'_i\to x$ as $i\to+\infty$; and
        \smallskip
        \item $x_i\neq x'_i$ and $f(x_i)=f(x'_i)=y_i$ for each $i\in\N$.
    \end{enumerate}
	It follows that $x\in B_f$ and $f(x)\in P_f$. 
    
    Then we consider the CPCF data $(\cB,\cP;\cP_0,\cD_0)$; cf.~Definition~\ref{def:CPCF}. Since $\cB$ and $\cP_0$ are cell decompositions of $B_f$ and $P_f$, respectively, by Lemma~\ref{l: properties of cell decompositions}~(ii), there are $\sigma\in\cB$ and $\tau'\in\cP_0$ such that 
    $x\in\stdCellint{\sigma}$ and $f(x)\in\stdCellint{\tau'}$.
    Since $\tau'\in\cP_0\subseteq\cD_0$ and $f(x)\in\brCellint{\tau'}\cap\brCellint{\tau}$, we have $\tau=\tau'\in\cP_0$. Since $f|_{B_f}\:B_f\to P_f$ is $(\cB,\cP)$-cellular (cf.~Definition~\ref{def:CPCF}~(b)), $f(\sigma)\in\cP$ and $f(x)\in\brCellint{f(\sigma)}\cap\brCellint{\tau}$. Consequently,
    since $\cP_0$ refines $\cP$ (cf.~Definition~\ref{def:CPCF}~(a)) and by Corollary~\ref{c: refinement: inte(c1) meet inte(c0)}, $\stdCellint{\tau}\subseteq \brCellint{f(\sigma)}$. Thus,
    $f(\stdCellint{c})= f(\stdCellint{c'})=\stdCellint{\tau}\subseteq \brCellint{f(\sigma)}=f(\stdCellint{\sigma}).$
	Since $\ind(\cdot,f)$ is constant on $\stdCellint{\sigma}$ (cf.~Definition~\ref{def:CPCF}~(b)), by Lemma~\ref{l: constant deg on int(c)=>neighborhood not map to f(c)}, there exists an open neighborhood $U_x$ of $x$ for which 
	\begin{equation*}
	U_x\cap\stdCellint{c}\subseteq U_x\cap f^{-1}(f(\stdCellint{\sigma}))\subseteq\stdCellint{\sigma}\quad\text{and}\quad U_x\cap\stdCellint{c'}\subseteq\stdCellint{\sigma}.
	\end{equation*} 

    Since $x_i\to x$ and $x'_i\to x$ as $i\to+\infty$, when $i\in\N$ is sufficiently large, we have $x_i\in U_x\cap\stdCellint{c}$ and $x'_i\in U_x\cap\stdCellint{c'}$, in which case $x_i,\,x'_i\in\stdCellint{\sigma}$. However, this contradicts that $f|_\sigma$ is a homeomorphism since, for all $i\in\N$, we have $x_i\neq x'_i$ and $f(x_i)=f(x'_i)=y_i$.
    
    Now it is verified that there exists an open neighborhood $V_x$ of $x$ such that $V_x\cap\stdCellint{c}\subseteq\stdCellint{c'}$.

	Since $x\in\brCellint{c}\cap\brCellint{c'}$ is arbitrary, $\brCellint{c}\cap\brCellint{c'}$ is a (relatively) open subset of $\stdCellint{c}$.
	Since $\brCellint{c}\cap\brCellint{c'}$ is a non-empty, (relatively) open, and (relatively) closed subset of $\stdCellint{c}$, we have $\stdCellint{c}\subseteq\stdCellint{c'}$. Similarly, we have $\stdCellint{c'}\subseteq\stdCellint{c}$, and thus $c=c'$.
	
	The proof of Claim is now complete.
	
	\smallskip
	
	Now it is verified that $\cD_1$ is a cell decomposition. From the construction of $\cD_1$, one can see that $f$ is  $(\cD_1,\cD_0)$-cellular.

    \smallskip

    {\it Step 3:} 
    Now we construct $\cB_1,\,\cP_1$ and verify that $f$ is CPCF with data $(\cB_1,\cP_0;\cP_1,\cD_1)$.
	By Proposition~\ref{p: B_f: is a sub complex}, there exists a subcomplex $\cB_1$ of $\cD_1$ such that $B_f=\abs{\cB_1}$. It follows from Lemma~\ref{l: multi: i(x,f)=max card U(Y)} that for each $c\in\cB_1$, the map $\ind(\cdot,f)|_{\stdCellint{c}}\:\stdCellint{c}\to\N$ is a constant depending on $c$. Since $\cP_0$ is a subcomplex of $\cD_0$ with $P_f=\abs{\cP_0}$, by Lemma~\ref{l: cell decomp: restrict}~(iii), $\cP_0=\cD_0|_{P_f}$.
    Consequently, for each $c\in\cB_1$, since $f(c)\subseteq P_f$, $f(c)\in\cD_0|_{P_f}=\cP_0$. Thus, $f|_{B_f}$ is $(\cB_1,\cP_0)$-cellular.
     
     By Lemma~\ref{l: cellular Markov: pullbcak of ess also ess}, there exists a cell decomposition $\cP_1$ of $P_f$ such that $\cP_1$ refines $\cP_0$, and $f|_{P_f}$ is $(\cP_1,\cP_0)$-cellular. By the definition of $\cD'_1$, we have $\cP_1\subseteq \cD'_1$.

	Moreover, we show $\cP_1\subseteq\cD_1$ by showing that $\cD_1$ is identical to $\cD'_1$. Indeed, by the definition of $\cD'_1$ and $\cD_1$, $\cD_1\subseteq\cD'_1$. On the other hand, for each $c'\in\cD'_1$, since $\cD_1$ is a cell decomposition, by Lemma~\ref{l: properties of cell decompositions}~(ii), there exists a cell $c\in\cD_1\subseteq\cD'_1$ such that $\brCellint{c'}\cap\brCellint{c}\neq\emptyset$, and thus $c'=c\in\cD_1$ by the above claim in Step~2. Thus, $\cD'_1\subseteq\cD_1$.
    
	The above discussion concludes that $f$ is a CPCF branched cover with data $(\cB_1,\cP_0;\cP_1,\cD_1)$.
\end{proof}

\section{Quasisymmetric uniformization of visual metrics}\label{sct: QS uniformization of visual}

In this section, we discuss the quasisymmetric uniformization of visual metrics for expanding {\itCel} branched covers, in terms of uniform expanding BQS maps (cf.~Definition~\ref{d: UEBQS}). More precisely, we establish \ref{item:UQR-1}$\Leftrightarrow$\ref{item:UQR-2} in Theorem~\ref{tx: QS-UEBQS-Lattes-UQR}, as the following theorem.

\begin{theorem}\label{t: CBC: QS sph <=> UBQS}
	Let $f\:\mfd^n\to\mfd^n$, $n\geq 2$, be an expanding {\itCel} branched cover on a {\clcnR} $n$-manifold $\mfd^n$, and $\varrho$ be a visual metric for $f$. Let $\dg$ be the distance function on $\mfd^n$. Then the identity map $\id\:(\mfd^n,\varrho)\to(\mfd^n,\dg)$ is quasisymmetric if and only if $f\:(\mfd^n,\dg)\to(\mfd^n,\dg)$ is a uniform expanding BQS map.
\end{theorem}

\subsection{Visual metrics}\label{subsct: visual metrics}
For expanding CPCF branched covers, there exists a natural class of metrics called \emph{visual metrics}, which is closely related to the visual metrics on the boundaries of Gromov hyperbolic spaces, and the very similar notions in \cite{BM17,HP09}.
Here we briefly give some definitions with necessary remarks, and postpone more detailed discussions to Appendix~\ref{Ap: visual metric}.

Visual metrics are defined via the following function of ``separation level".

Let $f\:\mfd^n\to\mfd^n$ be an {\itCel} map and $\{\cD_l\}_{l\in\N_0}$ be a {\celSeq} of $f$. We define $m_{f,\cD_0}\:\mfd^n\times\mfd^n\to\N_0\cup\{+\infty\}$ by
	\begin{equation}\label{e: the constant m(x,y)}
		\begin{aligned}
			m_{f,\cD_0}(x,y)\=\sup\bigl\{l\in\N_0: 
			\text{there exist }X,\,Y\in\cD_l^{\top}
			\text{ with }
			x\in X,\,y\in Y,\text{ and } X\cap Y\neq\emptyset\bigr\},
		\end{aligned}
	\end{equation}
where we use the convention $\sup\emptyset=0$ in (\ref{e: the constant m(x,y)}).

Since $\{\cD_l\}_{l\in\N_0}$ is determined by $\cD_0$ and $f$ (cf.~Lemma~\ref{l: cellular: pull back ess decomp}), it suffices to indicate $\cD_0$ in the notation $m_{f,\cD_0}$. If $f$ and $\cD_0$ are clear from the context, we write $m(x,y)\=m_{f,\cD_0}(x,y)$.
\begin{rem}\label{r: remarks on m(x,y)}
(a) If, for all $m\in\N$, $X,\,Y\in\cD_m^{[n]}$ with $x\in X$, $y\in Y$, we have $X\cap Y=\emptyset$, then $m_{f,\cD_0}(x,y)=0$.
(b) When the {\celSeq} $\{\cD_m\}_{m\in\N_0}$ of $f$ is expanding,  $m_{f,\cD_0}(x,y)=+\infty$ if and only if $x=y$. 
\end{rem}

\begin{definition}[Visual metrics]\label{d: visual metrics}
	Let $f\:\mfd^n\to\mfd^n$ be an expanding {\itCel} map on a {\clcnmtr} $n$-manifold. A metric $\varrho$ on $\mfd^n$ is a \defn{visual metric} for $f$ if there exists a {\celSeq} $\{\cD_m\}_{m\in\N_0}$ of $f$, and constants $C>1$ and $\Lambda>1$ such that 
	\begin{equation}\label{e: visual metric}
		C^{-1}\Lambda^{-m_{f,\cD_0}(x,y)} \leq \varrho(x,y) \leq  C\Lambda^{-m_{f,\cD_0}(x,y)}
	\end{equation}
	for all $x,y\in\mfd^n$, where $m_{f,\cD_0}(x,y)$ is defined by (\ref{e: the constant m(x,y)}). We call the constant $\Lambda$ the \defn{expansion factor} for $\varrho$ and the metric space $(\mfd^n,\varrho)$ the \defn{visual space}.
\end{definition}

\begin{remark}
Here we state some facts about visual metrics without proof. For a detailed discussion, see Appendix~\ref{Ap: existence of visual metric}. Let $f\:\mfd^n\to\mfd^n$ be an expanding {\itCel} branched cover, then
\begin{enumerate}[label=(\roman*),font=\rm]
	\smallskip
    \item visual metrics for $f$ exist and induce the given topology on $\mfd^n$;
    \smallskip
    \item visual metrics are well defined, i.e.,~independent of the choice of {\celSeq}s of $f$;
    \smallskip
    \item the expansion factor of a given visual metric is unique.
\end{enumerate}
\end{remark}

\subsection{Cellular neighborhoods}\label{subsct: cellular neighborhood}
We introduce a class of neighborhoods $U_\cD^i(G)$, called \emph{cellular neighborhoods}, constructed from cell decompositions as a key tool in the current section and Section~\ref{sct: Rigidity of UQR}.

Let $\cD$ be a cell decomposition of a {\clcnR} $n$-manifold $\mfd^n$. 
For each subset $G\subseteq\mfd^n$, we denote
\begin{equation}\label{e: U^1 and U^2}
    \begin{aligned}
		\cU_{\cD}^1(G)&\=\bigl\{Y\in\cD^{[n]}:Y\cap G\neq\emptyset\bigr\}, &U_\cD^1(G)&\=\stdInt{\Absbig{\cU_\cD^1(G)}},\\
		\cU^2_{\cD}(G)&\=\cU^1_\cD\bigl(\Absbig{\cU^1_\cD(G)}\bigr),
        &U_\cD^2(G)&\=\stdInt{\Absbig{\cU_\cD^2(G)}}.
    \end{aligned}
\end{equation}

With the cell decomposition $\cD$ understood, we often omit $\cD$ and denote $\cU^i(G)\=\cU_\cD^i(G)$ and $U^i(G)\=U_\cD^i(G)$ for $i\in\{1,\,2\}$.
With a {\celSeq} $\{\cD_m\}_{m\in\N_0}$ of an {\itCel} map $f\:\mfd^n\to\mfd^n$ understood, we denote, for $m\in\N_0$, $G\subseteq\mfd^n$, and $i\in\{1,\,2\}$,
\begin{equation}\label{e: U^i_m(G)}
	\begin{aligned}
		\cU_{m}^i(G)\=\cU_{\cD_m}^i(G)\quad \text{ and } \quad
		U_{m}^i(G)\=U_{\cD_m}^i(G).
	\end{aligned}
\end{equation}

\begin{remark}
(i) It is a natural idea to recursively define $\cU^j(G)$ for $j\in\{3,\,4,\,\dots\}$. However, in this article, it suffices to consider $\cU^1(G)$ and $\cU^2(G)$.
(ii) Although for a point $x\in\mfd^n$, the open sets $U^1(\{x\})$ and $\flower(x)$ (defined in (\ref{e: Fl(x)})) look similar, they generally do not coincide.
(iii) In this article, we usually consider the cases where $G$ is a single point, or $G\=X$ for some $X\in\cD^{[n]}$. 
\end{remark}

We now give some elementary properties of these sets.
\begin{lemma}\label{l: cellular neighborhood: int(cX) - cls, comp, bd}
	Let $\cD$ be a cell decomposition of a {\clcnR} $n$-manifold $\mfd^n$, $\cX$ be a subset of $\cD^{[n]}$, and denote $U\=\stdInt{\abs{\cX}}$. Then 
    \begin{enumerate}[label=(\roman*),font=\rm]
    	\smallskip
        \item $\cls{U}=\abs{\cX}$,
        \smallskip
        \item $\mfd^n\smallsetminus U=\Absbig{\cD^{[n]}\smallsetminus\cX}$, and
        \smallskip
        \item $\partial U=\bigcup\{c\in\cD:\partial U\cap\stdCellint{c}\neq\emptyset\}=\bigcup\{c\in\cD:c\subseteq \partial U\}$.
    \end{enumerate}
\end{lemma}
\begin{proof}
    (i) For each $Y\in\cX$, since $\stdCellint{Y}$ is open (cf.~Lemma~\ref{l: properties of cells}~(ii)), $\stdCellint{Y}\subseteq U$ and $Y=\cls{\stdCellint{Y}}\subseteq\cls{U}$. Thus $\abs{\cX}\subseteq\cls{U}$. By Lemma~\ref{l: cell decomp: compact, and Mfd}~(i), $\cX$ is finite and thus $\abs{\cX}$ is closed. Then $\cls{U}=\cls{\stdInt{\abs{\cX}}}\subseteq\abs{\cX}$, and thus $\abs{\cX}=\cls{U}$.
    
	\smallskip
    
    (ii) For each $Y\in\cD^{\top}\smallsetminus\cX$, by Lemma~\ref{l: properties of cell decompositions}~(vi) and (v), we have $\stdCellint{Y}\subseteq\mfd^n\smallsetminus\abs{\cX}$, and thus $Y=\cls{\stdCellint{Y}}\subseteq\mfd^n\smallsetminus U$.
	On the other hand, assume $x\in\mfd^n\smallsetminus U$. Since $\flower(x)$ is an open neighborhood of $x\in\mfd^n\smallsetminus U$ (cf.~Corollary~\ref{c: x-flower structure}~(i)), we have $\flower(x)\nsubseteq\abs{\cX}$. Thus, by Corollary~\ref{c: x-flower structure}~(ii), there exists $Y\in\cD^{[n]}\smallsetminus\cX$ with $x\in Y$. 

    \smallskip

    (iii) It follows from Lemma~\ref{l: properties of cell decompositions}~(ii) that $\partial U\subseteq\bigcup\{c\in\cD:\partial U\cap\stdCellint{c}\neq\emptyset\}$.
    It is also clear that $\bigcup\{c\in\cD:c\subseteq\partial U\}\subseteq\partial U$. Now it suffices to show $\bigcup\{c\in\cD:\partial U\cap\stdCellint{c}\neq\emptyset\}\subseteq\bigcup\{c\in\cD:c\subseteq\partial U\}$.
	
	Let $c\in\cD$ satisfy $\partial U\cap\stdCellint{c}\neq\emptyset$. By~(i) and~(ii),
	$\partial U=\cls{U}\smallsetminus U=\abs{\cX}\cap\Absbig{\cD^{[n]}\smallsetminus\cX}$. Thus,
	there exist $Y\in\cX$ and $Y'\in\cD^{[n]}\smallsetminus\cX$ such that $Y\cap Y'\cap\stdCellint{c}\neq\emptyset$. Then it follows from Lemma~\ref{l: properties of cell decompositions}~(v) that $c\subseteq Y\cap Y'\subseteq\abs{\cX}\cap\Absbig{\cD^{[n]}\smallsetminus\cX}=\partial U$. 
\end{proof}

\begin{lemma}\label{l: cell Nbhd: U^i(G)}
    Let $\cD$ be a cell decomposition of a {\clcnR} $n$-manifold $(\mfd^n,\dg)$. Then for each subset $G\subseteq\mfd^n$, the following statements are true
    \begin{enumerate}[label=(\roman*),font=\rm]
    	\smallskip
		\item $G\subseteq U^1(G)\subseteq\cls{U^1(G)}=\Absbig{\cU^1(G)}\subseteq U^2(G)$.
		\smallskip
		\item If $G$ is path-connected, then $\Absbig{\cU^i(G)}$ and $U^i(G)$, $i\in\{1,\,2\}$, are path-connected.
		\smallskip
		\item $\dist_{\dg}(G,Z) \geq \dist_{\dg}\bigl(G,\partial U^1(G)\bigr)$ for each $Z\in\cD^{\top}\smallsetminus\cU^1(G)$. 
	\end{enumerate}
\end{lemma}
\begin{proof}
	Let $G\subseteq\mfd^n$ be arbitrary.
	
	\smallskip
	
	(i) We first show $G\subseteq U^1(G)$. Fix arbitrary $x\in G$. Then by the definition of $\cU^1(G)$ and Corollary~\ref{c: x-flower structure}~(ii), $\cls{\flower(x)}\subseteq\Absbig{\cU^1(G)}$, and thus, since $\flower(x)$ is open (cf.~Corollary~\ref{c: x-flower structure}~(i)), we have $x\in\flower(x)\subseteq\stdInt{\Absbig{\cU^1(G)}}=U^1(G)$. So $G\subseteq U^1(G)$. By Lemma~\ref{l: cellular neighborhood: int(cX) - cls, comp, bd}~(i) and the above discussion, $\cls{U^1(G)}=\Absbig{\cU^1(G)}\subseteq U^1\bigl(\Absbig{\cU^1(G)}\bigr)=U^2(G)$ (cf.~(\ref{e: U^1 and U^2})).
	
	\smallskip
	
	(ii) Suppose $G$ is path-connected. It is clear from (\ref{e: U^1 and U^2}) and the path-connectedness of $G$ that $\Absbig{\cU^i(G)}$, $i\in\{1,\,2\}$, are path-connected. Now consider $U^i(G)$, $i\in\{1,\,2\}$. The definition (cf.~(\ref{e: U^1 and U^2})) yields $U^1(G)=\bigcup\bigl\{Y\cap U^1(G):Y\in\cU^1(G)\bigr\}$. Thus, since each $Y\in\cU^1(G)$ meets $G$ (cf.~(\ref{e: U^1 and U^2})), and the path-connected set $G$ is contained in $U^1(G)$ (by~(i)), it suffices to show that for each $Y\in\cU^1(G)$, $Y\cap U^1(G)$ is path-connected. Fix $Y\in\cU^1(G)$. Since $\stdCellint{Y}$ is open (cf.~Lemma~\ref{l: properties of cells}~(ii)),  we have $\stdCellint{Y}\subseteq U^1(G)$, and thus $Y\cap U^1(G)=(U^1(G)\cap\cellbound Y)\cup\stdCellint{Y}$. Consequently, $Y\cap U^1(G)$ is path-connected since $Y$ is an $n$-dimensional cell. 
    Hence, $U^1(G)$ is path-connected. Since $\Absbig{\cU^1(G)}$ is path-connected, by the above discussion, $U^2(G)=U^1\bigl(\Absbig{\cU^1(G)}\bigr)$ is path-connected.
	
	\smallskip
	
	(iii) Consider $Z\in\cD^{\top}$ with $Z\cap G=\emptyset$. By Lemma~\ref{l: cellular neighborhood: int(cX) - cls, comp, bd}~(ii), 
	$Z\subseteq\mfd^n\smallsetminus U^1(G)$, and thus $\dist_{\dg}(G,Z) \geq \dist_{\dg}\bigl(G,\mfd^n\smallsetminus U^1(G)\bigr)=\dist_{\dg}\bigl(G,\partial U^1(G)\bigr)$.
\end{proof}

We usually apply the following special case of Lemma~\ref{l: cell Nbhd: U^i(G)}, where $G$ is a chamber.
\begin{lemma}\label{l: cel Nbhd: U^1(X)}
	Let $\cD$ be a cell decomposition of a {\clcnR} $n$-manifold $(\mfd^n,\dg)$.
	Then for each $X\in\cD^{\top}$, the following statements are true: 
	\begin{enumerate}[label=(\roman*),font=\rm]
    	\smallskip
		\item $X\subseteq U^1(X)\subseteq\cls{U^1(X)}=\Absbig{\cU^1(X)}\subseteq U^2(X)$.
		\smallskip
		\item $\Absbig{\cU^i(X)}$ and $U^i(X)$, $i\in\{1,\,2\}$, are path-connected.
		\smallskip
		\item $\dist_{\dg}(X,Z) \geq \dist_{\dg}\bigl(X,\partial U^1(X)\bigr)$ for each $Z\in\cD^{\top}\smallsetminus\cU^1(X)$. 
	\end{enumerate}
\end{lemma}

\begin{lemma}\label{l: cel Nbhd: fwd invariance}
    Let $\{\cD_m\}_{m\in\N_0}$ be a {\celSeq} of an {\itCel} branched cover $f\:\mfd^n\to\mfd^n$ on a {\clcnR} $n$-manifold. Then for all $G\subseteq\mfd^n$, $m\in\N_0$, $k\in\N$, and $i\in\{1,\,2\}$, $f^k\bigl(\Absbig{\cU^i_{m+k}(G)}\bigr)\subseteq \Absbig{\cU^i_m\bigl(f^k(G)\bigr)}$ and
    $f^k\bigl(U^i_{m+k}(G)\bigr)\subseteq U^i_m\bigl(f^k(G)\bigr)$.
\end{lemma}
\begin{proof}
    Fix $G\subseteq\mfd^n$, $m\in\N_0$, $k\in\N$, and $i\in\{1,\,2\}$. 
    It is easy to see from the construction of $\cU^i_{m+k}(G)$ and $\cU^i_m\bigl(f^k(G)\bigr)$ in (\ref{e: U^1 and U^2}) and the definition of {\itCel} maps that $f^k\bigl(\Absbig{\cU^i_{m+k}(G)}\bigr)\subseteq \Absbig{\cU^i_m\bigl(f^k(G)\bigr)}$. Thus, since $f$ is open, $f^k\bigl(U^i_{m+k}(G)\bigr)=f^k\bigl(\stdInt{\Absbig{\cU^i_{m+k}(G)}}\bigr)\subseteq\stdInt{\Absbig{\cU^i_m\bigl(f^k(G)\bigr)}} =U^i_m\bigl(f^k(G)\bigr)$.
\end{proof}

\subsection{Quasisymmetric uniformization}
Now we show that an expanding {\itCel} branched cover on a {\clcnR} $n$-manifold $\mfd^n$, $n\geq 2$, is a uniform expanding BQS map with respect to its visual metrics. Recall that, in Subsection~\ref{subsct: cell decomp}, we denote by $\coverF(\cD)$ the set of flowers in a cell decomposition $\cD$. Also recall from Proposition~\ref{p: p-flowers: invariance}~(ii) that for a {\celSeq} $\{\cD_m\}_{m\in\N_0}$ of an {\itCel} branched cover $f\:\mfd^n\to\mfd^n$, we have $\coverF(\cD_m)=(f^m)^*\coverF(\cD_0)$ for each $m\in\N$.
\begin{lemma}\label{l: cellular BQS: UEBQS on visual sphere}
	Let $f\:\mfd^n\to\mfd^n$, $n\geq 2$, be an expanding {\itCel} branched cover on a {\clcnR} $n$-manifold, $\{\cD_m\}_{m\in\N_0}$ be a {\celSeq} of $f$, and $\varrho$ be a visual metric for $f$ on $\mfd^n$. Then there exists a homeomorphism $\eta \: [0,+\infty) \to [0,+\infty)$ such that $f\:(\mfd^n,\varrho)\to(\mfd^n,\varrho)$ is a uniform expanding BQS map with data $(\coverF(\cD_0),\eta)$.
\end{lemma}
\begin{proof}
	Denote $m(x,y)\=m_{f,\cD_0}(x,y)$ for all $x,\,y\in\mfd^n$. Let $\Lambda>1$ be the expansion factor for $\varrho$, and let $C>1$ be a constant for which (\ref{e: visual metric})
	holds for all $x,\,y\in\mfd^n$. 
	By Lemma~\ref{l: visual metric: metric properties of cells}, there exists a constant $C'>1$ such that for all $m\in\N$ and $X\in\cD_m^{\top}$,
	\begin{equation}\label{e: BQS wrt visual metric: diam X}
		C'^{-1}\Lambda^{-m} \leq \diam X \leq  C'\Lambda^{-m},
	\end{equation}
    where we may also assume that $C'>\diam\mfd^n$.
    
    Fix $k_0\in\N$ such that $8C'\Lambda^{-k_0}$ is strictly less than the Lebesgue number of the open cover $\coverF(\cD_0)$.
	
	For each $m\in\N$ and each $x\in\mfd^n$, we denote
	\begin{equation*}
	\begin{aligned}
		\tW_m(x)\=\{y\in\mfd^n: m(x,y)\geq m\}.
	\end{aligned}
    \end{equation*}
    By (\ref{e: visual metric}), the definition of $m(x,y)$ (cf.~(\ref{e: the constant m(x,y)})), and (\ref{e: BQS wrt visual metric: diam X}),
    \begin{equation}\label{e: B() < tWm(x) < B()}
        B\bigl( x,C^{-1}\Lambda^{-m} \bigr)\subseteq\tW_m(x)\subseteq B(x, 4C'\Lambda^{-m})
    \end{equation}
	
	Fix arbitrary $m\in\N$ and an $m$-vertex $p$.
	Consider a continuum $A\subseteq\flower_m(p)$. We aim to give estimates of $A$ and $f^m(A)$. Since $A\subseteq\flower_m(p)$ and $f$ is expanding, we can find an integer $l\geq m$ and an $l$-flower that is an $(f,\coverF(\cD_0))$-approximation of $A$ (see Definition~\ref{d: approximate} and Proposition~\ref{p: p-flowers: invariance}~(ii)).
	
	First, we show $\diam A \geq  C^{-1}\Lambda^{-(l+k_0+1)}$. We argue by contradiction and assume $\diam A< C^{-1}\Lambda^{-(l+k_0+1)}$. Fix $x_0\in A$. Then by (\ref{e: B() < tWm(x) < B()}),
    $A\subseteq\tW_{l+k_0+1}(x_0)$, and thus, by Lemma~\ref{l: m(x,y) property: iteration} and (\ref{e: B() < tWm(x) < B()}), $f^{l+1}(A)\subseteq \tW_{k_0}\bigl(f^{l+1}(x_0)\bigr)\subseteq B\bigl(f^{l+1}(x_0),4C'\Lambda^{-k_0}\bigr)$. So it follows from the choice of $k_0$ that $f^{l+1}(A)\subseteq B\bigl(f^{l+1}(x_0),4C'\Lambda^{-k_0}\bigr))$ is contained in some $0$-flower, and thus, by Corollary~\ref{c: CBC: connected set in flower}~(i), $A$ is contained in some $(l+1)$-flower. This contradicts the choice of $l$. Thus, $\diam A \geq  C^{-1}\Lambda^{-(l+k_0+1)}$. 
    
    On the other hand, since $A$ is contained in some $l$-flower, it follows from Corollary~\ref{c: x-flower structure}~(ii) and (\ref{e: BQS wrt visual metric: diam X}) that $\diam A \leq  2C'\Lambda^{-l}$. Hence, we have
	$C^{-1}\Lambda^{-(l+k_0+1)} \leq \diam A \leq  2C'\Lambda^{-l}$.

    By Remark~\ref{r: approximate: iterate}, there exists an $(l-m)$-flower that is an $(f,\coverF(\cD_0))$-approximation of $f^m(A)$. Same arguments as above yield
	$C^{-1}\Lambda^{-(l+k_0+1)+m} \leq \diam(f^m(A)) \leq  2C'\Lambda^{-l+m}$.
	It follows that
	\begin{equation*}
        \bigl( 2C'C\Lambda^{k_0+1} \bigr)^{-1}\Lambda^{m} \leq \diam(f^m(A))/\diam A \leq 2C'C\Lambda^{k_0+1}\Lambda^m.
	\end{equation*}
	Applying the above inequality for continua $E$ and $E'$ contained in $\flower_m(p)$, we have 
	\begin{equation*}
	 \diam(f^m(E)) / \diam(f^m(E'))  \leq   \bigl(2C'C\Lambda^{k_0+1}\bigr)^2 \diam E / \diam E' .
    \end{equation*}
	Define $\eta\:[0,+\infty)\to[0,+\infty),\,t\mapsto\bigl(2C'C\Lambda^{k_0+1}\bigr)^2t$, then the proof is complete.
\end{proof}

Since the BQS property is preserved under a quasisymmetric conjugacy, one implication in Theorem~\ref{t: CBC: QS sph <=> UBQS} follows immediately.

\begin{cor}\label{c: cellular BQS: QS equiv => flower UBQS}
	Let $f\:\mfd^n\to\mfd^n$, $n\geq2$, be an expanding {\itCel} branched cover on a {\clcnR} $n$-manifold, $\{\cD_m\}_{m\in\N_0}$ be a {\celSeq} of $f$, and $\varrho$ a visual metric for $f$. Let $\dg$ be the Riemannian distance on $\mfd^n$. If the identity map $\id\:(\mfd^n,\varrho)\to(\mfd^n,\dg)$ is a quasisymmetry, then there exists a homeomorphism $\eta\:[0,+\infty)\to[0,+\infty)$ such that $f\:(\mfd^n,\dg)\to(\mfd^n,\dg)$ is a uniform expanding BQS map with data $(\coverF(\cD_0),\eta)$.
\end{cor}
\begin{proof}
	Suppose that the identity map $\id\:(\mfd^n,\varrho)\to(\mfd^n,\dg)$ is a quasisymmetry. By Lemma~\ref{l: cellular BQS: UEBQS on visual sphere}, for each $m\in\N_0$ and each $U\in\coverF(\cD_m)$, the map $f^m|_{U}$ is $\xi$-BQS (with respect to $\varrho$) for some $\xi$ independent of $m$ and $U$. Then it follows from {\cite[Lemmas~3.1 and~3.2]{LiP19}} that $f^m|_{U}$ is $\eta$-BQS (with respect to $\dg$) for some $\eta$ independent of $m\in\N_0$ and $U\in\coverF(\cD_m)$. 
\end{proof}

In what follows, we consider the converse of Corollary~\ref{c: cellular BQS: QS equiv => flower UBQS}. 
First, we show that if an expanding {\itCel} branched cover $f\:\mfd^n\to\mfd^n$ on a {\clcnR} $n$-manifold is a uniform expanding BQS map with respect to $\dg$, then each {\celSeq} of $f$ satisfies metric properties similar to those of the \emph{quasi-visual approximation} in \cite{BM22}, which allows us to establish the quasisymmetric equivalence between $\dg$ and the visual metrics. See \cite[Chapter~18]{BM17} for a similar discussion on Thurston maps. 

\begin{lemma}\label{l: QS unif: QV approx}
Let $f\:\mfd^n\to\mfd^n$, $n\geq 2$, be an expanding {\itCel} branched cover on a {\clcnR} $n$-manifold equipped with the Riemannian distance $\dg$, and $\{\cD_m\}_{m\in\N_0}$ be a {\celSeq} of $f$. If $f$ is a uniform expanding BQS map with data $(\coverF(\cD_0),\eta)$ for some homeomorphism $\eta \: [0,+\infty) \to [0,+\infty)$, then the following statements are true:
	\begin{enumerate}[label=(\roman*),font=\rm]
    \smallskip
    \item For each $k\in\N_0$, there exist $\alpha_k,\,\beta_k\in(0,+\infty)$ such  that for all $m\in\N_0$, $Z\in\cD_m^{[n]}$, and $W\in\cD_{m+k}^{[n]}$ with $Z\cap W\neq\emptyset$, we have $\alpha_k<\diam Z/\diam W<\beta_k$. Moreover, $\alpha_k\to+\infty$ and $\beta_k\to+\infty$ as $k\to+\infty$.
    \smallskip
    \item There exists $\lambda\in(0,+\infty)$ such that for each $m\in\N$ and all $Z,\,W\in\cD_m^{[n]}$ with $Z\cap W=\emptyset$, we have
	$\dist(Z,W) \geq  \lambda\diam Z$. 
    \smallskip
    \item There exists $\mu\in(1,+\infty)$ such that for all $x,\,y\in\mfd^n$, and each $Z\in\cD_m^{\top}$, where $m\=m_{f,\cD_0}(x,y)$, containing $x$, we have
		$\mu^{-1} \leq {\dg(x,y)}/{\diam Z} \leq  \mu$.
    \end{enumerate}
\end{lemma}
\begin{proof}
    Let $\delta_0$ be the Lebesgue number of the open cover $\coverF(\cD_0)$. Since $f$ is expanding, we may fix $m_0\in\N$ such that $\mesh(\cD_m)<\delta_0/4$ for all $m \geq  m_0$.

\smallskip
    
(i) For each $k\in\N_0$, denote
\begin{equation*}
	\begin{aligned}
        a_k&\=\min\bigl\{\diam X/\diam Y:X\in\cD_m^{\top},\,Y\in\cD_{m+k}^{\top},\,m \leq  m_0\bigr\} \quad\text{ and}\\
        b_k&\=\max\bigl\{\diam X/\diam Y:X\in\cD_m^{\top},\,Y\in\cD_{m+k}^{\top},\,m \leq  m_0\bigr\}.
	\end{aligned}
\end{equation*}
We also denote $\ta_k\=\eta^{-1}(a_k)$ and $\tb_k\=1/\eta^{-1}\bigl(b_k^{-1}\bigr)$ for each $k\in\N_0$.

Fix arbitrary $m,\,k\in\N_0$, $Z\in\cD_m^{\top}$, and $W\in\cD_{m+k}^{\top}$ satisfying $Z\cap W\neq\emptyset$. If $m \leq  m_0$, then $a_k \leq \diam Z/\diam W \leq  b_k$.
Now suppose $m>m_0$ and set $h\=f^{m-m_0}$. Then $\diam(h(Z\cup W))\leqslant 2\mesh(\cD_{m_0}) <\delta_0$, and $h(Z\cup W)$ is contained in a $0$-flower. Thus, by Corollary~\ref{c: CBC: connected set in flower}~(i), there exists $U\in\coverF(\cD_{m-m_0})$ such that $Z\cup W\subseteq U$.
Since $h|_{U}$ is $\eta$-BQS, we have 
    \begin{equation*}
        \eta^{-1}(\diam(h(Z))/\diam(h(W))) \leq \diam Z/\diam W \leq  1/\eta^{-1}(\diam(h(W))/\diam(h(Z))).
    \end{equation*}
Thus, $\ta_k \leq \diam Z/\diam W \leq  \tb_k$. It suffices to set $\alpha_k\=\min\{a_k,\,\ta_k\}$ and $\beta_k\=\max\bigl\{b_k,\,\tb_k\bigr\}$. 
Since $\mesh(\cD_i)\to0$ as $i\to+\infty$, we have $\alpha_i\to+\infty$ and $\beta_i\to+\infty$ as $i\to+\infty$.

\smallskip

(ii) Fix $m\in\N$ and $Z,\,W\in\cD_m^{\top}$ satisfying $Z\cap W=\emptyset$. 
	Set
	\begin{equation*}
	a\=\min\bigl\{\dist(X,Y)/\diam X:X,\,Y\in\cD_m^{\top},\,X\cap Y=\emptyset,\,m \leq  m_0\bigr\}.
\end{equation*}
	If $m \leq  m_0$, then $\dist(Z,W) \geq  a\cdot\diam Z$.
	
	Now suppose $m>m_0$. Denote $h\=f^{m-m_0}$. Consider $\cU_{m}^1(Z)$ and $\cU_{m_0}^1(h(Z))$ defined in (\ref{e: U^i_m(G)}). 
    It is straightforward to check from (\ref{e: U^i_m(G)}) that $\diam\bigl(\Absbig{\cU_{m_0}^1(h(Z))}\bigr) \leq 3\mesh(\cD_{m_0})<\delta_0$,
    which implies $\Absbig{\cU_{m_0}^1(h(Z))}$ is contained in some $0$-flower. Thus, since $\Absbig{\cU_m^1(Z)}$ is a connected set for which $h\bigl(\Absbig{\cU_m^1(Z)}\bigr)\subseteq\Absbig{\cU^1_{m_0}(h(Z))}$ (cf.~Lemmas~\ref{l: cel Nbhd: U^1(X)}~(ii) and~\ref{l: cel Nbhd: fwd invariance}), by Corollary~\ref{c: CBC: connected set in flower}~(i),
	we may assume $\Absbig{\cU_m^1(Z)}\subseteq\flower_{m-m_0}(p)$ for some $(m-m_0)$-vertex $p$.
    
    Let $x\in Z$ and $y\in\partial U_m^1(Z)$ satisfy $\dg(x,y)=\dist\bigl(Z,\partial U_m^1(Z)\bigr)$, and $\gamma$ be a geodesic curve from $x$ to $y$.
	By Lemma~\ref{l: cellular neighborhood: int(cX) - cls, comp, bd}~(ii), $\gamma$ meets both $Z$ and an $m$-chamber disjoint from $Z$, and thus, by Corollary~\ref{c: CBC: connected set in flower}~(iii), $h(\gamma)$ is not contained in any $m_0$-flower. Then $\diam(h(\gamma)) \geq \delta_{m_0}$, where $\delta_{m_0}$ is the Lebesgue number of the open cover $\coverF(\cD_{m_0})$.

    Clearly $\gamma$ and $Z$ are contained in $\Absbig{\cU_m^1(Z)}$ and thus in $\flower_{m-m_0}(p)$.
	Since $h|_{\flower_{m-m_0}(p)}$ is $\eta$-BQS, 
	\begin{equation*}
	 \dg(x,y) / \diam Z
     \geq \eta^{-1} ( \diam(h(\gamma)) / \diam(h(Z)) ) 
     \geq \eta^{-1} ( \delta_{m_0}/ \diam(h(Z)) ).
\end{equation*}
    Then by Lemma~\ref{l: cel Nbhd: U^1(X)}~(iii), for $b\=\min\bigl\{\eta^{-1}(\delta_{m_0}/\diam X):X\in\cD_{m_0}^{\top}\bigr\}$,
	\begin{equation*}
	\dist(Z,W) \geq \dist\bigl(Z,\partial U_m^1(Z)\bigr) \geq  b\diam Z.
    \end{equation*}
	Hence $\lambda\=\min\{a, \, b\}$ is the desired constant.

    \smallskip

    (iii) Let $x,\,y\in\mfd^n$ be arbitrary. Denote $m_0\=m_{f,\cD_0}(x,y)$, and suppose $Z\in\cD_{m_0}^{\top}$ contains $x$.
    Consider the following cases according to Remark~\ref{r: remarks on m(x,y)}~(a):

    \smallskip
    
    {\it Case~1.} For all $l\in\N_0$ and $X,\,Y\in\cD_l^{[n]}$ containing $x$ and $y$, respectively, $X\cap Y=\emptyset$. Then $m_0=0$ and $Z\cap Y=\emptyset$ for all $Y\in\cD_0^{[n]}$ containing $y$. Choose $Y\in\cD_0^{[n]}$ that contains $y$. By (ii),
	\begin{equation*}
	 (1/\mu' ) \diam Z
     \leq \dist(Z,Y) 
     \leq  \dg(x,y) 
     \leq  \mu'\diam Z,
    \end{equation*}
	where $\mu'$ is a constant with $\mu'>\max\bigl\{\lambda^{-1},\, \max \bigl\{ \diam \mfd^n / \diam X : X\in\cD_0^{[n]} \bigr\} \bigr\}$. Clearly $\mu'>1$.

    \smallskip

    {\it Case~2.} There exist $X_0,\,Y_0\in\cD_{m_0}$ satisfying $x\in X_0$, $y\in Y_0$, and $X_0\cap Y_0\neq\emptyset$. Choose $X_1,\,Y_1\in\cD_{m_0+1}$ satisfying $x\in X_1$, $y\in Y_1$. The definition of $m_{f,\cD_0}(x,y)$ implies $X_1\cap Y_1=\emptyset$. 
	First, 
    by~(i) and~(ii), we have
    \begin{equation*}
        \dg(x,y) \geq \dist(X_1,Y_1) \geq  \lambda\diam X_1 \geq  \lambda\beta_1^{-1}\diam Z.
    \end{equation*}
   On the other hand, by~(i), we have 
	\begin{equation*}
        \dg(x,y) \leq \diam X_0+\diam Y_0
         \leq (1+\beta_0)\diam X_0 \leq \beta_0(1+\beta_0)\diam Z.
	\end{equation*}

\smallskip

	Therefore, $\mu\=\max\bigl\{\mu',\,\beta_0(1+\beta_0),\,\lambda^{-1}\beta_1\bigr\}$ is the desired constant.
\end{proof}

We are now ready to prove the other direction in Theorem~\ref{t: CBC: QS sph <=> UBQS}.
\begin{theorem}\label{t: cellular BQS: flowerBQS => QS equivalence}
	Let $f\:\mfd^n\to\mfd^n$, $n\geq 2$, be an expanding {\itCel} branched cover on a {\clcnR} $n$-manifold, $\{\cD_m\}_{m\in\N_0}$ be a {\celSeq} of $f$, $\varrho$ be a visual metric for $f$, and $\dg$ be the distance function on $\mfd^n$. If $f\:(\mfd^n,\dg)\to(\mfd^n,\dg)$ is a uniform expanding BQS map with data $(\coverF(\cD_0),\eta)$ for some homeomorphism $\eta \: [0,+\infty) \to [0,+\infty)$, then the identity map $\id\:(\mfd^n,\dg)\to(\mfd^n,\varrho)$ is a quasisymmetry.
\end{theorem}

\begin{remark}
    	Theorem~\ref{t: cellular BQS: flowerBQS => QS equivalence} also follows from the verification that $\bigl\{\cD^{[n]}_m\bigr\}_{m\in\N_0}$ is a \emph{quasi-visual approximation} of both $(\mfd^n,\dg)$ and $(\mfd^n,\varrho)$; see~\cite[Section~2]{BM22}. Indeed, this is a consequence of Lemmas~\ref{l: QS unif: QV approx} and \ref{l: visual metric: metric properties of cells}. For the convenience of the reader, we include a direct proof.
\end{remark}

\begin{proof} 
	Denote $m(x,y)\=m_{f,\cD_0}(x,y)$ for all $x,\,y\in\mfd^n$.
	Let $\Lambda>1$ be the expansion factor of $\varrho$, and choose (by Lemma~\ref{l: visual metric: same expansion factor}) a constant $C>1$ such that (\ref{e: visual metric}) holds for all $x,\, y\in\mfd^n$.
	Let $\mu\in(1,+\infty)$ be the constant given by Lemma~\ref{l: QS unif: QV approx}~(iii). Let $\alpha_k,\,\beta_k\in(0,+\infty)$ for $k\in\N_0$ be constants given by Lemma~\ref{l: QS unif: QV approx}~(i), and $a\:\R\to(0,+\infty)$, $b\:\R\to(0,+\infty)$ be strictly increasing homeomorphisms such that $a(k)<\alpha_k<\beta_k<b(k)$ for each $k\in\N_0$.
	
	Fix arbitrary $x,\, y,\, z\in\mfd^n$. Denote $m_1\=m(x,y)$ and $m_2\=m(x,z)$. Choose $X_1\in\cD_{m_1}^{[n]}$ and $X_2\in\cD_{m_2}^{[n]}$ both containing $x$. By Lemma~\ref{l: QS unif: QV approx}~(iii),
    \begin{equation}\label{e: d(x,y)/d(x,z)==X1/X2}
        \dg(x,y)/\dg(x,z) \leq  \mu^2\diam_{\dg}X_1/\diam_{\dg}X_2.
    \end{equation}
	Now we give an estimate of $\diam_{\dg}X_1/\diam_{\dg}X_2$. Denote $k\=\abs{m_1-m_2}$. 
	
	If $m_1 \leq  m_2$, then by Lemma~\ref{l: QS unif: QV approx}~(i) and by (\ref{e: visual metric}), 
	\begin{equation*}
    {\diam_{\dg}X_1}/{\diam_{\dg}X_2} \leq \beta_k<b(k)\leq b\bigl(\log_{\Lambda}\bigl(C^2{\varrho(x,y)}/{\varrho(x,z)}\bigr)\bigr).\end{equation*}
    Likewise, if $m_1>m_2$, then 
    \begin{equation*}
        {\diam_{\dg}X_1}/{\diam_{\dg}X_2} \leq \alpha_k^{-1} < 1/a(k) \leq 1 \big/ a\bigl(-\log_\Lambda\bigl(C^2{\varrho(x,y)}/{\varrho(x,z)}\bigr)\bigr).
    \end{equation*}

    Combining the above two inequalities with (\ref{e: d(x,y)/d(x,z)==X1/X2}),
	${\dg(x,y)}/{\dg(x,z)} \leq \theta({\varrho(x,y)}/{\varrho(x,z)})$ 
    for the function
    $\theta\:[0,+\infty)\to[0,+\infty)$ defined by $\theta(t)\=\mu^2 \max\bigl\{b\bigl(\log_{\Lambda}\bigl(C^2t\bigr)\bigr),\, 1\big/ a\bigl(-\log_\Lambda \bigl(C^2t\bigr)\bigr)\bigr\}$. It is easy to verify that $\theta$ is strictly increasing, continuous, with $\theta(0)=0$ and $\theta(t)\to+\infty$ as $t\to+\infty$. Hence $\theta$ is a homeomorphism and the identity map $\id\:(\mfd^n,\varrho)\to(\mfd^n,\dg)$ is $\theta$-quasisymmetric.
\end{proof}

 We finish this section with a proof of Theorem~\ref{t: CBC: QS sph <=> UBQS}.
\begin{proof}[Proof of {Theorem~\ref{t: CBC: QS sph <=> UBQS}}]
	Let $\{\cD_m\}_{m\in\N_0}$ be a {\celSeq} of $f$.
	If the identity map $\id\:(\mfd^n,\dg)\to(\mfd^n,\varrho)$ is a quasisymmetry, then by Corollary~\ref{c: cellular BQS: QS equiv => flower UBQS}, $f\:(\mfd^n,\dg)\to(\mfd^n,\dg)$ is a uniform expanding BQS map.
	
	For the converse, suppose that $f\:(\mfd^n,\dg)\to(\mfd^n,\dg)$ is a uniform expanding BQS map with data $(\cV,\eta)$. Since $f$ is expanding, we may fix a constant $m_0\in\N$ such that $2\mesh_{\dg}(\cD_{m_0})$ is strictly less than the Lebesgue number of $\cV$. This implies that $f$ is a uniform expanding BQS map with data $(\coverF(\cD_{m_0}),\eta)$. Applying Theorem~\ref{t: cellular BQS: flowerBQS => QS equivalence} with the {\celSeq} $\{\cD_{m_0+k}\}_{k\in\N_0}$,  the identity map $\id\:(\mfd^n,\dg)\to(\mfd^n,\varrho)$ is a quasisymmetry. 
\end{proof}

\section{Cellular uniformly quasiregular maps} 
\label{sct: Rigidity of UQR}
In this section, we investigate rigidity
of expanding {\itCel} branched covers that are uniformly quasiregular,  leading to a proof of \ref{item:UQR-3}$\Leftrightarrow$\ref{item:UQR-4} in Theorem~\ref{tx: QS-UEBQS-Lattes-UQR}. More precisely, we establish the following theorem.
\begin{theorem}\label{t: cellular UQR => Lattes}
	Let $f\:\mfd^n\to\mfd^n$, $n \geq 3$, be an expanding {\itCel} branched cover on a {\clcnR} $n$-manifold.
	If $f$ is uniformly quasiregular, then $f$ is a chaotic Latt\`es map.
\end{theorem}

For the convenience of the reader, we recall the notion of Latt\`es maps (cf.~\cite{Ka22,May97,IM01}).
\begin{definition}\label{d: Lattes triple}
    Let $\mfd^n$, $n\ge 3$, be an {\orclcnR} $n$-manifold. A triple $(\Gamma,h,A)$ is a \defn{Latt\`es triple} on $\mfd^n$ if the following conditions are satisfied: 
\begin{enumerate}[label=(\roman*),font=\rm]
	\smallskip
	\item $\Gamma$ is a discrete subgroup of the isometry group $\Isom(\R^n)$.
	\smallskip
	\item $h\:\R^n\to\mfd^n$ is a quasiregular map that is strongly automorphic with respect to $\Gamma$, i.e., for all $x,\,y\in\R^n$, $h(x)=h(y)$ if and only if $y=\gamma(x)$ for some $\gamma\in\Gamma$.
	\smallskip
	\item $A$ is a conformal affine map (i.e.,~$A=\lambda U+v$ for some $\lambda>0$, $U\in O(n)$, and $v\in\R^n$) with $A\Gamma A^{-1}\subseteq\Gamma$.
\end{enumerate}
A \defn{Latt\`es map} is a uniformly quasiregular map $f\:\mfd^n\to\mfd^n$ for which there exists a Latt\`es triple $(\Gamma,h,A)$ with
$f\circ h=h\circ A$.
In this case, we call $f$ a \defn{Latt\`es map with respect to $(\Gamma,h,A)$}.
We call a Latt\`es map with respect to a Latt\`es triple $(\Gamma,h,A)$ \emph{chaotic} if $\Gamma$ is cocompact and $A$ is expanding, i.e., $A=\lambda U+v$ for some $\lambda>1$, $U\in O(n)$, and $v\in\R^n$.
\end{definition}

If a Latt\`es map is expanding (or more generally, topologically exact), then it is chaotic.
\begin{lemma}\label{l: Lattes: top exact => chaotic}
    If a Latt\`es map $f\:\mfd^n\to\mfd^n$, $n\ge 3$, is topologically exact, then $f$ is chaotic. In particular, if $f$ is expanding, then $f$ is chaotic.
\end{lemma}
\begin{proof}
    Let $f\:\mfd^n\to\mfd^n$ be a Latt\`es map with respect to a Latt\`es triple $(\Gamma,h,A)$, where $A=\lambda U+v$ for some $\lambda>0$, $U\in O(n)$, and $v\in\R^n$. Suppose $f$ is topologically exact. Let $D\subseteq\R^n$ be an arbitrary open ball. Since $f$ is topologically exact, there exists $k\in\N$ such that $f^k(h(D))=h(A^k(D))=\mfd^n$. Hence, $\R^n/\Gamma \cong h(\R^n)=\mfd^n$ is compact and $\Gamma$ is cocompact. Since $D$ can be arbitrarily small, it is easy to see $\lambda>1$. Thus, $f$ is chaotic.

    When $f$ is expanding, $f$ is topologically exact (see Proposition~\ref{p: expans=>LEO}), and hence chaotic.
\end{proof}

\begin{rem}
    The classical Latt\`es maps on $\widehat\C$ satisfy Definition~\ref{d: Lattes triple} in the case where $n=2$, and are moreover chaotic.
\end{rem}

\subsection{Cellular neighborhoods revisited}
First, we recall from Subsection~\ref{subsct: cellular neighborhood} (see (\ref{e: U^1 and U^2}) and (\ref{e: U^i_m(G)})) the open sets $U^i(G)$ for $G\subseteq\mfd^n$ and $i\in\{1,\,2\}$, defined with respect to a cell decomposition $\cD$, and $U^i_m(G)$ for $m\in\N_0$ with respect to a cellular sequence $\{\cD_m\}_{m\in\N_0}$. These open sets play an important role in the current section. Here we give several properties in the case where $G\=X$ for some $X\in\cD^{[n]}$.

\begin{lemma}\label{l: cellular neighborhood: Mayer-Vietoris}
	Let $\cD$ be a cell decomposition of a {\clcnR} $n$-manifold $\mfd^n$, $n\geq 2$, and consider $X\in\cD^{[n]}$.
	If $U^1(X)$ is contained in a chart, then the set $U^1(X)\smallsetminus X$ is path-connected. 
\end{lemma}
\begin{proof}
	Let $(V,\varphi)$ be a chart, where $V\subseteq\mfd^n$ is an open set containing $U^1(X)$, and $\varphi\:V\to \varphi(V)\subseteq\R^n$ is a homeomorphism. By Lemma~\ref{l: cel Nbhd: U^1(X)}~(i), $X\subseteq U^1(X)\subseteq V$ and $\varphi(X)\subseteq\varphi\bigl(U^1(X)\bigr)$.
	Set $A\=\R^n\smallsetminus \varphi(X)$ and $B\=\varphi\bigl(U^1(X)\bigr)$. Then $A\cup B=\R^n$, and $A\cap B=\varphi\bigl(U^1(X)\smallsetminus X\bigr)$.
	
	Consider the following Mayer--Vietoris sequence
	\begin{equation*}
	\cdots\to \tH_{q+1}(\R^n)\to \tH_q(A\cap B)\to \tH_q(A)\oplus \tH_q(B)\to \tH_q(\R^n)\to\cdots
	\end{equation*}
	of reduced homology groups. By the Jordan--Brouwer separation theorem, $A$ is path-connected. By Lemma~\ref{l: cel Nbhd: U^1(X)}~(ii), $B$ is path-connected. Then $\tH_0(A)=\tH_0(B)=0=\tH_1(\R^n)$. It follows that $\tH_0(A\cap B)=0$. Hence, $A\cap B$ and $U^1(X)\smallsetminus X=\varphi^{-1}(A\cap B)$ are path-connected.
\end{proof}

\begin{lemma}\label{l: cellular neighborhood: a chain around X}
	Let $\cD$ be a cell decomposition of a {\clcnR} $n$-manifold $\mfd^n$, $n\geq 2$. Let $Y\in\cD^{[n]}$ be such that $U^1(Y)$ is contained in a chart. Then for all $Y_1,\,Y_2\in\cU^1(Y)\smallsetminus\{Y\}$, there are $X_1,\,\dots,\,X_l\in\cU^1(Y)\smallsetminus\{Y\}$ with the following properties:
	\begin{enumerate}[label=(\roman*),font=\rm]
    	\smallskip
		\item $X_1=Y_1$ and $X_l=Y_2$.
		\smallskip
		\item For each $i\in\{1,\,\dots,\,l-1\}$, we have $X_i\cap X_{i+1}\neq\emptyset$ and $X_{i}\cap X_{i+1}\nsubseteq Y$.
	\end{enumerate}
\end{lemma}
\begin{proof}
	Denote $D\=U^1(Y)$. By Lemma~\ref{l: cellular neighborhood: Mayer-Vietoris}, $D\smallsetminus Y$ is connected.
	
	For each $X\in\cU^1(Y)\smallsetminus\{Y\}$, it is easy to check (cf.~the proof of Lemma~\ref{l: cellular neighborhood: int(cX) - cls, comp, bd}~(i) and (ii)) that $\stdCellint{X}\subseteq D\smallsetminus Y$, and thus $X\cap(D\smallsetminus Y)$ is a non-empty relatively closed subset of $D\smallsetminus Y$.
	
	Consider the finite cover $\Omega\=\bigl\{X\cap(D\smallsetminus Y):X\in\cU^1(Y)\smallsetminus\{Y\}\bigr\}$ of $D\smallsetminus Y$ by non-empty (relatively) closed subsets of $D\smallsetminus Y$.
    The connectedness of $D\smallsetminus Y$ implies that there exist $D_1,\,\dots,\,D_l\in \Omega$ such that
	\begin{enumerate}[label=(\roman*),font=\rm]
    	\smallskip
		\item$D_1=Y_1\cap(D\smallsetminus Y)$, $D_l=Y_2\cap(D\smallsetminus Y)$, and
		\smallskip
		\item $D_i\cap D_{i+1}\neq\emptyset$ for each $i\in\{1,\,\dots,\,l-1\}$.
	\end{enumerate}
	For each $i\in\{2,\,\dots,\, l-1\}$, suppose that $D_i=X_i\cap(D\smallsetminus Y)$ for $X_i\in\cU^1(Y)\smallsetminus\{Y\}$. Then we obtain $X_1,\,\dots,\, X_l$ as desired. 
\end{proof}

\begin{lemma}\label{l: cellular neighborhood: bd multi =>card(U^2)<N}
	Let $\{\cD_m\}_{m\in\N_0}$ be a {\celSeq} of an {\itCel} map $f\:\mfd^n\to\mfd^n$, $n\geq2$, on a {\clcnR} $n$-manifold with 
	$\sup\{\ind(x,f^m):m\in\N,\,x\in\mfd^n\}<+\infty$. 
	Then there exists $N\in\N$ such that
	$\card\bigl(\cU_{m}^1(X)\bigr) \leq 
	\card\bigl(\cU_{m}^2(X)\bigr) \leq  N$ for all $m\in\N_0$ and $X\in\cD_m^{\top}$,.
\end{lemma}
\begin{proof}
	By Corollary~\ref{c: multiplilcity: characterization of bounded N_loc}, set $N_0\=\sup_{m\in\N_0,\,x\in\mfd^n}\card\bigl\{X\in\cD_m^{\top}:x\in X\bigr\}<+\infty$.
	
	Fix arbitrary $m\in\N_0$ and $X\in\cD_m^{\top}$. For each $Y\in\cU_m^1(X)$, by Lemmas~\ref{l: properties of cell decompositions}~(iv) and~\ref{l: cell decomp: compact, and Mfd}~(iii), $X\cap Y$ contains an $m$-vertex. It follows that $\cU^1_{m}(X)=\bigcup_{p}\bigl\{Y\in\cD_{m}^{\top}:p\in Y\bigr\}$, where $p$ ranges over all $m$-vertices contained in $X$. Since, by Corollary~\ref{c: cellular: cellular homeo}, $\cD_{m}|_X=(f^m|_X)^*\bigl(\cD_0|_{f^m(X)}\bigr)$, we have $\card\bigl((\cD_m|_X)^{[0]}\bigr)=\card\bigl((\cD_0|_{f^m(X)})^{[0]}\bigr) \leq \card\cD_0$. Hence, $\card\bigl(\cU^1_{m}(X)\bigr) \leq  N_0 \card\cD_0$.
		
    Consider now $\card\bigl(\cU^2_{m}(X)\bigr)$. 
    The definition of $\cU^2_{m}(X)$ implies $\cU^2_{m}(X)=\bigcup\bigl\{\cU^1_{m}(Y):Y\in\cU^1_{m}(X)\bigr\}$,
	and the above discussion shows that $\card\bigl(\cU^2_{m,k}(X)\bigr) \leq  (N_0\card\cD_0)^2$. 
    
    The lemma follows after we define $N\=(N_0\card\cD_0)^2$.
\end{proof}

 \subsection{Markings for {\celSeq}s}\label{subsct: marking}
A key tool of the current section is the following notion that we call \emph{markings for cellular sequences}.

\begin{definition}[Marking for {\celSeq}s]\label{d: marking}
	For a {\celSeq} $\{\cD_m\}_{m\in\N_0}$ of an {\itCel} map $f\:\mfd^n\to\mfd^n$, $n\ge 2$, on a {\clcnR} $n$-manifold, we call a sequence of maps $\{\mk_m\:\cD_m\to\mfd^n\}_{m\in\N_0}$ a \defn{marking} for $\{\cD_m\}_{m\in\N_0}$ if, for all $m\in\N_0$ and $c\in\cD_m$, $\mk_m(c)\in\stdCellint{c}$ and, if $m\in\N$, $f(\mk_m(c))=\mk_{m-1}(f(c))$.
\end{definition}

To be concrete, one may construct a marking in the following way. For each $c_0\in\cD_0$, fix a point $\mk_0(c_0)\in\stdCellint{c_0}$. Then define $\mk_m(c)\=(f^m|_c)^{-1}(\mk_0(f^m(c)))$ for each $m\in\N$ and each $c\in\cD_m$.

Using markings, we give the following Lemmas~\ref{l: cellular neighborhood: a chain join marked cells} and~\ref{l: cellular neighborhood: continua A and B}, constructing continua with good properties that yield the metric estimates needed in Subsection~\ref{subsct: cellular UQR => Lattes}; see Figure~\ref{f: A and B}.

\begin{lemma}\label{l: cellular neighborhood: a chain join marked cells}
	Let $f\:\mfd^n\to\mfd^n$, $n\ge 2$, be an expanding {\itCel} branched cover on a {\clcnR} $n$-manifold, $\{\cD_m\}_{m\in\N_0}$ be a {\celSeq} of $f$, and $\{\mk_m\}_{m\in\N_0}$ be a marking for $\{\cD_m\}_{m\in\N_0}$. Then there exists $k_0\in\N$ such that for all $m,\,k\in\N_0$ with $k \geq  k_0$, and all $X\in\cD_m^{[n]}$, $c_1,\,c_2\in\cD_m|_X$, there exists a subset $\cE\subseteq\cD_{m+k}^{[n]}$ with the following properties:
	\begin{enumerate}[font=\rm,label=(\roman*)]
		\smallskip
		\item For each $X'\in\cE$, $X'\cap X\neq\emptyset$.
		\smallskip
		\item $\abs{\cE}$ is path-connected and contains $\mk_m(c_1)$ and $\mk_m(c_2)$.
		\smallskip
		\item \label{i: a chain join marked cells: separate} If $c_3\in\cD_m$ and $Y'\in\cD_{m+k}^{[n]}$ satisfy $c_1\nsubseteq c_3$, $c_2\nsubseteq c_3$, and $Y'\cap c_3\neq\emptyset$, then $Y'\cap \abs{\cE}=\emptyset$.
	\end{enumerate}
\end{lemma}

\begin{proof}
	{\it Step~1: Definition of the number $k_0\in\N$}.
	
	For all $c,\,c'\in\cD_0$, if $c\nsubseteq c'$, then since $\mk_0(c)\in\stdCellint{c}$, by Lemma~\ref{l: properties of cell decompositions}~(v), $\mk_0(c)\notin c'$ and thus $\dist(\mk_0(c),c')>0$. For each $c\in\cD_0$, set
	\begin{equation}\label{e: delta0}
	\delta_0(c)\=\inf\{\dist(\mk_0(c),c'):c'\in\cD_0,\,c\nsubseteq c'\}.
	\end{equation}

	Fix arbitrary $X\in\cD_0^{[n]}$ and $c_1,\,c_2\in\cD_0|_X$, since $X$ is an $n$-dimensional cell, we can fix a curve $\gamma_{X,c_1,c_2}\:[0,1]\to X$ associated with $X,\,c_1,\,c_2$ for which $\gamma_{X,c_1,c_2}(0)=\mk_0(c_1)$, $\gamma_{X,c_1,c_2}(1)=\mk_0(c_2)$, and $\gamma_{X,c_1,c_2}((0,1))\subseteq\stdCellint{X}$. 
	
	For each $c_3\in\cD_0$ with $c_1\nsubseteq c_3$ and $c_2\nsubseteq c_3$, by Lemma~\ref{l: properties of cell decompositions}~(v), we have $\mk_0(c_1),\,\mk_0(c_2)\notin c_3$ and $c_3\cap\stdCellint{X}=\emptyset$, which implies that $c_3$ does not meet $\gamma_{X,c_1,c_2}$, and thus $\dist(c_3,\gamma_{X,c_1,c_2})>0$. Set 
	\begin{equation}\label{e: delta1}
	\delta_1(X,c_1,c_2)\=\inf\{\dist(c_3,\gamma_{X,c_1,c_2}):c_3\in\cD_0,\,c_1\nsubseteq c_3,\,c_2\nsubseteq c_3\}.
	\end{equation}
	
	For all $c_1',\,c'_2\in\cD_0$ with $c_1\nsubseteq c'_1$ and $c_2\nsubseteq c'_2$, by Lemma~\ref{l: properties of cell decompositions}~(v), we have $\mk_0(c_i)\notin c'_i$ and $c'_i\cap\stdCellint{X}=\emptyset$ for $i\in\{1,\,2\}$, and thus
	$t\mapsto\dist(\gamma_{X,c_1,c_2}(t),c'_1)+\dist(\gamma_{X,c_1,c_2}(t),c'_2)$
	gives a strictly positive continuous function on $[0,1]$, which admits a strictly positive minimum.
	Set 
	\begin{equation}\label{e: delta2}	\delta_2(X,c_1,c_2)\=\inf_{c'_1,\,c'_2}\min_{t\in[0,1]}\{\dist(\gamma_{X,c_1,c_2}(t),c'_1)+\dist(\gamma_{X,c_1,c_2}(t),c'_2)\},
	\end{equation}
	where the infimum is taken over all $c_1',\,c'_2\in\cD_0$ with $c_1\nsubseteq c'_1$ and $c_2\nsubseteq c'_2$.
	
	Since $\cD_0$ is finite (cf.~Lemma~\ref{l: cell decomp: compact, and Mfd}~(i)), $\delta_0(c_1)$, $\delta_0(c_2)$, $\delta_1(X,c_1,c_2)$, $\delta_2(X,c_1,c_2)$, and
	\begin{equation*}
	\delta(X,c_1,c_2)\=\min\{\delta_0(c_1),\,\delta_0(c_2),\,\delta_1(X,c_1,c_2),\,\delta_2(X,c_1,c_2)\}.
    \end{equation*} 
    are strictly positive.
    We define $k_0\in\N$ to be an integer such that for all $k \geq  k_0$,
\begin{equation}\label{e: k0}
	\mesh(\cD_k)<\min \bigl\{  \delta(X,c_1,c_2) / 4 : X\in\cD_0^{[n]}, \, c_1,c_2\in\cD_0|_X \bigr\}.
    \end{equation}
    Such a $k_0$ exists since $f$ is expanding.

    \smallskip
	
	{\it Step~2: Construction of $\cE$ in the case where $m=0$}.
	
	Fix arbitrary $k \geq  k_0$, $X\in\cD_0^{[n]}$, and $c_1,\,c_2\in\cD_0|_X$.
	Define
	\begin{equation*}
	\cE\=\bigl\{X'\in\cD_{k}^{[n]}:X'\cap\gamma_{X,c_1,c_2}\neq\emptyset\bigr\}.
	\end{equation*}
	Clearly each $X'\in\cE$ meets $\gamma_{X,c_1,c_2}\subseteq X$. By Lemma~\ref{l: cell decomp: compact, and Mfd}~(ii), $\gamma_{X,c_1,c_2}\subseteq\abs{\cE}$ and thus $\abs{\cE}$ is a path-connected set containing $\mk_0(c_1)$ and $\mk_0(c_2)$.

    It remains to verify property~\ref{i: a chain join marked cells: separate} in the statement.
	Suppose $c_3\in\cD_0$ and $Y'\in\cD_{k}^{[n]}$ satisfy $c_1,\,c_2\nsubseteq c_3$ and $Y'\cap c_3\neq\emptyset$. If $Y'\cap\abs{\cE}\neq\emptyset$, then there exists
    $X'\in\cE$ such that $X'\cap Y'\neq\emptyset$, and it follows from $k\geq k_0$ and the definition of $k_0$ (cf.~(\ref{e: k0})) that
	\begin{equation*}
	\begin{aligned}
		\dist(\gamma_{X,c_1,c_2},c_3) \leq \diam X'+\diam Y'
		 \leq 2\mesh(\cD_{k})<\delta_1(X,c_1,c_2),
	\end{aligned}
\end{equation*}
	which contradicts the definition of $\delta_1(X,c_1,c_2)$ in (\ref{e: delta1}).
	
	\smallskip
	
	{\it Step~3: Construction of $\cE$ in the case where $m\in\N$.}
	
	Fix arbitrary $m,\,k\in\N$ with $k \geq  k_0$. Fix $X\in\cD_m^{[n]}$ and $\sigma_1,\,\sigma_2\in\cD_m|_X$. Denote $Y\=f^m(X)$, $\tau_1\=f^m(\sigma_1)$, $\tau_2\=f^m(\sigma_2)$, and let $\gamma_{Y,\tau_1,\tau_2}\:[0,1]\to Y$ be the curve associated with $Y,\,\tau_1,\,\tau_2$ defined in step~1.
	Define a curve $\tgamma_{X,\sigma_1,\sigma_2}\:[0,1]\to X$ by $\tgamma_{X,\sigma_1,\sigma_2}\=(f^m|_X)^{-1}\circ\gamma_{Y,\tau_1,\tau_2}$. The construction of $\gamma_{Y,\tau_1,\tau_2}$ in step~1 and definition of the marking $\{\mk_i\}_{i\in\N}$ imply that $\tgamma_{X,\sigma_1,\sigma_2}(0)=\mk_m(\sigma_1)\in\stdCellint{\sigma_1}$,
    $\tgamma_{X,\sigma_1,\sigma_2}((0,1))\subseteq\stdCellint{X}$, and
    $\tgamma_{X,\sigma_1,\sigma_2}(1)=\mk_m(\sigma_2)\in\stdCellint{\sigma_2}$.
	Set 
	\begin{equation*}
	\cE\=\bigl\{X'\in\cD_{m+k}^{[n]}:X'\cap\tgamma_{X,\sigma_1,\sigma_2}\neq\emptyset\bigr\}.
\end{equation*}
	Then each $X'\in\cE$ meets $\tgamma_{X,\sigma_1,\sigma_2}\subseteq X$, and $\abs{\cE}$ is a path-connected set that contains $\mk_m(\sigma_1)$ and $\mk_m(\sigma_2)$. 
	
	Now we verify property~\ref{i: a chain join marked cells: separate}.
	Suppose that $\sigma_3\in\cD_m$ and $\sigma_1,\,\sigma_2\nsubseteq\sigma_3$. We argue by contradiction and assume that $Z'\in\cD_{m+k}^{[n]}$ satisfy $Z'\cap \sigma_3\neq\emptyset$ and $Z'\cap\abs{\cE}\neq\emptyset$.
	Let $X'\in\cE$ satisfy $X'\cap Z'\neq\emptyset$. 

Since $\tgamma_{X,\sigma_1,\sigma_2}(0)=\mk_m(\sigma_1)\in\stdCellint{\sigma_1}$ and
$\tgamma_{X,\sigma_1,\sigma_2}((0,1))\subseteq\stdCellint{X}$ as shown above, by (\ref{e: cF(c)}), $\tgamma_{X,\sigma_1,\sigma_2}([0,1))\subseteq\flower_m(\sigma_1)$. Likewise, $\tgamma_{X,\sigma_1,\sigma_2}((0,1])\subseteq\flower_m(\sigma_2)$.
Then since $X'$ meets $\tgamma_{X,\sigma_1,\sigma_2}$, $X'$ meets $\flower_m(\sigma_1)\cup\flower_m(\sigma_2)$, and
we may assume $X'\cap \flower_m(\sigma_1)\neq\emptyset$. Since $\sigma_1\nsubseteq\sigma_3$,
    by Proposition~\ref{p: c-flower: structure}~(i), $\sigma_3\subseteq\mfd^n\smallsetminus\flower_m(\sigma_1)$ and thus $Z'\smallsetminus\flower_m(\sigma_1)\neq\emptyset$. The above discussion, with the fact that $X'\cup Z'$ is path-connected, implies $X'\cup Z'$ meets $\partial\flower_m(\sigma_1)$ (cf.~Lemma~\ref{l: there is a sub-curve}).
    Thus, we can pick $W'\in\{X',\,Z'\}$ satisfying $X'\cap W'\neq\emptyset$ and $W'\cap\partial\flower_m(\sigma_1)\neq\emptyset$.
    Then by Proposition~\ref{p: c-flower: fwd invariance}~(i), $f^m(W')\cap \partial\flower_0(\tau_1)\neq\emptyset$ and thus, by Proposition~\ref{p: c-flower: structure}~(iii), there exists $\tau'_1\in\cD_0$ such that $f^m(W')\cap\tau'_1\neq\emptyset$ and $\tau_1\nsubseteq\tau'_1$.
	
	Consider the following cases according to $X'$ and $\mk_m(\sigma_1)$:
	
	\smallskip
	
	{\it Case~1.} $\mk_m(\sigma_1)\in X'$. Then $\mk_0(\tau_1)=f^m(\mk_m(\sigma_1))\in f^m(X')$. Thus, it follows from $k\geq k_0$ and the definition of $ k_0$ (cf.~(\ref{e: k0})) that
	$\dist(\mk_0(\tau_1),\tau'_1) \leq \diam(f^m(X')\cup f^m(W'))\leq  2\mesh(\cD_{k})<\delta_0(\tau_1)/2$, 
	which contradicts the definition of $\delta_0(\tau_1)$ in (\ref{e: delta0}).
	
	\smallskip
	
	{\it Case~2.} $\mk_m(\sigma_1)\notin X'$.
    Then $X'$ meets $\tgamma_{X,\sigma_1,\sigma_2}((0,1])$, and thus, since $\tgamma_{X,\sigma_1,\sigma_2}((0,1])\subseteq\flower_m(\sigma_2)$ as shown above, $X'$ meets $\flower_m(\sigma_2)$.
    Similar to our construction of $W'$, we choose $W''\in\cD_{m+k}^{[n]}$ in such a way that $W''\cap X'\neq\emptyset$, $W''\cap\partial\flower_m(\mk_m(\sigma_2))\neq\emptyset$, and there exists $\tau'_2\in\cD_0$ such that $f^m(W'')\cap\tau'_2\neq\emptyset$ and $\tau_2\nsubseteq\tau'_2$. 
	Let $t_0\in[0,1]$ be such that $\gamma_{Y,\tau_1,\tau_2}(t_0)\in f^m(X')$. Then it follows from $k\geq k_0$ and the definition of $ k_0$ (cf.~(\ref{e: k0})) that
	\begin{equation*}
	\begin{aligned}
		\dist(\gamma_{Y,\tau_1,\tau_2}(t_0),\tau'_1)+\dist(\gamma_{Y,\tau_1,\tau_2}(t_0),\tau'_2)
		 \leq  &2\diam(f^m(X'))+\diam(f^m(W'))+\diam(f^m(W''))\\
		 \leq  &4\mesh(\cD_{k})<\delta_2(Y,\tau_1,\tau_2).
	\end{aligned}
    \end{equation*}
	This contradicts the definition of $\delta_2(Y,\tau_1,\tau_2)$ in (\ref{e: delta2}).
\end{proof}

\begin{figure}[h]
	\includegraphics[scale=0.3]{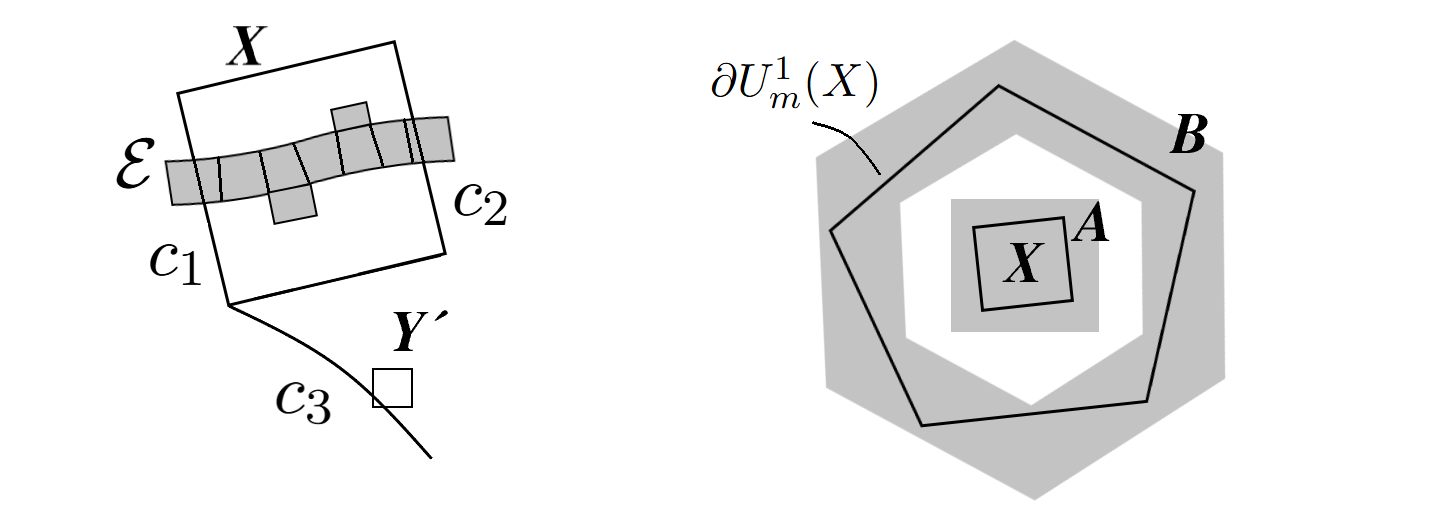}
    \caption{Illustrations for Lemmas~\ref{l: cellular neighborhood: a chain join marked cells}~(left) and~\ref{l: cellular neighborhood: continua A and B}~(right).}
    \label{f: A and B}
\end{figure}

\begin{lemma}\label{l: cellular neighborhood: continua A and B}
    Let $f\:\mfd^n\to\mfd^n$, $n\ge 2$, be an expanding {\itCel} branched cover on a {\clcnR} $n$-manifold and $\{\cD_m\}_{m\in\N_0}$ be a {\celSeq} of $f$. Then there exist $m_0,\,k_0\in\N$ such that for each pair of integers $m \geq  m_0$, $k\geq k_0$, and each $X\in\cD_m$, there exist continua $A$ and $B$ such that the following statements are true:
	\begin{enumerate}[label=(\roman*),font=\rm]
    	\smallskip
		\item Both $A$ and $B$ are unions of chambers in $\cD_{m+k}^{[n]}$.
		\smallskip
		\item $X\subseteq A\subseteq U_m^1(X)$ and $\partial U_m^1(X)\subseteq B\subseteq U_m^2(X)$.
		\smallskip
		\item $A\cap B=\emptyset$. 
	\end{enumerate}
\end{lemma}
\begin{proof}
	Since $f$ is expanding, there exists $m_0\in\N$ such that, when $m \geq  m_0$, for each $X\in\cD_m^{[n]}$, $U^1_m(X)$ is contained in a chart.
	
	Fix a marking $\{\mk_m\}_{m\in\N_0}$ for $\{\cD_m\}_{m\in\N_0}$, and let $k_0\in\N$ be the constant given by Lemma~\ref{l: cellular neighborhood: a chain join marked cells}. Without loss of generality, we may assume $k_0$ to be sufficiently large that, by Lemma~\ref{l: expansion: J_m--->+infty}, $J(\cD_k,\cD_0)>2$ for all $k\geq k_0$; recall that $J(\cD_{k},\cD_0)$ is defined in (\ref{e: the joining number J(m>k)}).
	
	Fix arbitrary $m,\,k\in\N_0$ with $m \geq  m_0$ and $k\ge k_0$. Fix $X\in\cD_m^{[n]}$. 
	We define
	\begin{equation*}
	A\=\bigcup\bigl\{X'\in\cD_{m+k}^{[n]}:X'\cap X\neq\emptyset\bigr\}.
    \end{equation*}
	Clearly $A\supseteq X$ (cf.~Lemma~\ref{l: cell decomp: compact, and Mfd}~(ii)) and hence $A$ is path-connected. Then since $A$ is a union of finitely many cells (cf.~Lemma~\ref{l: cell decomp: compact, and Mfd}~(i)), $A$ is a continuum. Suppose that $A\nsubseteq U^1_m(X)$, then by Lemma~\ref{l: cellular neighborhood: int(cX) - cls, comp, bd}~(ii), there exists $X'\in\cD_{m+k}^{[n]}$ that meets both $X$ and some $Y\in\cD_m^{[n]}\smallsetminus\cU^1_m(X)$, and it follows from Corollary~\ref{c: CBC: connected set in flower}~(iii) that $f^m(X')$ {\jsOpSd} of $\cD_0$, which contradicts $J(\cD_{k},\cD_0)>2$. Hence $A\subseteq U^1_m(X)$.
    
	Now we construct the continuum $B$. 
	First, set 
    \begin{equation}\label{e: B_0}
        B_0\=\bigcup\bigl\{X'\in\cD_{m+k}^{[n]}:X'\cap\partial U_m^1(X)\neq\emptyset\bigr\}.
    \end{equation}
    Clearly $B_0\supseteq\partial U_m^1(X)$.
    The arguments showing $A\subseteq U^1_m(X)$ yield that $B_0\subseteq U^2_m(X)$.
	Suppose $B_0$ meets $A$, then we apply Lemma~\ref{l: cellular neighborhood: int(cX) - cls, comp, bd}~(ii) and Corollary~\ref{c: CBC: connected set in flower}~(iii) as above to obtain $X',\,X''\in\cD_{m+k}^{[n]}$ such that $X'\cap X''\neq\emptyset$ and $f^m(X')\cup f^m(X'')$ {\jsOpSd} of $\cD_0$, which contradicts $J(\cD_{k},\cD_0)>2$. Hence $B_0\cap A=\emptyset$.
	
	Now we extend $B_0$ to a continuum $B$ satisfying the desired properties.

	By Lemma~\ref{l: cellular neighborhood: int(cX) - cls, comp, bd}~(iii), $\partial U_m^1(X)=\abs{\cC}$, where 
    \begin{equation}\label{e: the collection cC}
        \cC\=\bigl\{c\in\cD_m:\partial U_m^1(X)\cap\stdCellint{c}\neq\emptyset\bigr\}=\bigl\{c\in\cD_m:c\subseteq\partial U_m^1(X)\bigr\}.
    \end{equation}
	Consider $\sigma_0,\,\sigma_1\in\cC$. In what follows, we construct a connected set $B(\sigma_0,\sigma_1)$ that contains $\mk_m(\sigma_0)$ and $\mk_m(\sigma_1)$. Since $\partial U^1_m(X)\cap\stdCellint{\sigma_0}\neq\emptyset$, by Lemma~\ref{l: properties of cell decompositions}~(v), there exists $Y_0\in\cU_m^1(X)$ containing $\sigma_0$. Likewise, there is $Y_1\in\cU_m^1(X)$ containing $\sigma_1$. By the choice of $m\geq m_0$, $U^1_m(X)$ is contained in a chart, and we can use Lemma~\ref{l: cellular neighborhood: a chain around X} to find $X_1,\,\dots,\,X_l\in\cU_m^1(X)\smallsetminus\{X\}$ with the following properties: 
	\begin{itemize}
    	\smallskip
		\item[(1)] $X_1=Y_0$ and $X_l=Y_1$.
		\smallskip
		\item[(2)] For each $1 \leq  i \leq  l-1$, we have $X_{i}\cap X_{i+1}\neq\emptyset$ and $X_{i}\cap X_{i+1}\nsubseteq X$.
	\end{itemize}
	For each integer $1 \leq  i \leq  l-1$, since $X_i\cap X_{i+1}\nsubseteq X$ and by Lemma~\ref{l: properties of cell decompositions}~(iv), there is $c_i\in\cD_m$ such that $c_i\subseteq X_{i}\cap X_{i+1}$ and $c_i\nsubseteq X$. We also denote $c_0\=\sigma_0$ and $c_l\=\sigma_1$. For each integer $1 \leq  i \leq  l$, by Lemma~\ref{l: cellular neighborhood: a chain join marked cells}, we can find $\cE_i\subseteq\cD_{m+k}^{[n]}$ with the following properties:
    \begin{enumerate}[label=(\roman*),font=\rm]
		\smallskip
		\item For each $X'\in\cE_i$, $X'\cap X_i\neq\emptyset$.
		\smallskip
		\item $\abs{\cE_i}$ is path-connected and contains $\mk_m(c_{i-1}),\,\mk_m(c_i)$.
        \smallskip
        \item For each $Y'\in\cD_{m+k}^{[n]}$ that meets $X$, $\abs{\cE_i}\cap Y'=\emptyset$.
	\end{enumerate}
	
    Fix $1\le i\le l$. First, by the construction of $A$ and the above property (iii) of $\cE_i$, we have $\abs{\cE_i}\cap A=\emptyset$. Then we show $\abs{\cE_i}\subseteq U_m^2(X)$. Suppose that $\abs{\cE_i}\nsubseteq U_m^2(X)=\mfd^n\smallsetminus\Absbig{\cD_m^{[n]}\smallsetminus\cU^2_{m}(X)}$ (cf.~Lemma~\ref{l: cellular neighborhood: int(cX) - cls, comp, bd}~(ii)), then by the above property~(i) of $\cE_i$, there exists
	$X'\in\cE_i$ such that $X'$ meets both $X_i$ and some $Y\in\cD_m^{[n]}\smallsetminus\cU^2_{m}(X)$, which satisfy $X_i\cap Y=\emptyset$ (cf.~(\ref{e: U^1 and U^2})). Thus, by Corollary~\ref{c: CBC: connected set in flower}~(iii), $f^m(X')$ {\jsOpSd} of $\cD_0$, which contradicts $J(\cD_{k_0},\cD_0)>2$. Hence, $\abs{\cE_i}\subseteq U_m^2(X)$.
	
	Set $B(\sigma_0,\sigma_1)\=\abs{\cE_1}\cup\cdots\cup \abs{\cE_l}$. By the property~(ii) of $\cE_i$, $1\le i\le l$, and the above discussion, $B(\sigma_0,\sigma_1)$ is a path-connected subset of $U^2_m(X)\smallsetminus A$ that contains $\mk_m(\sigma_0)$ and $\mk_m(\sigma_1)$.
    
    Define 
	\begin{equation}\label{e: continuum B}
	\begin{aligned}
		B\=B_0\cup\bigcup\{B(\sigma,\tau):\sigma,\,\tau\in\cC\}.
	\end{aligned}
\end{equation}
	Since $\partial U_m^1(X)\subseteq B_0\subseteq U^2_m(X)\smallsetminus A$, we have $A\cap B=\emptyset$ and $\partial U_m^1(X)\subseteq B\subseteq U_m^2(X)$.
	
	It remains to verify that $B$ is a continuum. First, since $B$ is a union of finitely many cells (cf.~Lemma~\ref{l: cell decomp: compact, and Mfd}~(i)), $B$ is compact. Then we show that $B$ is connected---more precisely, for each pair $x_0,\,x_1$ of points in $B$, we can find a path-connected subset of $B$ that contains $x_0,\,x_1$. Consider the following cases according to $x_0,\,x_1\in B$.
	
	\smallskip
	
	{\it Case~1.} $\{x_0,\,x_1\}\subseteq B_0$. For each $i\in\{0,\,1\}$, suppose $X'_i\in\cD_{m+k}^{[n]}$ and $\sigma_i\in\cC$ satisfy $x_i\in X'_i$ and $X'_i\cap \sigma_i\neq\emptyset$, and set $B(\sigma_i)\=\bigcup\bigl\{X'\in\cD_{m+k}^{[n]}:X'\cap \sigma_i\neq\emptyset\bigr\}$. Then for each $i\in\{0,\,1\}$, $B(\sigma_i)$ is path-connected and $\{x_i,\,\mk_m(\sigma_i)\}\subseteq B(\sigma_i)\subseteq B_0$ (cf.~(\ref{e: B_0}) and (\ref{e: the collection cC})). Thus, since $B(\sigma_0,\sigma_1)$ is constructed to be a path-connected set containing $\mk_m(\sigma_0)$ and $\mk_m(\sigma_1)$, the subset $B(\sigma_0)\cup B(\sigma_0,\sigma_1)\cup B(\sigma_1)$ of $B$ (cf.~(\ref{e: continuum B})) is path-connected and contains $x_0,\,x_1$.
	
	\smallskip
	
	{\it Case~2.} $\{x_0,\,x_1\}\nsubseteq B_0$. We may assume $x_0\in B(\sigma_0,\tau_0)$ for some $\sigma_0,\,\tau_0\in\cC$. Since $B(\sigma_0,\tau_0)$ is path-connected and contains $\mk_m(\sigma_0)$, it suffices to find a path-connected subset of $B$ that contains $\mk_m(\sigma_0)$ and $x_1$. If $x_1\in B_0$, then since $\mk_m(\sigma_0)\in\partial U_m^1(X)\subseteq B_0$ (cf.~(\ref{e: B_0}) and (\ref{e: the collection cC})), this case reduces to Case~1. If $x_1\notin B_0$, then there exist $\sigma_1,\,\tau_1\in\cC$ such that $x_1\in B(\sigma_1,\tau_1)$ (cf.~(\ref{e: continuum B})), and thus, by a similar argument as above, it suffices to find a path-connected subset of $B$ that contains $\mk_m(\sigma_0)$ and $\mk_m(\sigma_1)$. This case also reduces to Case~1 since $\mk_m(\sigma_0),\,\mk_m(\sigma_1)\in B_0$.
\end{proof}

\subsection{Rigidity of UQR maps}
\label{subsct: cellular UQR => Lattes}

Here, we prove Theorem~\ref{t: cellular UQR => Lattes}.
First, we give several technical lemmas, providing key metric estimates for the proof of Theorem~\ref{t: cellular UQR => Lattes}.
\begin{lemma}\label{l: cellular + QR + big i(x,f) => vertex}
    Let $f\:\mfd^n\to\mfd^n$, $n\geqslant3$, be a $(\cD_1,\cD_0)$-cellular branched cover on a {\clcnR} $n$-manifold, and suppose $f$ is $K$-quasiregular for some $K\geq 1$. Then there exists a constant $I_{n,K}$, depending only on $n$ and $K$, having the following property: if $x\in\mfd^n$ satisfies $i(x,f)\geq I_{n,K}$, then $x$ is a $1$-vertex.
\end{lemma}
\begin{proof}
    In what follows, we show that $I_{n,K}\=(16^nnK)^{n-1}$ is the desired constant.

    We argue by contradiction and assume that a point $x\in\mfd^n$, which satisfies $i(x,f)\geq I_{n,K}$, is not a $1$-vertex. 
    By Lemma~\ref{l: properties of cell decompositions}~(ii), there exists $c\in\cD_1$ such that $x\in\stdCellint{c}$ and $\dim(c)=d\geq 1$.

    Fix $L\in[1,2)$. Let $(U,\phi)$ and $(V,\psi)$ be $L$-bi-Lipschitz charts at $x$ and $f(x)$, respectively, such that $f(U)\subseteq V$. Then $g\=\psi\circ f\circ\phi^{-1}\:\phi(U)\to \psi(V)\subseteq\R^n$ is a $\bigl(KL^{4n}\bigr)$-quasiregular map. Since $d\geq 1$ and $x\in\stdCellint{c}$, there exists a continuum $E\subseteq \stdCellint{c}$ that is also contained in $U$; for example, one may choose a homeomorphism $h\:c\to\cls{\B^d}$ and set $E\=h^{-1}\bigl(\cls{\B^d(h(x),\epsilon)}\bigr)$, where the radius $\epsilon>0$ is sufficiently small. By Corollary~\ref{c: multi: N(f,x)=const on int(c)}, $i(x',f)=i(x,f)$ for all $x'\in\stdCellint{c}$. It follows that $i(z,g)=i(x,f)\geq I_{n,K}\geq \bigl(nKL^{4n}\bigr)^{n-1}$ for all $z\in \phi(E)$. However, by \cite[Chapter~III, Corollary~5.9]{Ri93}, $\inf_{z\in\phi(E)}i(z,g)<\bigl(nKL^{4n}\bigr)^{n-1}$, which yields a contradiction.
\end{proof}

\begin{lemma}\label{l: itCel: inf multi + UQR => period branch}
    Let $f\:\mfd^n\to\mfd^n$, $n \geq 3$, be an {\itCel} branched cover on a {\clcnR} $n$-manifold. If $f$ is uniformly quasiregular and $\sup\{\ind(x,f^m):m\in\N,\,x\in\mfd^n\}=+\infty$, then $f$ has a periodic branch point.
\end{lemma}
\begin{proof}
   Let $f$ be uniformly $K$-quasiregular for some $K\geqslant 1$, and $I_{n,K}$ be the constant given by Lemma~\ref{l: cellular + QR + big i(x,f) => vertex}. Suppose $\sup\{\ind(x,f^m):m\in\N,\,x\in\mfd^n\}=+\infty$. 

    Fix a {\celSeq} $\{\cD_m\}_{m\in\N_0}$ of $f$. Let $x_0\in\mfd^n$ and $m\in\N$ satisfy 
    \begin{equation}\label{e: i(x,f^m)>(N(f)In,K)^cD+1}
        \ind(x_0,f^m)> (N(f) I_{n,K})^{1+\card\cD_0}.
    \end{equation}
    By the definition of the local multiplicity (cf.~(\ref{e: N_loc(f,x)})), for all $x\in\mfd^n$ and $k,\,l\in\N_0$, 
    \begin{equation}\label{e: i(x,f^k+l) = i(x,f^k) i(f^k(x),f^l)}
    \ind\bigl(x,f^k\bigr)\leqslant\ind\bigl(x,f^{k+l}\bigr)\leqslant \ind\bigl(x,f^k\bigr)\cdot\ind\bigl(f^k(x), f^l\bigr)\leqslant\ind\bigl(x,f^k\bigr)\cdot N\bigl(f^l\bigr),
    \end{equation} 
    and, in particular, $\ind\bigl(x,f^k\bigr)\leq\ind\bigl(x,f^{k+1}\bigr)\leq \ind\bigl(x,f^k\bigr)\cdot N(f)$. It follows that if $x\in\mfd^n$ and $k,\,M\in\N$ satisfy $\ind\bigl(x,f^k\bigr)>MN(f)$, then there exists $1\leq l< k$ with
    \begin{equation}\label{e: M<i(x,f^l)<MN(f)}
        M\leq \ind\bigl(x,f^l\bigr)\leq M N(f).
    \end{equation}
    Using (\ref{e: i(x,f^m)>(N(f)In,K)^cD+1}), (\ref{e: i(x,f^k+l) = i(x,f^k) i(f^k(x),f^l)}), and (\ref{e: M<i(x,f^l)<MN(f)}), we can find $a_1,\,\dots,\,a_s\in\N$ satisfying the following conditions:
    \begin{enumerate}[label=(\roman*),font=\rm]
        \smallskip
        \item For each $1\leq j\leq s$, we have $I_{n,K}\leqslant\ind\bigl(f^{b_{j-1}}(x_0),f^{a_j}\bigr)\leqslant I_{n,K}N(f)$. Here $b_{j}\= a_1+\dots+a_j$ for each $1\leqslant j\leqslant s$ and $b_0\=0$.
        \smallskip
        \item $a_1+\cdots+a_s\leqslant m$ and $\ind(x_0,f^{a_1+\cdots+a_s})\leqslant\ind(x_0,f^m)\leqslant \ind(x_0,f^{a_1+\cdots+a_s})\cdot(N(f)I_{n,K})$.
    \end{enumerate}
    For each integer $1\leq j\leq s$, since $f^{a_j}$ is a $\bigl(\cD_{a_j},\cD_0\bigr)$-cellular (cf.~Remark~\ref{r: basic cellular property}) $K$-quasiregular map, and $f^{b_{j-1}}(x_0)$ satisfies $I_{n,K}\leqslant\ind\bigl(f^{b_{j-1}}(x_0),f^{a_j}\bigr)$, it follows from Lemma~\ref{l: cellular + QR + big i(x,f) => vertex} that $f^{b_{j-1}}(x_0)$ is an $a_j$-vertex, and thus $f^{b_j}(x_0)=f^{a_j}\bigl(f^{b_{j-1}}(x_0)\bigr)$ is a $0$-vertex.
    
    Combining (\ref{e: i(x,f^k+l) = i(x,f^k) i(f^k(x),f^l)}) and condition~(i) above, we get
    \begin{equation*}
        (I_{n,K}N(f))^s\geq\prod_{i=1}^s\ind\bigl(f^{b_{i-1}}(x_0),f^{a_i}\bigr)\geq\ind\bigl(x_0,f^{b_s}\bigr).
    \end{equation*}
    Thus, by (\ref{e: i(x,f^m)>(N(f)In,K)^cD+1}) and condition~(ii) above, 
    \begin{equation*}
    (I_{n,K}N(f))^s\geq(N(f)I_{n,K})^{-1}\ind(x_0,f^m)>(N(f)I_{n,K})^{\card\cD_0},    
    \end{equation*}
    which implies $s>\card\cD_0$. Since all $f^{b_j}(x_0)$, $1\leqslant j\leqslant s$, are $0$-vertices, there exist integers $1\leqslant k<l\leqslant s$ such that $f^{b_k}(x_0)=f^{b_l}(x_0)$. Then $f^{b_k}(x_0)$ is a periodic point of $f$. Since $\ind\bigl(f^{b_k}(x_0),f^{b_l-b_k}\bigr)\geq I_{n,K}>1$, the forward orbit of $f^{b_k}(x_0)$ contains a branch point $x^*$, which is periodic.
\end{proof}

\begin{prop}\label{p: CPCF: UQR => bd multi}
	Let $f\:\mfd^n\to\mfd^n$, $n\geq3$, be an expanding {\itCel} branched cover on a {\clcnR} $n$-manifold. If $f$ is uniformly quasiregular, then
	$\sup\{\ind(x,f^m):m\in\N,\,x\in\mfd^n\}<+\infty$.
\end{prop}
\begin{proof}
    Assume that $f$ is uniformly quasiregular.
	We argue by contradiction and assume $\sup\{\ind(x,f^m):m\in\N,\,x\in\mfd^n\}=+\infty$. Then by Lemma~\ref{l: itCel: inf multi + UQR => period branch}, there exists a periodic point $x^*\in B_f$, which is in a super-attracting basin in the Fatou set $\cF(f)$ of $f$ (cf.~\cite{HMM04}). Thus, $\cF(f)\neq\emptyset$.
	
	By Proposition~\ref{p: expans=>LEO}, $f$ is topologically exact. 
	Consequently, since $\cF(f)$ is a non-empty open set that is totally invariant for $f$ (cf.~\cite{HMM04}), $\cF(f)=f^k(\cF(f))=\mfd^n$ for some $k\in\N$. However, this is a contradiction, because $f$ being a non-injective uniformly quasiregular map on a {\clcnR} $n$-manifold implies that the Julia set $\cJ(f)=\mfd^n\smallsetminus \cF(f)$ is non-empty (cf.~\cite{IM96, HMM04}).
\end{proof}

A key observation in the proof of Theorem~\ref{t: cellular UQR => Lattes} is the following relative separation property of {\celSeq}s. 

\begin{prop}\label{prop: UQR + bd mul (no need for Mkv) => dist/diam>}
	Let $f\:\mfd^n\to\mfd^n$, $n\geq 3$, be an expanding {\itCel} branched cover on a {\clcnR} $n$-manifold $(\mfd^n,\dg)$, and $\{\cD_m\}_{m\in\N_0}$ be a {\celSeq} of $f$. 
	If $f$ is uniformly quasiregular, 
	then there exists a constant $\lambda\in(0,1]$ such that, for all $m\in\N_0$ and $X,\,Y\in\cD_m^{[n]}$ with $X\cap Y=\emptyset$, we have
	$\dist_{\dg}(X,Y) \geq \lambda \diam_{\dg} X$.
\end{prop}
\begin{proof}
	Suppose that $f$ is uniformly $K$-quasiregular for a constant $K \geq  1$.
	We fix $L\in[1,2]$ and an atlas $\{(U_\alpha,\phi_\alpha)\}_{\alpha\in\sA}$ for which each coordinate map $\phi_\alpha\:U_\alpha\to \phi_\alpha(U_\alpha)\subseteq\B^n$ is $L$-bi-Lipschitz.
	
	Let $m_0,\,k_0\in\N$ be constants given by Lemma~\ref{l: cellular neighborhood: continua A and B}. We may also assume $m_0$ is so large that if $m \geq  m_0$,
	then for each $X\in\cD_m^{[n]}$, $U^2_m(X)$ is contained in some chart among $\{(U_\alpha,\phi_\alpha)\}_{\alpha\in\sA}$.
	Set 
	\begin{equation*}
	\lambda_0\=\min\bigl\{\dist_{\dg}(X,Y)\big/\diam_{\dg} X: 0\leq l \leq  m_0,\,X,\,Y\in\cD_l^{[n]},\,X\cap Y=\emptyset\bigr\}.
    \end{equation*}
	Now it suffices to give a constant $\lambda_1>0$ such that $\dist_{\dg}(X,Y) \geq \lambda_1\diam_{\dg} X$ for all $X,\,Y\in\cD_m^{[n]}$ satisfying $m \geq  m_0$ and $X\cap Y=\emptyset$. For this, we utilize the moduli of curve families (see~Appendix~\ref{Ap: moduli} for a discussion) to give metric estimates.
	
	Fix arbitrary $m\in\N$ with $m>m_0$ and $X\in\cD_m^{[n]}$. Denote $X_0\=f^{m-m_0}(X)\in\cD_{m_0}^{[n]}$. 
    Then by Lemma~\ref{l: cellular neighborhood: continua A and B}, we can find continua $A$ and $B$ having the following properties:
	\begin{enumerate}[label=(\roman*),font=\rm]
    	\smallskip
		\item Both $A$ and $B$ are unions of chambers in $\cD_{m+k_0}^{[n]}$.
		\smallskip
		\item $X\subseteq A\subseteq U_m^1(X)$ and $\partial U_m^1(X)\subseteq B\subseteq U_m^2(X)$.
		\smallskip
		\item $A\cap B=\emptyset$. 
	\end{enumerate}

    Since $m>m_0$, $U^2_m(X)$ and $U^2_{m_0}(X_0)$ are contained in charts among $\{(U_\alpha,\phi_\alpha)\}_{\alpha\in\sA}$.
    Let $\phi$ and $\psi$ be $L$-bi-Lipschitz coordinate maps in $\{\phi_\alpha\}_{\alpha\in\sA}$ that map $U^2_m(X)$ and $U^2_{m_0}(X_0)$, respectively, into $\B^n$. Since $f^{m-m_0}\bigl(U_m^2(X)\bigr)\subseteq U^2_{m_0}(X_0)$ (cf.~Lemma~\ref{l: cel Nbhd: fwd invariance}),
	we may consider the map
	\begin{equation*}
	h\:\phi\bigl(U^2_m(X)\bigr)\to\psi\bigl(U^2_{m_0}(X_0)\bigr),\quad x\mapsto\psi\bigl(f^{m-m_0}\bigl(\phi^{-1}(x)\bigr)\bigr).
    \end{equation*}
	Since $f^{m-m_0}$ is $K$-quasiregular and $\phi,\,\psi$ are $L$-bi-Lipschitz, $h$ is $\bigl(KL^{4n}\bigr)$-quasiregular.
	
    Denote $D\=\phi\bigl(U^1_m(X)\bigr)$, $U\=\phi\bigl(U^2_m(X)\bigr)$, $V\=\psi\bigl(U^2_{m_0}(X_0)\bigr)$, $E\=\phi(A)$, and $F\=\phi(B)$. Then $D,\,U,\,V$ are domains for which $\cls{D}\subseteq U\subseteq\B^n$ (cf.~Lemma~\ref{l: cel Nbhd: U^1(X)}~(i) and~(ii)) and $h(U)\subseteq V\subseteq\B^n$. By the choice of $A,\,B$, it is clear that $E,\,F$ are disjoint continua in $U$ satisfying $E\subseteq D$ and $\partial D\subseteq F$. 
   
    The properties $\cls{D}\subseteq\B^n$, $E\subseteq D$, and $\partial D\subseteq F$ imply $\diam E \leq  \diam F$. 
    Thus, by {\cite[Lemma~7.38]{Vu88}}, we have the following lower bound for the modulus of the curve family $\Delta(E,F)$ (see Appendix~\ref{Ap: moduli} for definition):
	\begin{equation}\label{e: UQR: MOD>min diam/dist}
		\modul(\Delta(E,F)) \geq  c_n\log\biggl(1+\frac{\min\{\diam E,\,\diam F\}}{\dist(E,F)}\biggr) = c_n\log\biggl(1+\frac{\diam E}{\dist(E,F)}\biggr),
	\end{equation}
    where $c_n>0$ is a constant depending only on $n$.
	Since $\partial D\subseteq F$ and by Lemma~\ref{l: there is a sub-curve}, for each curve $\gamma\in\Delta(E,F)$, there is a sub-curve of $\gamma$ in the curve family $\Delta(E,F;D)$ (see Appendix~\ref{Ap: moduli} for definition). Thus, by {\cite[Lemma~5.3]{Vu88}}, we have
	\begin{equation}\label{e: UQR: MOD(E,F,D)}
		\modul(\Delta(E,F)) \leq \modul(\Delta(E,F;D)).
	\end{equation}
	
	Consider an arbitrary curve $\gamma\in\Delta(E,F;D)$. Since $\phi^{-1}(\gamma)\in\Delta\bigl(A,B; U^1_m(X)\bigr)$ meets both $A$ and $B$, which are disjoint unions of cells in $\cD_{m+k_0}^{[n]}$ (see properties (i) and (iii) of $A,\,B$), by Corollary~\ref{c: CBC: connected set in flower}~(iii), $f^{m-m_0}\bigl(\phi^{-1}(\gamma)\bigr)$ cannot be contained in any $(m_0+k_0)$-flower. This implies
	$\diam_{\dg}\bigl(f^{m-m_0}\bigl(\phi^{-1}(\gamma)\bigr)\bigr) \geq \delta_0$, 
	where $\delta_0>0$ is the Lebesgue number of the open cover $\coverF(\cD_{m_0+k_0})$. Since
	$\psi$ is $L$-bi-Lipschitz, the curve $h(\gamma)=\psi\bigl(f^{m-m_0}\bigl(\phi^{-1}(\gamma)\bigr)\bigr)$ satisfies $\diam(h(\gamma)) \geq  \delta_0/L$.
    So
    \begin{equation*}
        h(\Delta(E,F;D))\subseteq\{\zeta:\zeta\text{ is a curve in }V,\,\length(\zeta) \geq \delta_0/L\},
    \end{equation*}
	and thus, by {\cite[Lemmas~5.2~(ii) and~5.5]{Vu88}} and the fact that $V\subseteq\B^n$,
    \begin{equation}\label{e: M(h(Delta))}
        \modul(h(\Delta(E,F;D))) \leq  m(V)\cdot(\delta_0/L)^{-n} \leq \alpha_n\cdot(\delta_0/L)^{-n},
    \end{equation}
	where $m(\cdot)$ stands for the $n$-dimensional Lebesgue measure and $\alpha_n\=m(\B^n)$.
    
    By Proposition~\ref{p: CPCF: UQR => bd multi}, $\sup\bigl\{\ind\bigl(x,f^j\bigr):j\in\N,\,x\in\mfd^n\bigr\}<+\infty$.
    Let $N\in\N$ be the constant given by Lemma~\ref{l: cellular neighborhood: bd multi =>card(U^2)<N}, for which $\card\bigl(\cU^2_m(X)\bigr)\le\sup\bigl\{\card\bigl(\cU^2_l(Z)\bigr):l\in\N_0,Z\in\cD_l^{[n]}\bigr\}\le N$. Then since $f^{m-m_0}$ is injective on each $m$-chamber, $N(h,U)=N\bigl(f^{m-m_0},U^2_m(X)\bigr) \leq  N$. Thus, applying Theorem~\ref{t: BQS and UQR: modul char for QR} to the $\bigl(KL^{4n}\bigr)$-quasiregular map $h\:U\to V$ and using (\ref{e: M(h(Delta))}), we have
	\begin{equation}\label{e: UQR: MOD < fMOD}
		\begin{aligned}
			\modul(\Delta(E,F;D)) \leq  KL^{4n}\cdot N(h,U)\cdot\modul(h(\Delta(E,F;D)))
			 \leq \alpha_nNKL^{5n}\delta_0^{-n}.
		\end{aligned}
	\end{equation}
	Combining (\ref{e: UQR: MOD>min diam/dist}), (\ref{e: UQR: MOD(E,F,D)}), and (\ref{e: UQR: MOD < fMOD}), we obtain
	$c_n\log(1+{\diam E}/{\dist(E,F)}) \leq \alpha_nNKL^{5n}\delta_0^{-n}$, 
	and thus
	$\dist(E,F)/\diam E \geq  \bigl(\exp\bigl\{c_n^{-1}\alpha_nNKL^{5n}\delta_0^{-n}\bigr\}-1\bigr)^{-1}$.
	Since $\phi$ is $L$-bi-Lipschitz,
	\begin{equation*}
	\dist_{\dg}(A,B)/\diam_{\dg} A \geq  L^{-2}\bigl(\exp\bigl\{c_n^{-1}\alpha_nNKL^{5n}\delta_0^{-n}\bigr\}-1\bigr)^{-1}.
\end{equation*}
	Denote $\lambda_1\=L^{-2}\bigl(\exp\bigl\{c_n^{-1}\alpha_nNKL^{5n}\delta_0^{-n}\bigr\}-1\bigr)^{-1}$.
	By Lemma~\ref{l: cel Nbhd: U^1(X)}~(iii), for each $Y\in\cD_m^{[n]}$ satisfying $Y\cap X=\emptyset$,
	\begin{equation*}
	\begin{aligned}
		\frac{\dist_{\dg}(X,Y)}{\diam_{\dg} X} \geq \frac{\dist_{\dg}(X,\partial U^1_m(X))}{\diam_{\dg} X} \geq \frac{\dist_{\dg}(A,B)}{\diam_{\dg}A} \geq  \lambda_1.
	\end{aligned}
\end{equation*}
	Defining $\lambda\=\min\{\lambda_0,\,\lambda_1,\,1\}$, we obtain the constant desired.
\end{proof}

\begin{cor}\label{c: UQR+bd multi=>dist/diam Fl>}
	Let $f\:\mfd^n\to\mfd^n$, $n\geq 3$, be an expanding {\itCel} branched cover on a {\clcnR} $n$-manifold $(\mfd^n,\dg)$, and $\{\cD_m\}_{m\in\N_0}$ be a {\celSeq} of $f$.
    If $f$ is uniformly quasiregular, 
	then there exists a constant $\tlambda\in(0,1)$ such that for all $m\in\N_0$, $m$-vertex $p$, and $Y\in \cD_m^{[n]}$ with $Y\cap\cls{\flower_m(p)}=\emptyset$, we have $\dist\bigl(Y,\cls{\flower_m(p)}\bigr) \geq \tlambda\diam\cls{\flower_m(p)}$.
\end{cor}
\begin{proof}
	Let $\lambda\in(0,1]$ be a constant, given by Proposition~\ref{prop: UQR + bd mul (no need for Mkv) => dist/diam>}, such that for all $m\in\N_0$ and $X,\,Y\in\cD_m^{[n]}$ with $X\cap Y=\emptyset$, we have
	$\dist(X,Y) \geq \lambda \diam X$.
	Then set $\tlambda\=\min\bigl\{\lambda/4,\,\lambda^2/4\bigr\}=\lambda^2/4$.
	
	Let $p$ be an $m$-vertex, and $Y\in \cD_m^{[n]}$ satisfy $Y\cap\cls{\flower_m(p)}=\emptyset$. By Proposition~\ref{p: c-flower: structure}~(ii), we can choose $X,\,\hX\in\cD_m^{[n]}$ in a such a way that $p\in X\cap \hX$, $\dist(X,Y)=\dist\bigl(\cls{\flower_m(p)},Y\bigr)$, and $\diam \hX=\max\bigl\{\diam X':X'\in\cD_m^{[n]},\,p\in X'\bigr\}\geq  (1/2) \diam \cls{\flower_m(p)}$.
	Consider the following cases according to $\diam X$ and $\diam\hX$.
	
	\smallskip
	
	{\it Case~1.} $\diam X<(\lambda/2)\diam \hX$. By the fact that $p\in X\cap\hX$ and Proposition~\ref{prop: UQR + bd mul (no need for Mkv) => dist/diam>}, we have
    $\dist(X,Y)+\diam X \geq \dist\bigl(\hX,Y\bigr) \geq \lambda\diam \hX$,
	and thus, by the choice of $X$ and $\hX$,
    \begin{equation*}
        \dist\bigl(\cls{\flower_m(p)},Y\bigr)=\dist(X,Y) \geq (\lambda/2)\diam \hX \geq (\lambda/4)\diam \cls{\flower_m(p)}\geq \tlambda\diam \cls{\flower_m(p)}.
    \end{equation*}
	
	\smallskip
	
	{\it Case~2.}
    $\diam X \geq (\lambda/2)\diam \hX$. Then by Proposition~\ref{prop: UQR + bd mul (no need for Mkv) => dist/diam>} and the choice of $\hX$, we have
    \begin{equation*}
        \dist(X,Y) \geq \lambda\diam X \geq \bigl(\lambda^2/2\bigr)\diam \hX \geq \bigl(\lambda^2/4\bigr)\diam \cls{\flower_m(p)}\geq \tlambda\diam \cls{\flower_m(p)} . \qedhere
    \end{equation*}
\end{proof}

We are now ready to prove Theorem~\ref{t: cellular UQR => Lattes}.

\begin{proof}[Proof of Theorem~\ref{t: cellular UQR => Lattes}]
    We argue by showing that every point in $\mfd^n$ is conical (see Definition~\ref{d: conical points}). 
    
	Let $\{\cD_m\}_{m\in\N_0}$ be a {\celSeq} of $f$. 
    For each $m\in\N_0$ and each $m$-vertex $p$, we denote 
    \begin{equation*}
	   \begin{aligned}
       W_m(p)\= U^1_m\bigl(\cls{\flower_m(p)}\bigr)=U^2_m(p),
	   \end{aligned}
    \end{equation*}
    where the equality holds by Proposition~\ref{p: c-flower: structure}~(ii). The following properties are clear:
    \begin{itemize}
        \smallskip
        \item [(a)] $W_m(p)\subseteq B(x,4\mesh(\cD_m))$ for each $x\in W_m(p)$,
        \smallskip
        \item [(b)] $\cls{\flower_m(p)}\subseteq W_m(p)$ (by Lemma~\ref{l: cell Nbhd: U^i(G)}~(i)), and
        \smallskip
        \item [(c)] $f^k(W_{m}(p))\subseteq W_{m-k}\bigl(f^k(p)\bigr)$ (by Lemma~\ref{l: cel Nbhd: fwd invariance})
    \end{itemize}
    for all $m,\,k\in\N_0$ with $k\leq m$ and $m$-vertex $p$.
    
    By Lemma~\ref{l: cellular neighborhood: int(cX) - cls, comp, bd}~(ii) and Corollary~\ref{c: UQR+bd multi=>dist/diam Fl>}, there exists a constant $\tlambda>0$ such that for all $m\in\N_0$ and $m$-vertex $p$ satisfying $\partial W_m(p)\neq\emptyset$,
    \begin{equation}\label{e: dist(F(p),dW(p))>}
        \dist\bigl(\cls{\flower_{m}(p)},\partial W_{m}(p)\bigr)\geq \tlambda\diam\cls{\flower_{m}(p)}.   
    \end{equation}

    Let $\{(U_\alpha,\phi_\alpha)\}_{\alpha\in\sA}$ be an atlas such that the coordinate maps $\phi_\alpha\:U_\alpha\to\B^n$, $\alpha\in\sA$, are $L$-bi-Lipschitz for some fixed $L>1$.
    
	Fix an arbitrary point $x\in\mfd^n$ and an $L$-bi-Lipschitz chart $(U,\phi)$ in $\{(U_\alpha,\phi_\alpha)\}_{\alpha\in\sA}$ that contains $x$. Since $f$ is expanding, fix a sufficiently large $m_0\in\N$ and a small constant $\delta_0>0$ such that
    \begin{enumerate}[label=(\roman*),font=\rm]
    	\smallskip
        \item $B(x,5\mesh(\cD_m))\subseteq U$ for each $m \geq  m_0$;
        \smallskip
        \item $10\mesh(\cD_{m_0})$ is strictly less than the Lebesgue number of 
        $\{U_\alpha\}_{\alpha\in\sA}$; and
        \smallskip
        \item $2\delta_0$ is strictly less than the Lebesgue numbers of $\coverF(\cD_{m_0})$.
    \end{enumerate}
     
     \begin{figure}[h]
     	\includegraphics[scale=0.3]{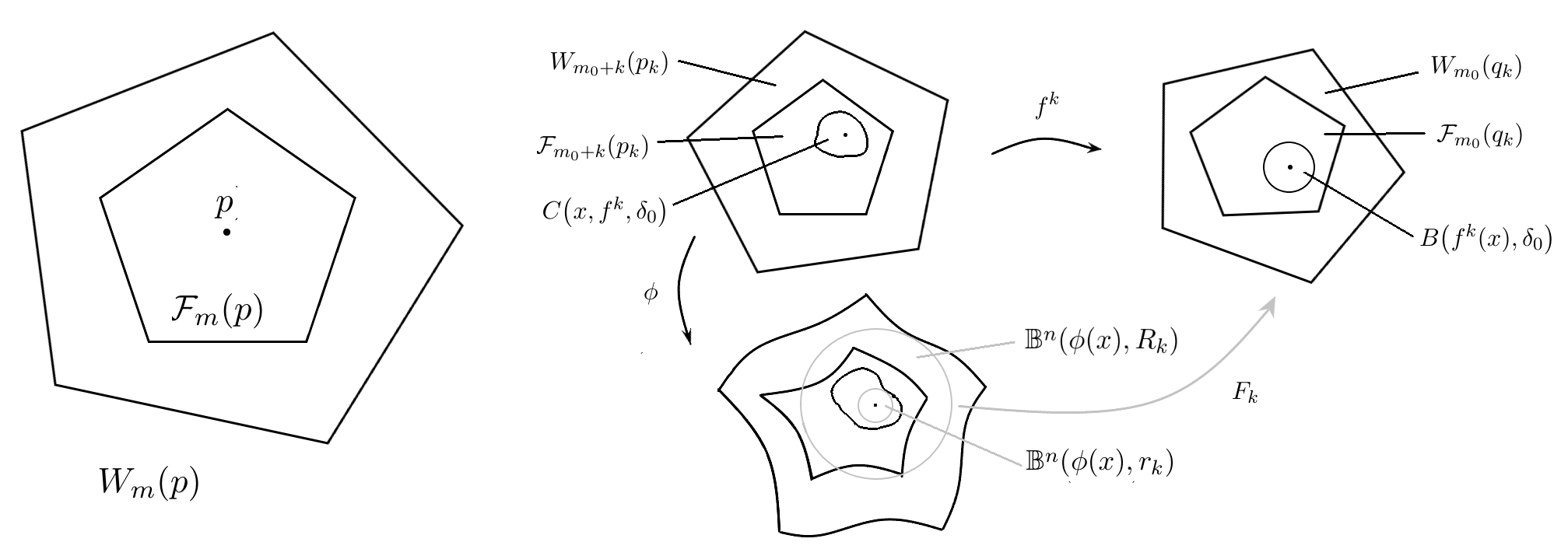}
        \caption{}
        \label{f: illustrate proof of main thm}
     \end{figure}
     
    Now fix arbitrary $k\in\N$.
    Choose an $m_0$-vertex $q_k$ such that $\flower_{m_0}(q_k)$ contains $\cls{B\bigl(f^k(x),\delta_0\bigr)}$, and, by Proposition~\ref{p: p-flowers: invariance}~(ii),
	an $(m_0+k)$-vertex $p_k\in f^{-k}(q_k)$ for which 
    $\flower_{m_0+k}(p_k)$ is the connected component of $f^{-k}(\flower_{m_0}(q_k))$ that contains $x$. 
    Let $C\bigl(x,f^k,\delta_0\bigr)$ be the connected component of $f^{-k}\bigl(B\bigl(f^k(x),\delta_0\bigr)\bigr)$ containing $x$. Then $\cls{C\bigl(x,f^k,\delta_0\bigr)}$ is contained in $\flower_{m_0+k}(p_k)$ since $\cls{B\bigl(f^k(x),\delta_0\bigr)}\subseteq \flower_{m_0}(q_k)$. Moreover, by the above properties~(a), (b) of $W_{m_0+k}(p_k)$ and (i) of $m_0$, we have
    \begin{equation*}
        \cls{C\bigl(x,f^k,\delta_0\bigr)}\subseteq\flower_{m_0+k}(p_k)\subseteq\cls{\flower_{m_0+k}(p_k)}\subseteq W_{m_0+k}(p_k)\subseteq\cls{W_{m_0+k}(p_k)}\subseteq U.
    \end{equation*}
	See Figure~\ref{f: illustrate proof of main thm} for an illustration of the construction above.
	
	Denote $r_k\=\dist\bigl(\phi(x),\partial\phi\bigl(C\bigl(x,f^k,\delta_0\bigr)\bigr)\bigr)$ and $R_k\=\dist(\phi(x),\partial\phi(W_{m_0+k}(p_k)))$. It is clear that $R_k-r_k\geq\dist\bigl(\cls{\phi\bigl(C\bigl(x,f^k,\delta_0\bigr)\bigr)},\partial\phi(W_{m_0+k}(p_k))\bigr)\geq \dist\bigl(\cls{\phi(\flower_{m_0+k}(p_k))},\,\partial \phi(W_{m_0+k}(p_k))\bigr)$.
    Thus, since $\cls{W_{m_0+k}(p_k)}\subseteq U$ guarantees that $\partial{W_{m_0+k}(p_k)}$ is non-empty and equal to $\phi^{-1}(\partial\phi(W_{m_0+k}(p_k)))$, we combine (\ref{e: dist(F(p),dW(p))>}) with the fact that $\phi$ is $L$-bi-Lipschitz to obtain
    \begin{equation*}
    	R_k-r_k \geq \dist\bigl(\cls{\phi(\flower_{m_0+k}(p_k))},\,\partial \phi(W_{m_0+k}(p_k))\bigr) \geq \tlambda L^{-2}\diam \bigl(\cls{\phi(\flower_{m_0+k}(p_k))}\bigr) \geq  \tlambda L^{-2} r_k.
    \end{equation*}
    
	Fix $0<\epsilon_0<\tlambda L^{-2}$. Then $\B^n(\phi(x),(1+\epsilon_0)r_k)\subseteq \phi(W_{m_0+k}(p_k))$. Define
	\begin{equation*}
	F_k\:\B^n\to f^k(W_{m_0+k}(p_k))\subseteq\mfd^n,\quad z\mapsto \bigl(f^k\circ\phi^{-1}\bigr)(\phi(x)+(1+\epsilon_0)r_k\cdot z).
    \end{equation*}
    By the choice of $r_k$, $\cls{\B^n(\phi(x),r_k)}$ meets $\partial \phi\bigl(C\bigl(x,f^k,\delta_0\bigr)\bigr)$, and we can choose a point $y_k\in\cls{\B^n(\phi(x),r_k)}\cap\partial \phi\bigl(C\bigl(x,f^k,\delta_0\bigr)\bigr)$. Since $C\bigl(x,f^k,\delta_0\bigr)$ is a connected component of $f^{-k}\bigl(B\bigl(f^k(x),\delta_0\bigr)\bigr)$, we have $f^k\bigl(\partial C\bigl(x,f^k,\delta_0\bigr)\bigr)\subseteq\partial B\bigl(f^k(x),\delta_0\bigr)$ (cf.~\cite[Chapter~I, Lemma~4.7]{Ri93}), and thus
    $\bigl(f^k\circ\phi^{-1}\bigr)(y_k)\in \partial B\bigl(f^k(x),\delta_0\bigr)$. 
    Denote 
	$z_k\=\frac{y_k-\phi(x)}{(1+\epsilon_0)r_k}\in\B^n$. Then
    \begin{equation}\label{e: dg(Fk(zk) Fk(0))=d/2}
    \dg(F_k(0),F_k(z_k))=\dg\bigl(f^k(x),f^k\bigl(\phi^{-1}(y_k)\bigr)\bigr)=\delta_0,    
    \end{equation}
    where $\dg$ is the Riemannian distance on $\mfd^n$.

	Consider the family $\{F_j:j\in\N\}$.
	Since $f$ is uniformly quasiregular, $\phi$ is bi-Lipschitz, and $z\mapsto (1+\epsilon_0)r_j\cdot z$ is a similarity transformation, there exists $K \geq  1$ such that all $F_j$, $j\in\N$, are $K$-quasiregular.
	
	Since $\cD_{m_0}$ is finite, we may fix an $m_0$-vertex $q$ and a subsequence $\{k_i\}_{i\in\N}$ of $\{k\}_{k\in\N}$ such that $q_{k_i}=q$, $i\in\N$. By the above properties~(a), (c) of $W_{m}(p)$ and~(ii) of $m_0$, there exists a bi-Lipschitz chart $(V,\psi)$ in $\{(U_\alpha,\phi_\alpha)\}_{\alpha\in\sA}$ such that for each $i\in\N$, $F_{k_i}(\B^n)\subseteq f^{k_i}(W_{m_0+k_i}(p_{k_i}))\subseteq W_{m_0}(q)\subseteq V$. Then it follows from Montel's theorem (cf.~\cite[Theorem~19.8.1]{IM01}) that, there is a subsequence of $\{F_{k_i}\}_{i\in\N}$ that converges locally uniformly in $\B^n$ to some $K$-quasiregular map $\Phi\:\B^n\to\mfd^n$ as $j\to+\infty$. Without loss of generality, we may assume $\{F_{k_i}\}_{i\in\N}$ itself converges locally uniformly in $\B^n$ to $\Phi\:\B^n\to\mfd^n$.

    Since, for each $i\in\N$, $z_{k_i}\in\cls{\B^n\bigl(0,\frac{1}{1+\epsilon_0}\bigr)}$, we can find a subsequence of $\{z_{k_i}\}_{i\in\N}$ converging to some $w^*\in\cls{\B^n\bigl(0,\frac{1}{1+\epsilon_0}\bigr)}$. Without loss of generality, we may assume $z_{k_i}\to w^*$ as $i\to+\infty$.
    Since $\{F_{k_i}\}_{i\in\N}$ converges locally uniformly in $\B^n$ to $\Phi$,
	there is a small open neighborhood $D$ of $w^*$ for which $F_{k_i}|_D$ converges uniformly to $\Phi|_D$ as $i\to+\infty$.
    
    Suppose $\Phi\equiv c\in\mfd^n$ is constant. Then when $i\in\N$ is sufficiently large, we have $\dg(F_{k_i}(0),c)<\delta_0/2$ and $\dg(c,F_{k_i}(z_{k_i}))<\delta_0/2$, which contradicts $\dg(F_{k_i}(0),F_{k_i}(z_{k_i}))=\delta_0$ (see (\ref{e: dg(Fk(zk) Fk(0))=d/2})). Hence, $\Phi$ is non-constant.

    Choose $\epsilon\in(0,1)$ for which $F_{k_i}|_{\epsilon\B^n}$ converges uniformly to $\Phi|_{\epsilon\B^n}$. Define $\rho_i\=(1+\epsilon_0)\epsilon r_{k_i}$. Then $\rho_i\to 0$ as $i\to+\infty$. Define $f_i\:\B^n\to\mfd^n$ by $f_i(z)\=F_{k_i}(\epsilon z)=\bigl(f^{k_i}\circ\phi^{-1}\bigr)(x+\rho_{i}z)$.
	Then $f_i$ converges uniformly to
	$\Psi\:\B^n\to\mfd^n$ given by $\Psi(z)\=\Phi(\epsilon z)$.
	
	Since $\Phi$ is a non-constant quasiregular map, it is discrete and open, which implies that $\Phi|_{\epsilon\B^n}$ and $\Psi$ cannot be constant. 
	Now we can see that $x$ is a conical point. By Corollary~\ref{c: conical: +measure conical=>Lattes} (see also \cite{MM03}), $f$ is a Latt\`es map. Since $f$ is expanding, $f$ is chaotic (cf.~Lemma~\ref{l: Lattes: top exact => chaotic}).
\end{proof}

\section{Metric theory of Latt\`es maps}
\label{sct: UEBQS <=> Lattes}
In this section, we investigate metrically defined classes of maps closely related to UQR maps, giving a characterization of (chaotic) Latt\`es maps on an {\orclcnR} $n$-manifold $\mfd^n$, $n \geq  3$, in terms of uniform expanding BQS maps (cf.~Definition~\ref{d: UEBQS}).

More precisely, we establish Theorem~\ref{tx: UeBQS-Lattes} via the following approach:
one of the implications in Theorem~\ref{tx: UeBQS-Lattes} is formulated as Theorem~\ref{t: UBQS=>Lattes}, and the converse follows from Proposition~\ref{p: CXC=>UBQS} and \cite[Theorem~4.4.3]{HP09}.
See~\cite[Section~4.4]{HP09} for a related discussion.

\subsection{Uniform expanding BQS maps are Latt\`es maps}\label{subsct: UEBQS=>Lattes}

Recall that BQS maps can be viewed as metric counterparts of quasiregular maps (cf.~Section~\ref{sct: intro} and Appendix~\ref{Ap: BQS=>QR}).
Our first observation is that orientation-preserving uniform expanding BQS maps are uniformly quasiregular.
\begin{lemma}\label{l: UEBQS => UQR}
    Let $f\:(\mfd^n,\dg)\to(\mfd^n,\dg)$, $n \geq 3$, be an orientation-preserving uniform expanding BQS map on an {\orclcnR} $n$-manifold. Then $f$ is uniformly quasiregular.
\end{lemma}
\begin{proof}
    Let $(\cU,\eta)$ be uniform expanding BQS data of $f$.
    Then for each $k\in\N$, $f^k$ is a discrete, open, and orientation-preserving \emph{local $\eta$-BQS map} (cf.~Appendix~\ref{Ap: BQS=>QR} for definition). By Theorem~\ref{thm: LiP19: BQS=>QR}, there exists $K \geq  1$, depending only on $\eta$, such that for each $k\in\N$, the map $f^k$ is $K$-quasiregular. Hence, $f$ is uniformly quasiregular.
\end{proof}

\begin{theorem}\label{t: UBQS=>Lattes}
	An orientation-preserving uniform expanding BQS map $f\:(\mfd^n,\dg)\to(\mfd^n,\dg)$, $n \geq 3$, on an {\orclcnR} $n$-manifold is a chaotic Latt\`es map.
\end{theorem}

We first establish a technical lemma.

\begin{lemma}\label{l: exp UBQS => non-zero distortion at a point}
	Let $f\:(\mfd^n,\dg)\to(\mfd^n,\dg)$, $n \geq 3$, be an orientation-preserving uniform expanding BQS map with data $(\cU,\eta)$ on an {\orclcnR} $n$-manifold. Then there exist constants $m_0\in\N$ and $\alpha,\,\delta,\,\tdelta>0$ with the following property: for each $x\in\mfd^n$ and each $m\in\N$, there exists $U\in (f^{m_0+m})^*\cU$ such that $B(x,2\alpha\diam U)\subseteq U$ and
	\begin{equation*}
	\begin{aligned}
		B\bigl(f^m(x),\tdelta\bigr)\subseteq f^m(B(x,\alpha\diam U))
		\subseteq f^m(B(x,2\alpha\diam U))\subseteq B(f^m(x),\delta)\subseteq f^m(U).
	\end{aligned}
\end{equation*}
\end{lemma}

\begin{proof}
	Let $f$ be a uniform expanding BQS map with data $(\cU,\eta)$. Denote $\cU_m\=(f^m)^*\cU$ for each $m\in\N_0$.
	Let $\delta_0$ be the Lebesgue number of $\cU_0$ and
    fix a sufficiently large $m_0\in\N$ such that
    \begin{itemize}
        \smallskip
        \item[(1)] for each $m \geq  m_0$, $\mesh(\cU_{m})<\delta_0/4$;
        \smallskip
        \item[(2)] for each $m \geq  m_0$, each element in $\cU_m$ is contained in some chart. 
    \end{itemize} 
	Let $\delta_{m_0}$ be the Lebesgue number of $\cU_{m_0}$ and fix a constant $\delta\in(0,\delta_{m_0}/4)$.
    Then fix constants
	\begin{equation*}
	\alpha\=\frac{1}{4}\eta^{-1}\biggl(\frac{\delta}{2\mesh(\cU_{m_0})}\biggr)
    \quad\text{ and }\quad
	\tdelta\=\frac{\min\{\diam W : W\in\cU_{m_0}\}}{\eta(2)\eta({1}/{\alpha})}.
    \end{equation*}
	It is clear that $\alpha$ and $\tdelta$ depend only on $\eta$, $\delta$, and $\cU_{m_0}$.
	
	Fix arbitrary $x\in\mfd^n$ and $m\in\N_0$. Denote $y\=f^m(x)$.
	Since $\delta<\delta_{m_0}/4$, there exists $V\in\cU_{m_0}$ such that $B(y,2\delta)\subseteq V$. Let $U\in\cU_{m_0+m}$ be the connected component of $f^{-m}(V)$ with $x\in U$, then $f^m(U)=V$ (cf.~Remark~\ref{r: forward inv of pullback}). By the choice of $m_0$, there exists $V_0\in\cU_0$ containing the closure $\cls{V}$ of $V$. Thus,
	$\cls{U}\subseteq f^{-m}\bigl(\cls{f^m(U)}\bigr)\subseteq f^{-m}(V_0).$ 
    Since $U$ and thus $\cls{U}$ are connected, there exists $U_0\in\cU_m$ such that $\cls{U}\subseteq U_0$ and $f^m(U_0)=V_0$.

First, we verify $B(x,2\alpha\diam U)\subseteq U$ and $
        f^m(B(x,2\alpha\diam U))\subseteq B(y,\delta)\subseteq V$.
	Let $C$ be the connected component of $f^{-m}(B(y,\delta))$ that contains $x$. As above, since $\cls{B(y,\delta)}\subseteq V$, we have $\cls{C}\subseteq U$, and then (since $U$ and $V$ are contained in charts by above property~(2) of $m_0$ above) it follows directly from~\cite[Chapter~I, Lemma~4.7]{Ri93} that $C$ is a domain such that $f^m(C)=B(y,\delta)$ and $f^m(\partial C)=\partial f^m(C)=\partial B(y,\delta)$. Set $r\=\dist(x,\partial C)$. Then $B(x,r)\subseteq C$ and $\cls{B(x,r)}$ meets $\partial C$, which, combined with $f^m(\partial C)=\partial B(y,\delta)$, implies $\diam\bigl(f^m\bigl(\cls{B(x,r)}\bigr)\bigr) \geq \delta$. 
    
    Since $f^m|_{U_0}\:U_0\to V_0$ is $\eta$-BQS, we have 
	\begin{equation*}
	\frac{\delta}{\diam \cls{V}}
	 \leq 
	\frac{\diam\bigl(f^m\bigl(\cls{B(x,r)}\bigr)\bigr)}{\diam\bigl(f^m\bigl(\cls{U}\bigr)\bigr)}
	 \leq 
	\eta\biggl(\frac{2r}{\diam \cls{U}}\biggr),
\end{equation*}
	and thus
	\begin{equation*}
	r \geq 
	\frac{\diam \cls{U}}{2}\eta^{-1}\biggl( \frac{\delta}{\diam \cls{V}} \biggr)
	 \geq 
	\frac{\diam \cls{U}}{2}\eta^{-1}\biggl(\frac{\delta}{2\mesh(\cU_{m_0})}\biggr)=2\alpha\diam \cls{U}.
\end{equation*}
	From the above discussion, we conclude that
    \begin{equation}\label{e: metric Lattes: B(x,2aU) subset U}
        B(x,2\alpha\diam U)\subseteq C\subseteq U \quad\text{and}\quad 
        f^m(B(x,2\alpha\diam U))\subseteq f^m(C)= B(y,\delta)\subseteq V.
    \end{equation}
	
	It remains to show $B\bigl(y,\tdelta\bigr)\subseteq f^m(B(x,\alpha\diam U))$. Denote $B_x\=B(x,\alpha\diam U)$.
	Set $s\=\dist(y,\partial f^m(B_x))$. Then $B(y,s)\subseteq f^m(B_x)$ and $\cls{B(y,s)}$ meets $\partial f^m(B_x)$. Pick $z\in\cls{B(y,s)}\cap\partial f^m(B_x)$, and let $\gamma_{yz}\:[0,1]\to \cls{B(y,s)}$ be a geodesic curve from $y$ to $z$. Since $U$ and $V=f^m(U)$ are contained in charts (see property~(2) of $m_0$), it follows directly from \cite[Lemma~9.20]{Vu88} that there exists a curve\footnote{Such a curve is called the ``maximal lift" of $\gamma_{yz}$ at $x$ (see e.g.~\cite{Ri93,Vu88}).} $\tgamma_{yz}:[0,c)\to U$, $0<c<1$, starting from $x$, such that $f^m\circ\tgamma_{yz}=\gamma_{yz}|_{[0,c)}$.
    By \cite[Lemma~9.20]{Vu88}, there are two cases as follows, according to which we define a curve $\tgamma$:
	\begin{enumerate}[label=(\roman*),font=\rm]
    	\smallskip
		\item $\tgamma_{yz}(t)\to x_1$ as $t\to c$ for some $x_1\in U$, in which case $c=1$ and $f^m(x_1)=\lim_{t\to 1}\gamma_{yz}(t)=z$. Define $\tgamma|_{[0,1)}\=\tgamma_{yz}$ and $\tgamma(1)\= x_1$. Note that $x_1\notin B_x$ since $z\notin f^m(B_x)$.
		\smallskip
		\item $\tgamma_{yz}(t)\to\partial U$ as $t\to c$. Choose $\epsilon\in(0,c)$ such that $\tgamma_{yz}(c-\epsilon)\notin\cls{B_x}$, and define $\tgamma\=\tgamma_{yz}|_{[0,c-\epsilon]}$.
	\end{enumerate}
	In neither case $\tgamma$ is contained in $B_x$, which means there exists a sub-curve $\gamma$ of $\tgamma$ (after re-parametrization) for which $\gamma(0)=x$, $\gamma(1)\in\partial B_x$, and $\gamma([0,1))\subseteq B_x$ (cf.~Lemma~\ref{l: there is a sub-curve}).
    Clearly $\diam \gamma \geq \alpha\diam U  \geq 2^{-1}\diam \cls{B_x}$.
	Thus, since $f^m|_{U_0}$ is $\eta$-BQS, we have
	\begin{align*}
	\frac{\diam\bigl(f^m\bigl(\cls{B_x}\bigr)\bigr)}{s}
	& = 
	\frac{\diam\bigl(f^m\bigl(\cls{B_x}\bigr)\bigr)}{\diam\gamma_{yz}}
	 \leq 
	\frac{\diam\bigl(f^m\bigl(\cls{B_x}\bigr)\bigr)}{\diam(f^m(\gamma))}
	 \leq 
	\eta\biggl(\frac{\diam \cls{B_x}}{\diam \gamma}\biggr) \leq \eta(2) \quad\text{ and} \\
    \frac{\diam V}{\diam(f^m(\gamma))}
	& = \frac{\diam\bigl(f^m\bigl(\cls{U}\bigr)\bigr)}{\diam\bigl(f^m(\gamma))}
	 \leq \eta\biggl(\frac{\diam \cls{U}}{\diam\gamma}\biggr)
	 \leq \eta\Bigl(\frac{1}{\alpha}\Bigr).
\end{align*}
	Combining the above inequalities, we get
	$s \geq \frac{\diam V}{\eta(2)\eta({1}/{\alpha})}=\tdelta$, 
	and thus
    $B\bigl(y,\tdelta\bigr)\subseteq B(y,s)\subseteq f^m(B(x,\alpha\diam U))$.
	This, together with (\ref{e: metric Lattes: B(x,2aU) subset U}), concludes that $U$ has the desired property.
\end{proof}

\begin{proof}[Proof of Theorem~\ref{t: UBQS=>Lattes}]
	By Lemma~\ref{l: UEBQS => UQR}, $f$ is uniformly quasiregular. Then by Corollary~\ref{c: conical: +measure conical=>Lattes} (see also \cite{MM03}), it suffices to show that every point in $\mfd^n$ is conical (cf.~Definition~\ref{d: conical points}). 
	
	Let $(\cU,\eta)$ be uniform expanding BQS data of $f$, and 
	$m_0,\,\alpha,\delta,\,\tdelta$ be constants given by Lemma~\ref{l: exp UBQS => non-zero distortion at a point}.
    
    Fix $x\in\mfd^n$, a constant $r>0$, and an $L$-bi-Lipschitz coordinate map $\phi\:B(x,r)\to\phi(B(x,r))\subseteq \R^n$, given by the exponential map, for a fixed $1 \leq  L<\sqrt{2}$. We also fix a constant $\lambda\in(L, 2/L)$.
	For each $m\in\N$, we use Lemma~\ref{l: exp UBQS => non-zero distortion at a point} to choose $U_m\in\cU_{m_0+m}$ with 
	\begin{equation}\label{e: distortion of UeBQS}
	\begin{aligned}
        B(x,2r_m)\subseteq &U_m \quad\text{and}\\
		B\bigl(f^m(x),\tdelta\bigr)\subseteq f^m(B(x,r_m))
		\subseteq &f^m(B(x,2r_m))
		\subseteq B(f^m(x),\delta)\subseteq f^m(U_m),
	\end{aligned}
\end{equation}
		where $r_m\=\alpha\diam U_m$. Since $f$ is expanding, $r_m\to 0$ as $m\to+\infty$,
        and we may fix $m_1\in\N$ such that for all $m \geq  m_1$, $U_m\subseteq B(x,r)$.
		
		Fix arbitrary $m \geq  m_1$. Denote $x_0\=\phi(x)$, $D_1\=\phi(B(x,r_m))$, and $D_2\=\phi(B(x,2r_m))$. Since $\phi$ is $L$-bi-Lipschitz and $\lambda\in(L,2/L)$, we have
        $D_1\subseteq\B^n(x_0,\lambda r_m)\subseteq\cls{\B^n(x_0,\lambda r_m)}\subseteq D_2$.
		
        Consider the map
        $F_m\:\cls{\B^n}\to \mfd^n$ given by $F_m(y)\=\bigl(f^m\circ\phi^{-1}\bigr)(x_0+\lambda r_m\cdot y)$. 
		Since $y\mapsto x_0+\lambda r_m \cdot y$ is a similarity transformation, $\phi$ is $L$-bi-Lipschitz, $\phi^{-1}(D_2)=B(x,2r_m)\subseteq U_m$, and $f^m|_{U_m}$ is $\eta$-BQS, we can find $\xi\:[0,+\infty)\to[0,+\infty)$, independent of $m$, such that $F_m$ is $\xi$-BQS. 
		
		Consider distinct points $y_1,\,y_2\in\cls{\B^n}$ and a geodesic line segment $\gamma$ in $\cls{\B^n}$ from $y_1$ to $y_2$. Then
		\begin{equation*}
		\begin{aligned}
			\frac{\dg(F_m(y_1),F_m(y_2))}{\diam_{\dg} \bigl(F_m\bigl(\cls{\B^n}\bigr)\bigr)} \leq \frac{\diam_{\dg}(F_m(\gamma))}{\diam_{\dg} \bigl(F_m\bigl(\cls{\B^n}\bigr)\bigr)} \leq \xi\biggl(\frac{\abs{y_1-y_2}}{\diam \cls{\B^n}}\biggr)
			 \leq \xi(\abs{y_1-y_2}/2).
		\end{aligned}
	\end{equation*}
		Thus, since $F_m(\cls{\B^n})\subseteq f^m\bigl(\phi^{-1}(D_2)\bigr)=f^m(B(x,2r_m))\subseteq f^m(U_m)$ (see~(\ref{e: distortion of UeBQS})), we have
		\begin{equation}\label{e: equicontinuous}
			\dg(F_m(y_1),F_m(y_2)) \leq \xi(\abs{y_1-y_2}/2)\mesh \cU_{m_0}.
		\end{equation}
	
	Consider the family $\{F_k\}_{k =  m_1}^{+\infty}$. 
	By~(\ref{e: equicontinuous}) and the fact that $(\mfd^n,\dg)$ is bounded, $\{F_k\}_{k = m_1}^{+\infty}$ is equicontinuous and uniformly bounded on $\cls{\B^n}$. It follows from the Arzel\`a--Ascoli theorem that there exists a subsequence $\bigl\{F_{k_j}\bigr\}_{j\in\N}$ that is uniformly convergent on $\cls{\B^n}$ to some $\Phi\:\cls{\B^n}\to\mfd^n$. Then $\bigl\{F_{k_j}|_{\B^n}\bigr\}_{j\in\N}$ converges uniformly to $\Phi|_{\B^n}$.
    Since, by Theorem~\ref{thm: LiP19: BQS=>QR}, there exists $K>1$ for which all $F_{k}|_{\B^n}$, $k\geqslant m_1$, are $K$-quasiregular, the limit map $\Phi|_{\B^n}$ is quasiregular (cf.~\cite[Chapter~VI, Theorem~8.6]{Ri93}).
	Moreover, $\Phi|_{\B^n}$ is non-constant, since, for each integer $k\geq m_1$, $B\bigl(f^k(x),\tdelta\bigr)\subseteq f^k(B(x,\alpha\diam U_k))\subseteq F_k(\B^n)$ (see~(\ref{e: distortion of UeBQS})). Therefore, $x$ is a conical point. By Corollary~\ref{c: conical: +measure conical=>Lattes}, $f$ is a Latt\`es map. Since $f$ is expanding, $f$ is chaotic (cf.~Lemma~\ref{l: Lattes: top exact => chaotic}).
\end{proof}

\subsection{Latt\`es maps are uniform expanding BQS maps}

To prove the converse of Theorem~\ref{t: UBQS=>Lattes}, we divide the arguments into two parts. First, a result of Ha\"issinsky and Pilgrim (\cite[Theorem~4.4.3]{HP09}) shows that chaotic Latt\`es maps are metric coarse expanding conformal (CXC) systems. Then we show that metric CXC systems are uniform expanding BQS maps. 

We recall the special case of metric CXC systems on closed Riemannian manifolds; for a more detailed discussion we refer the reader to \cite[Sections~2.2 and~2.5]{HP09}.

A finite branched cover $f\:(\mfd^n,\dg)\to(\mfd^n,\dg)$ with $\deg(f) \geq 2$ on a closed Riemannian $n$-manifold is a \defn{metric CXC system (for the global repellor $X=\mfd^n$)} if there exists a finite cover $\cU$ of $\mfd^n$ by connected open sets such that the covers $\cU_m\=(f^*)^m\cU$, ${m\in\N_0}$, satisfy the following conditions:
\begin{enumerate}
	\smallskip
	\item[(i)] Expansion Axiom (abbreviated as~[Expans]): 
	$\mesh(\cU_m)\to0$ as $m\to+\infty$.
	\smallskip
	\item[(ii)] Irreducibility Axiom (abbreviated as~[Irred]): For each $x\in \mfd^n$ and each neighborhood $W$ of $x$ in $\mfd^n$, there exists $m\in\N$ with $f^m(W)=\mfd^n$.
	\smallskip
	\item[(iii)] Degree Axiom (abbreviated as~[Deg]): The following set has a finite maximum:
    \begin{equation*}
        \bigl\{\deg\bigl(f^k|_{\widetilde{U}}\:\widetilde{U}\to U\bigr): U\in\mathcal{U}_m,\,\widetilde{U}\in\mathcal{U}_{m+k},\,m,\,k\in\N\bigr\}.
    \end{equation*}
	\item[(iv)] Roundness Distortion Axiom (abbreviated as~[Round]): There exist continuous increasing functions $\rho_+\:[1,+\infty)\to [1,+\infty)$ and $\rho_-\:[1,+\infty)\to [1,+\infty)$ such that 
	for all 
	$m,\,k \in \N_0$, $\tU\in\cU_{m+k}$, $\ty\in\tU$, and $U\=f^k\bigl(\tU\bigr)\in\cU_m$, $y\=f^k(\ty)\in U$, we have
	\begin{equation*}
		\begin{aligned}
			\Round\bigl(\tU,\ty\bigr)<\rho_-(\Round(U,y))\quad \text{ and } \quad
			\Round(U,y)<\rho_+\bigl(\Round\bigl(\tU,\ty\bigr)\bigr),
		\end{aligned}
	\end{equation*}
	where 
	$
	\Round(A,a)\={\sup\{\abs{x-a}:x\in A\}}/{\sup\{r>0:B(a,r)\subseteq A\}}
	$
	for an open subset $A$ and $a\in A$.

	\smallskip
	\item[(v)] Diameter Distortion Axiom (abbreviated as~[Diam]): There exist increasing homeomorphisms 
	$\delta_+\:[0,1]\to[0,1]$ and $\delta_-\:[0,1]\to[0,1]$ such that, for all $m_0,\,m_1,\,k\in\N_0$, $\tU\in\cU_{m_0+k}$, $\tU'\in\cU_{m_1+k}$ with $\tU'\subseteq\tU$, and $U\=f^k\bigl(\tU\bigr)\in\cU_{m_0}$, $U'\=f^k\bigl(\tU'\bigr)\in\cU_{m_1}$, we have
	\begin{equation*}
		\begin{aligned}
			{\diam \tU'}/{\diam \tU}<\delta_-({\diam U'}/{\diam U})\quad\text{ and }\quad
			{\diam U'}/{\diam U}<\delta_+\bigl({\diam \tU'}/{\diam \tU}\bigr).
		\end{aligned}
	\end{equation*}
\end{enumerate}

We record some properties of metric CXC systems; see \cite[Propositions~2.6.2,~2.6.4, and~2.6.5]{HP09}.

\begin{lemma}[Properties of CXC systems]\label{l: CXC metric properties}
	Let $f\:(\mfd^n,\dg)\to(\mfd^n,\dg)$ be a metric CXC system and $\cU$ be a finite cover of $\mfd^n$ by connected open subsets such that the covers $\cU_m\=(f^*)^m\cU$, ${m\in\N_0}$, satisfy {\rm [Expans], [Irred], [Deg], [Round], and [Diam]}. Then there exist constants $K>1$, $C>1$, $C'>0$, and $\theta\in(0,1)$ with the following properties:
	\begin{enumerate}[label=(\roman*),font=\rm]
    	\smallskip
		\item For each $x\in\mfd^n$ and each $m\in\N_0$, there exist $r>0$ and $U\in\cU_{m}$ such that 
		$B(x,r)\subseteq U\subseteq \cls{B(x,Kr)}$.
		\smallskip
		\item For all $m\in\N_0$, $U\in\cU_m$, $U'\in\cU_{m+1}$, if $U\cap U'\neq\emptyset$, then
		$
		C^{-1} \leq \diam U'/\diam U \leq  C.
		$
		\smallskip
		\item For all $m,\,k\in\N_0$, $U\in\cU_m$, $U'\in\cU_{m+k}$, if $U\cap U'\neq\emptyset$, then
		$
		\diam U' \leq  C'\theta^k\diam U.
		$
	\end{enumerate}
\end{lemma}

By \cite[Theorem~4.4.3]{HP09}, Latt\`es maps are metric CXC systems. To show the converse of Theorem~\ref{t: UBQS=>Lattes}, it suffices to prove the following.

\begin{prop}\label{p: CXC=>UBQS}
	A metric CXC system $f\:(\mfd^n,\dg)\to(\mfd^n,\dg)$, $n \geq 3$, on a {\clcnR} $n$-manifold $(\mfd^n,g)$ is a uniform expanding BQS map.
\end{prop}
\begin{proof}
	Let $f\:(\mfd^n,\dg)\to(\mfd^n,\dg)$ be a metric CXC system and $\cU$ be a finite cover of $\mfd^n$ by connected open subsets such that $\{\cU_m\}_{m\in\N_0}$, where $\cU_m\=(f^*)^m\cU$, satisfies {[Expans], [Irred], [Deg], [Round], and [Diam]}. 
    Let $K>1$, $C>1$, $C'>0$ and $\theta\in(0,1)$ be constants given by Lemma~\ref{l: CXC metric properties}.

    We estimate the sizes of continua using open sets in $\{\cU_m\}_{m\in\N_0}$.
    Suppose $A$ is a continuum and $U\in\cU_m$ is an $(f,\cU)$-approximation of $A$ (see Definition~\ref{d: approximate}). Let $x\in A$ be a point.
    By Lemma~\ref{l: CXC metric properties}~(i), there are $r>0$ and $U'\in\cU_{m+1}$ such that $B(x,r)\subseteq U'\subseteq \cls{B(x,Kr)}$.
    Since $U$ is an $(f,\cU)$-approximation of $A$, we have $A\subseteq U$, $A\nsubseteq U'$, and thus $r \leq \diam A \leq \diam U$.
    This, combined with $2Kr \geq\diam U'>C^{-1}\diam U$ (by Lemma~\ref{l: CXC metric properties}~(ii)), yields
    \begin{equation}\label{e: CXC=>UBQS: A and U_A comparable}
		(2KC)^{-1}\diam U \leq \diam A \leq \diam U.
	\end{equation}

    Now we show that $f$ is uniform expanding BQS.
    Fix $m\in\N$ and $U\in\cU_m$. Let $E,\,F\subseteq U$ be intersecting continua. Since $E,\,F\subseteq U$ and $\mesh(\cU_i)\to+\infty$ as $i\to+\infty$, there exist $k,\,l\in\N_0$ and $U_E\in\cU_{m+k}$, $U_F\in\cU_{m+l}$ such that $U_E$ and $U_F$ are $(f,\cU)$-approximations of $E$ and $F$, respectively. By Remark~\ref{r: approximate: iterate}, $f^m(U_E)$ and $f^m(U_F)$ are $(f,\cU)$-approximations of $f^m(E)$ and $f^m(F)$, respectively. In what follows we give an estimate of $\diam(f^m(U_E))/\diam(f^m(U_F))$.

\smallskip
    
	{\it Case~1.} $k>l$. Since $E$ intersects $F$, $U_E\cap U_F\neq\emptyset$ and $f^m(U_E)\cap f^m(U_F)\neq\emptyset$. Then by~Lemma~\ref{l: CXC metric properties}~(ii) and (iii), $\frac{\diam U_E}{\diam U_F} \geq  C^{-(k-l)}$ and $\frac{\diam(f^m(U_E))}{\diam(f^m(U_F))} \leq  C'\theta^{k-l}$,
	which yield
	\begin{equation*}
	{\diam(f^m(U_E))}/{\diam(f^m(U_F))} \leq  C'({\diam U_E}/{\diam U_F})^{\log_{1/C}\theta}.
\end{equation*}

\smallskip
	
	{\it Case~2.} $k \leq  l$. As above, $\frac{\diam(f^m(U_E))}{\diam(f^m(U_F))} \leq  C^{(l-k)+2}$ and 
	$
	\frac{\diam U_E}{\diam U_F} \geq \frac{1}{C'}\theta^{-(l-k)}.
	$
	Then
	\begin{equation*}
	{\diam(f^m(U_E))}/{\diam(f^m(U_F))} \leq  C^2(C'{\diam U_E}/{\diam U_F})^{\log_{1/\theta}C}.
\end{equation*}
	
	Define a function $\xi\:[0,+\infty)\to[0,+\infty)$ by
	$\xi(t)\=\max\bigl\{C't^{\log_{1/C}\theta},\,C^2(C't)^{\log_{1/\theta}C}\bigr\}$.
	Then 
	\begin{equation*}
	{\diam(f^m(U_E))}/{\diam(f^m(U_F))} \leq \xi({\diam U_E}/{\diam U_F}).
\end{equation*}
	This, combined with (\ref{e: CXC=>UBQS: A and U_A comparable}), yields
	$\diam(f^m(E)) \leq \xi(2KC\diam E/\diam F)\cdot 2KC\diam(f^m(F))$.
	Consider the function $\eta\:[0,+\infty)\to[0,+\infty)$ given by $\eta(t)\=2KC\cdot\xi(2KCt)$. Then $\eta$ is a homeomorphism for which $f^m|_U$ is $\eta$-BQS.
	Since $m\in\N$ and $U\in\cU_m$ are arbitrary, $f$ is a uniform expanding BQS map with data $(\cU,\eta)$.
\end{proof}

\begin{rem}
We find it interesting that the proofs of Lemma~\ref{l: CXC metric properties} and Proposition~\ref{p: CXC=>UBQS} do not rely on [Deg] axiom. Hence, it may not require \emph{a priori} the degree axiom [Deg] in the result \cite[Theorem~4.4.4]{HP09} of Ha\"issinsky and Pilgrim. More precisely, by combining Proposition~\ref{p: CXC=>UBQS} with Theorem~\ref{t: UBQS=>Lattes}, we have that an orientation-preserving branched cover $f\:\mfd^n\to\mfd^n$ satisfying {[Expans], [Irred], [Deg], [Round], and [Diam]}, on an {\orclcnR} $n$-manifold $\mfd^n$ of dimension $n\ge 3$, is a Latt\`es map.
\end{rem}

\subsection{Proofs of the main theorems} 
We conclude this section by collecting the arguments for the proofs of the main theorems stated in the introduction.

\begin{proof}[Proof of Theorem~\ref{tx: UeBQS-Lattes}]
Let $f\:(\mfd^n,\dg)\to(\mfd^n,\dg)$, $n\ge 3$, be an orientation-preserving branched cover on an {\orclcnR} $n$-manifold.
If $f$ is uniform expanding BQS, then by Theorem~\ref{t: UBQS=>Lattes}, $f$ is a chaotic Latt\`es map. The converse implication follows from Proposition~\ref{p: CXC=>UBQS} and the fact that chaotic Latt\`es maps on $\mfd^n$ are metric CXC systems (see~\cite[Theorem~4.4.3]{HP09}).
\end{proof}

\begin{proof}[Proof of Theorem~\ref{tx: QS-UEBQS-Lattes-UQR}]
Let $f\:\mfd^n\to\mfd^n$, $n\ge 3$, be an orientation-preserving expanding {\itCel} branched cover on an {\orclcnR} $n$-manifold.
The equivalence of~\ref{item:UQR-1} and~\ref{item:UQR-2} follows from Theorem~\ref{t: CBC: QS sph <=> UBQS}.
The equivalence of~\ref{item:UQR-3} and~\ref{item:UQR-4} follows from Theorem~\ref{t: cellular UQR => Lattes}.
The equivalence of~\ref{item:UQR-2},~\ref{item:UQR-4}, and~\ref{item:UQR-5} is obtained by combining Theorem~\ref{tx: UeBQS-Lattes} with the fact that metric CXC systems on $\mfd^n$ coincide with chaotic Latt\`es maps (see~\cite[Theorems~4.3.3 and~4.3.4]{HP09}). 
\end{proof}

\begin{proof}[Proofs of Theorems~\ref{tx: QS uniformization of visual metric} and \ref{tx: UQR = Lattes}]
Let $f\:\mfd^n\to\mfd^n$, $n\ge 3$, be an orientation-preserving expanding Thurston-type map on an {\orclcnR} $n$-manifold.
By Theorem~\ref{t: CPCF=>encoding sequence}, $f$ is an {\itCel} branched cover.
Then Theorem~\ref{tx: QS uniformization of visual metric} follows from \ref{item:UQR-1} $\Leftrightarrow$ \ref{item:UQR-2} in Theorem~\ref{tx: QS-UEBQS-Lattes-UQR}, and Theorem~\ref{tx: UQR = Lattes} follows from \ref{item:UQR-3} $\Leftrightarrow$ \ref{item:UQR-4} in Theorem~\ref{tx: QS-UEBQS-Lattes-UQR}.
\end{proof}

\appendix

\section{~}\label{Ap: QR and visual}
In this appendix, we collect various facts on quasiregular mappings and visual metrics.

\subsection{Quasiregular maps}
Here, we provide more information on {quasiregular maps}, including characterizations via different approaches, as well as the notion of conical points of UQR maps; we refer the reader to \cite{GW16,LiP19,MM03,Ri93,Vu88}.

\subsubsection{Geometric definition}\label{Ap: moduli}
	We record a geometric characterization for quasiregular maps using {moduli of curve families}; see e.g.~\cite{Ri93,Vu88} for detailed discussions.
	
	Let $\fX$ be a topological space. A continuous map $\gamma\:I\to \fX$, where $I\subseteq\R$ is an interval, is called a \defn{curve} in $\fX$. When there is no ambiguity, we denote the image $\gamma(I)$ by $\gamma$. For subsets $E,\, F,\, G\subseteq \fX$, we denote by $\Delta(E,F;G)$ the family of non-constant curves defined on closed intervals joining $E$ and $F$ in $G$, i.e., non-constant curves $\gamma\:[a,b]\to \fX$, $a<b$, with $\gamma(a)\in E$, $\gamma(b)\in F$, and $\gamma((a,b))\subseteq G$. In particular, when $G=\fX$, we write $\Delta(E,F)\=\Delta(E,F;G)$. For curves $\gamma\:I\to \fX$ and $\gamma'\:I'\to \fX$, we call $\gamma'$ a \defn{sub-curve} of $\gamma$ if $I'\subseteq I$ and $\gamma'=\gamma|_{I'}$. 
    
    We recall the following elementary fact regarding sub-curves.
	\begin{lemma}\label{l: there is a sub-curve}
		Let $G$ be an open subset of a topological space $\fX$, $D$ be a subset of $G$, and $\gamma\:[0,1]\to \fX$ be a curve in $\Delta(D, \fX\smallsetminus G)$. Then $\gamma'\=\gamma|_{[0,a]}$, where $a\=\inf \gamma^{-1}(\fX\smallsetminus G)$, is a sub-curve of $\gamma$ in $\Delta(D,\partial G; G)$.
	\end{lemma}

	For a curve $\gamma$ in $\R^n$, the length of $\gamma$ is defined in the usual way and denoted by $\length(\gamma)$. We call a curve $\gamma$ \defn{rectifiable} if $\length(\gamma)<+\infty$.
	Let $\Gamma$ be a family of curves in $\R^n$. A Borel function $\rho\:\R^n\to[0,+\infty]$ is called \defn{admissible} for $\Gamma$ if $\int_\gamma\!\rho\,\di s \geq  1$ for each rectifiable curve $\gamma\in\Gamma$. For each $p\in[1,+\infty)$, the \defn{$p$-modulus} of $\Gamma$ is defined as
	\begin{equation*}
	\modul_p(\Gamma)\=\inf_\rho\int_{\R^n}\!\rho^p\,\di m,
	\end{equation*}
	where $\rho$ ranges over all admissible Borel functions for $\Gamma$. Here and in what follows, $m$ is the $n$-dimensional Lebesgue measure on $\R^n$.
	With the dimension $n$ understood, the $n$-modulus $\modul_n(\Gamma)$ of $\Gamma$ is usually abbreviated as \defn{modulus} and denoted by $\modul(\Gamma)$.

	Moduli of curve families yield the following characterization of quasiregular maps.
    Here the notation $N(f,A)$ is defined in (\ref{e: N(f,A)}). 
	\begin{theorem}[{\cite[Theorem~II.6.7]{Ri93}}]\label{t: BQS and UQR: modul char for QR}
		Let $f\:G\to\R^n$ be a non-constant continuous map on a domain $G\subseteq\R^n$ and $K\in[1,+\infty)$. Then $f$ is $K$-quasiregular map if and only if the following statements are true:
		\begin{enumerate}[label=(\roman*),font=\rm]
        	\smallskip
			\item $f$ is orientation-preserving, discrete, and open.
			\smallskip
			\item For each Borel set $A\subseteq G$ with $N(f,A)<+\infty$, each curve family $\Gamma$ in $A$, and the image $f\Gamma\=\{f\circ\gamma:\gamma\in\Gamma\}$ of $\Gamma$, we have 
			$\modul(\Gamma) \leq  K\cdot N(f,A)\cdot\modul(f\Gamma)$.
		\end{enumerate}
	\end{theorem}

\subsubsection{Metric counterpart}\label{Ap: BQS=>QR}
In \cite{LiP19}, Lindquist and the second-named author gave a characterization of quasiregular maps in terms of branched quasisymmetric maps. 
One of their results asserts that local branched quasisymmetric maps are quasiregular, quantitatively.

Recall that a continuous map $f\:\fX\to \fY$ between metric spaces is a \defn{local $\eta$-branched quasisymmetric} (abbreviated as~\defn{local $\eta$-BQS}) \defn{map} for a homeomorphism $\eta\:[0,+\infty)\to[0,+\infty)$ if there exists $r>0$ such that  for each $x\in \fX$, $f|_{B(x,r)}$ is an $\eta$-BQS map.
\begin{theorem}[{\cite[Theorem~A.10]{LiP19}}]\label{thm: LiP19: BQS=>QR}
    Let $\mfd^n$ and $\cN^n$ be oriented Riemannian $n$-manifolds and $f\:\mfd^n\to\cN^n$ a discrete, open, and orientation-preserving local $\eta$-BQS map. Then $f$ is $K$-quasiregular for some $K \geq 1$ depending only on $\eta$.
\end{theorem}
We record a self-contained proof of Theorem~\ref{thm: LiP19: BQS=>QR} for the convenience of the reader (cf.~the proof of \cite[Theorem~A.10]{LiP19}).
\begin{proof}
	By localizing the map with $2$-bi-Lipschitz charts, it suffices to consider an $\eta$-branched quasisymmetric map $f\:\B^n\to\R^n$.
	
	For $x\in\B^n$ and $r>0$, we denote by $U(x,f,r)$ the connected component of $f^{-1}(\B^n(f(x),r))$ that contains $x$, and denote $L^*(x,f,r)\=\sup_{z\in\partial U(x,f,r)}\abs{x-z}$ and $l^*(x,f,r)\=\inf_{z\in\partial U(x,f,r)}\abs{x-z}$.
	
	Fix arbitrary $x\in\B^n$. Let $r_x > 0$ be a small radius such that for each $r \in(0, r_x)$, $f(U(x,f,r))=\B(f(x),r)$ and $f(\partial U(x,f,r))=\partial f(U(x,f,r))=\partial \B^n(f(x),r)$ (cf.~\cite[Lemma~9.14]{Vu88}). 
    
	Fix arbitrary $r\in(0,r_x)$ and points $y,\,z\in\partial U(x, f, r)$ for which $\abs{y-x} = l^*(x, f, r)$ and $\abs{z-x}=L^*(x, f, r)$.
	We use \cite[Lemma~2.7]{MRV69} (see also~\cite[Lemma~2.6]{LiP19}) to conclude that there exists a continuum $E \subseteq \cls{U(x, f, r)}$ containing $z$ and $x$ for which $f(E)$ is the line segment $[f(z), f(x)]$ containing $f(z)$ and $f(x)$. Since $f(\partial U(x,f,r))=\partial\B^n(f(x),r)$, we have $f(y)\in\partial\B^n(f(x),r)$ and thus $\diam(f([x,y])) \geq  r$. Combining this with $\diam E \geq \abs{x-z}$, $f(E)=[f(z),f(x)]$, and the fact that $f$ is $\eta$-BQS, we have
    \begin{equation*}
    \begin{aligned}
          r\leq\diam(f[y,x])
          &\leq  \eta(\diam[y,x]/\diam E) \cdot \diam(f(E))  \\
          &\leq  \eta(\abs{y-x}/\abs{z-x}) \cdot \diam[f(z),f(x)]
          =\eta(\abs{y-x}/\abs{z-x}) \cdot r.
    \end{aligned}
    \end{equation*}
    Thus, $\eta(\abs{y-x}/\abs{z-x}) \geq 1$ and $L^*(x,f,r)/l^*(x,f,r)=\abs{z-x}/\abs{y-x}\leq 1/\eta^{-1}(1)$.
	Consequently, $H^*(x,f)\=\limsup_{r\to0}L^*(x,f,r)/l^*(x,f,r)\leq  1/\eta^{-1}(1)$. Therefore, since $x$ is arbitrary, $f$ is $\bigl(1/\eta^{-1}(1)\bigr)^n$-quasiregular (cf.~\cite[Theorem~4.14]{MRV69}).
\end{proof}

\subsubsection{Conical points}
\label{Ap: conical pt of UQR}

We revisit a rigidity result due to Martin and Mayer~\cite{MM03} for UQR maps on $n$-spheres, $n\geq 3$, showing that UQR maps having a set of conical points with positive measure are Latt\`es maps. As the discussion below is analogous to \cite{MM03}, readers familiar with the argument of Martin and Mayer may want to skip this part.

\begin{definition}[Conical points]\label{d: conical points}
For a uniformly quasiregular map $f\:\mfd^n\to\mfd^n$ on a {\clcnR} $n$-manifold, a point $p\in\mfd^n$ is a \defn{conical point} if there exists a bi-Lipschitz smooth chart $(U,\phi)$, a subsequence $\{k_j\}_{j\in\N}$ of $\{k\}_{k\in\N}$, and a sequence of sizes $\{\rho_j\}_{j\in\N}$ with $\rho_j\to0$ as $j\to+\infty$ such that the family 
\begin{equation*}
\Psi_j\:\B^n\to\mfd^n,\quad x\mapsto \bigl(f^{k_j}\circ\phi^{-1}\bigr)(\phi(p)+\rho_j x)
\end{equation*}
converges uniformly to a non-constant quasiregular map $\Psi\:\B^n\to\mfd^n$.
\end{definition}

We formulate the following manifold version of \cite[Theorem~6.1]{MM03}.
\begin{theorem}\label{t: Ap: cont in measure at conical=>Latt}
	Let $f\:\mfd^n\to\mfd^n$, $n\ge 3$, be a uniformly quasiregular map on a closed Riemannian $n$-manifold, and $\mu$ be an $f$-invariant conformal structure on $\mfd^n$. If $\mu$ is continuous in measure at a conical point $p_0\in\mfd^n$, then $f$ is a Latt\`es map.
\end{theorem}

First, recall that a \defn{(bounded measurable) conformal structure} on a Riemannian $n$-manifold $\mfd^n$ is a bounded measurable map $\mu\:\mfd^n\to S(n)$, where $S(n)$ is the space of positive-definite symmetric matrices of determinant $1$, equipped with a Riemannian metric; see~\cite[Section~20.1]{IM01} for a detailed discussion. The \defn{pullback} of a conformal structure $\mu$ under $f\in W_{\loc}^{1,n}(\mfd^n)$ is 
\begin{equation*}
    f^*\mu(x)=J_f(x)^{-2/n}(Df(x))^{\operatorname{t}}\mu(f(x))Df(x) \quad\text{a.e.\ } x\in\mfd^n.
\end{equation*}
Here $(Df(x))^{\operatorname{t}}$ denotes the transpose of $(Df(x))^{\operatorname{t}}$. We say that $\mu$ is \defn{invariant under $f\in W_\loc^{1,n}(\mfd^n)$ (or $f$-invariant)} if $f^*\mu=\mu$. Each uniformly quasiregular map $f\:\mfd^n\to\mfd^n$ admits an $f$-invariant conformal structure (see~\cite[Theorem~4.1]{Ka21} or \cite[Theorem~5.1]{IM96}) and the converse is also true (cf. the discussion in \cite[Section~2]{IM96}).

We also recall that a map $f\:\mfd^n\to(\fX,d)$ between a Riemannian $n$-manifold and a metric space is \defn{continuous in measure at a point $p_0\in\mfd^n$} if, for each $\epsilon>0$,
\begin{equation*}
\frac{m(B(p_0,\delta)\cap E_\epsilon(p_0))}{m(B(p_0,\delta))}\to 0\quad\text{as }\delta\to 0,
\end{equation*}
where $m$ is the $n$-dimensional Lebesgue measure and $E_\epsilon(p_0)\=\{p\in\mfd^n:d(f(p),f(p_0))>\epsilon\}$. 

\begin{rem}\label{r: continuity in measure}
It is easy to check that $f\:\mfd^n\to(\fX,d)$ is continuous in measure at $p_0\in\mfd^n$ if and only if for each smooth chart $(U,\phi)$, $f\circ\phi^{-1}$ is continuous in measure at $\phi(p_0)\in\R^n$.
\end{rem}

Since a conformal structure $\mu$ on $\mfd^n$ is measurable, it is continuous in measure almost everywhere (cf.~\cite[Theorem~2.9.13]{Fe69}). Thus, Theorem~\ref{t: Ap: cont in measure at conical=>Latt} immediately leads to the following result, as a manifold version of \cite[Theorem~1.3]{MM03}.
\begin{cor}\label{c: conical: +measure conical=>Lattes}
	Let $\mfd^n$ be a closed Riemannian $n$-manifold with $n \geq 3$ and $f\:\mfd^n\to\mfd^n$ be a uniformly quasiregular map. If the set of conical points of $f$ has a positive measure, then $f$ is a Latt\`es map.
\end{cor}

\subsection{Visual metrics}\label{Ap: visual metric}

In what follows, we discuss the visual metrics for expanding {\itCel} branched covers introduced in Subsection~\ref{subsct: visual metrics}.
Our class of visual metrics is in the spirit of the seminal works of Bonk--Meyer \cite{BM17} and Ha\"isinsky--Pilgrim \cite{HP09}, where visual metrics with respect to dynamics on $2$-dimensional spheres and more general topological spaces, respectively, are introduced. The discussion in this subsection is analogous to \cite[Chapter~8]{BM17}.

\subsubsection{The separation level $m_{f,\cD}$} \label{Ap: m(x,y)}
The function $m_{f,\cD_0}$, defined by (\ref{e: the constant m(x,y)}) for a {\celSeq} $\{\cD_m\}_{m\in\N_0}$ of an expanding {\itCel} map $f$, plays a role similar to the Gromov products on Gromov hyperbolic spaces, and yields the definition of visual metrics for $f$. We give some elementary properties of $m_{f,\cD_0}$.

\begin{lemma}\label{l: m(x,y) property: iteration}
	Let $f\:\mfd^n\to\mfd^n$ be an {\itCel} map on a {\clcntop} $n$-manifold and $\{\cD_m\}_{m\in\N_0}$ be a {\celSeq} of $f$. Then for all $x,y\in\mfd^n$, we have $m_{f,\cD_0}(f(x),f(y)) \geq  m_{f,\cD_0}(x,y)-1$.
\end{lemma}
\begin{proof}
	Suppose $x,\, y\in\mfd^n$ satisfy $m_{f,\cD_0}(f(x),f(y)) <  m_{f,\cD_0}(x,y)-1$. This implies that $m\=m_{f,\cD_0}(x,y)>1$, and thus, by definition (cf.~(\ref{e: the constant m(x,y)})), there exists non-disjoint $m$-chambers $X,\,Y$ with $x\in X$ and $y\in Y$. Then $f(X)$ and $f(Y)$ are non-disjoint $(m-1)$-chambers containing $f(x)$ and $f(y)$, respectively. This implies $m_{f,\cD_0}(f(x),f(y))\ge m-1$, which is a contradiction.
\end{proof}

When $f$ is an expanding {\itCel} branched cover, we have the following properties.
\begin{lemma}\label{l: m(x,y) properties: hyperbolic}
	Let $f\:\mfd^n\to\mfd^n$ be an expanding {\itCel} branched cover on a {\clcnmtr} $n$-manifold, and let $\{\cD_m\}_{m\in\N_0}$ and $\{\cC_m\}_{m\in\N_0}$ be {\celSeq}s of $f$. Then the following statements are true: 
	\begin{enumerate}[label=(\roman*),font=\rm]
        \smallskip
		\item There exists $k_0\in\N$, depending only on $\{\cD_m\}_{m\in\N_0}$, such that for all $x,y,z\in\mfd^n$,
		\begin{equation*}
			\min\{m_{f,\cD_0}(x,z),\,m_{f,\cD_0}(z,y)\} \leq  m_{f,\cD_0}(x,y)+k_0.
		\end{equation*}
		\item There exists $k_1\in\N$, depending only on $\{\cC_m\}_{m\in\N_0}$ and $\{\cD_m\}_{m\in\N_0}$, such that for all $x,\,y\in\mfd^n$,
	       \begin{equation*}
	            m_{f,\cD_0}(x,y)-k_1 \leq  m_{f,\cC_0}(x,y) \leq  m_{f,\cD_0}(x,y)+k_1.
        	 \end{equation*}
	\end{enumerate}
\end{lemma}
\begin{proof}
	(i) By Lemma~\ref{l: expansion: J_m--->+infty}, fix $k_0\in\N$ such that $J\bigl(\bigcup_{i \geq  k_0}\cD_i,\cD_0\bigr)> 10$. To show that $k_0$ is the desired constant, we argue by contradiction and assume that $x,y,z\in\mfd^n$ satisfy
	\begin{equation*}
	\min\{m_{f,\cD_0}(x,z),\,m_{f,\cD_0}(z,y)\}> m_{f,\cD_0}(x,y)+k_0.
\end{equation*}
	Denote $m\=m_{f,\cD_0}(x,y)\in\N_0$ and let $X,\,Y$ be $(m+1)$-chambers containing $x$ and $y$, respectively. By the definition of $m_{f,\cD_0}(x,y)$, we have $X\cap Y=\emptyset$.
	Since $m_{f,\cD_0}(x,z)>m+k_0>0$, there exist (cf.~Remark~\ref{r: remarks on m(x,y)}~(a)) non-disjoint $(m_{f,\cD_0}(x,z))$-chambers $X',\,Z'$ with $x\in X',\,z\in Z'$. Likewise, choose non-disjoint $(m_{f,\cD_0}(z,y))$-chambers $Y'',\,Z''$ with $y\in Y'',\,z\in Z''$.
	Consider the connected set $A\=X'\cup Z'\cup Z''\cup Y''$ that intersects both $X$ and $Y$ (which are disjoint $(m+1)$-chambers). By Corollary~\ref{c: CBC: connected set in flower}~(iii), $f^{m+1}(A)=f^{m+1}(X')\cup f^{m+1}(Y'')\cup f^{m+1}(Z')\cup f^{m+1}(Z'')$ {\jsOpSd} of $\cD_0$. This contradicts $J\bigl(\bigcup_{i \geq  k_0}\cD_i,\cD_0\bigr)>10$ since $f^m(A)$ is a union of four cells in $\bigcup_{i \geq  k_0}\cD_i$ (which follows from $\min\{m_{f,\cD_0}(x,z),m_{f,\cD_0}(z,y)\}>m+k_0$). 
	
	\smallskip
	
	(ii) 
	Let $M=M(\cC_0,\coverF(\cD_0))\in\N$ be the constant given by Lemma~\ref{l: CBC: cover cell by M flowers}. First, we show that for all $x,\,y\in\mfd^n$, $m_{f,\cC_0}(x,y) \leq  m_{f,\cD_0}(x,y)+(4M-2)k_0$, where $k_0\in\N$ is the constant in (i).
	 
	Fix arbitrary $x,\,y\in\mfd^n$. It suffices to consider the case $m_{f,\cC_0}(x,y)>0$. Suppose $\tm\in\N_0$ and $\tX,\,\tY\in\cC_{\tm}^{\top}$ satisfy $x\in \tX$, $y\in \tY$, and $\tX\cap \tY\neq\emptyset$. By Lemma~\ref{l: CBC: cover cell by M flowers}, there exists a cover of $\tX\cup\tY$ by no more than $2M$ elements in $\coverF(\cD_{\tm})$, and thus,
    by the connectedness of $\tX\cup\tY$ and Proposition~\ref{p: c-flower: structure}~(ii), there exist $X_1,\,\dots,\, X_{4M}\in\cD_{\tm}^{\top}$ satisfying $x\in X_1$, $y\in X_{4M}$, and $X_{i}\cap X_{i+1}\neq\emptyset$ for each $i\in\{1,\,\dots,\,4M-1\}$. Set $x_1\=x$, $x_{4M}\=y$, and pick $x_i\in X_i$ for each $i\in\{2,\,\dots,\,4M-1\}$. Clearly $\tm\leq m_{f,\cD_0}(x_i,x_{i+1})$ for each $i\in\{1,\,\dots,\,4M-1\}$.
	
	Now we show $\tm \leq  m_{f,\cD_0}(x_1,x_{i})+(i-2)k_0$ for each $i\in\{2,\,\dots,\,4M\}$.
	Indeed, $\tm\leqslant m_{f,\cD_0}(x_1,x_2)$, and we can conclude using the following induction argument: for each $i\in\{3,\,\dots,\,4M\}$, if $\tm \leq  m_{f,\cD_0}(x_1,x_{i-1})+(i-3)k_0$, then since $\tm\leq m_{f,\cD_0}(x_{i-1},x_{i})$ and by (i), we have
	\begin{equation*}
	\tm-(i-3)k_0 \leq \min\{m_{f,\cD_0}(x_1,x_{i-1}),m_{f,\cD_0}(x_{i-1},x_{i})\} \leq  m_{f,\cD_0}(x_1,x_{i})+k_0,
    \end{equation*}
	and thus $\tm \leq  m_{f,\cD_0}(x_1,x_{i})+(i-2)k_0$.
    
	Hence, 
	$\tm \leq  m_{f,\cD_0}(x_1,x_{4M})+(4M-2)k_0=m_{f,\cD_0}(x,y)+(4M-2)k_0$. Denote $k_1'\=(4M-2)k_0$. Since $\tm\in\N_0$ and non-disjoint $\tX,\,\tY\in\cC_{\tm}^{\top}$ satisfying $x\in \tX$, $y\in \tY$ are arbitrary, we have $m_{f,\cC_0}(x,y) \leq  m_{f,\cD_0}(x,y)+k'_1$. 
	By the same arguments, there exists $k''_1\in\N$ such that $m_{f,\cD_0}(x,y) \leq  m_{f,\cC_0}(x,y)+k''_1$ for all $x,\,y\in\mfd^n$.
	Then $k_1\=\max\{k'_1,\,k''_1\}$ is the desired constant.
\end{proof}

\subsubsection{Existence and properties of visual metrics}
\label{Ap: existence of visual metric}

Now we are ready to establish the existence of visual metrics (see Definition~\ref{d: visual metrics}) for expanding {\itCel} branched covers.

\begin{theorem}[Existence of visual metrics]\label{t: visual metric: existence}
	Let $f\:\mfd^n\to\mfd^n$ be an expanding {\itCel} branched cover on a {\clcnmtr} $n$-manifold. Then $f$ has a visual metric.
\end{theorem}
\begin{proof}
	Let $\{\cD_m\}_{m\in\N_0}$ be a {\celSeq} of $f$.
	Fix $\lambda>1$ and set $q(x,y)\=\lambda^{-m_{f,\cD_0}(x,y)}$ for all $x,\,y\in\mfd^n$. For all 
    $x,\,y\in\mfd^n$, it can be directly checked that $q(x,y)=q(y,x)$, and that $q(x,y)=0$ if and only if $x=y$. In addition, by Lemma~\ref{l: m(x,y) properties: hyperbolic}~(i), there exists $k_0\in\N$ such that
	$q(x,y) \leq \lambda^{k_0}(q(x,z)+q(z,y))$
	for all $x,\,y,\,z\in\mfd^n$. Having these properties of $q(x,y)$, we use \cite[Proposition~14.5]{He01} to assert that there exist constants $\epsilon>0$ and $C>1$, and a metric $\varrho$ on $\mfd^n$ satisfying
	$C^{-1}(q(x,y))^\epsilon\le \varrho(x,y)\le C(q(x,y))^\epsilon$
	for all $x,\,y\in\mfd^n$.
	Such a metric $\varrho$ is a visual metric for $f$.
\end{proof}

In what follows, we give some elementary properties of visual metrics. 

	\begin{lemma}\label{l: visual metrics: compatibale topo, and lipschitz}
		Let $f\:\mfd^n\to\mfd^n$ be an expanding {\itCel} map on a {\clcnmtr} $n$-manifold, and $\varrho$ be a visual metric for $f$. Then 
		\begin{enumerate}[label=(\roman*),font=\rm]
        	\smallskip
			\item $\varrho$ induces the given topology on $\mfd^n$ and
			\smallskip
			\item $f\:(\mfd^n,\varrho)\to(\mfd^n,\varrho)$ is a Lipschitz map.
		\end{enumerate}
	\end{lemma}
    We omit the proof of the above lemma,
    where (i) follows from standard arguments, and (ii) follows immediately from Lemma~\ref{l: m(x,y) property: iteration}.
    
	The following lemma shows that the definition of visual metrics of an expanding {\itCel} branched cover does not rely on the choice of the {\celSeq}.
    
	\begin{lemma}\label{l: visual metric: same expansion factor}
		Let $f\:\mfd^n\to\mfd^n$ be an expanding {\itCel} branched cover on a {\clcnmtr} $n$-manifold. Let $\varrho$ be a visual metric for $f$ and $\Lambda>1$ be an expansion factor of $\varrho$.
		Then for each {\celSeq} $\{\cC_m\}_{m\in\N_0}$ of $f$, there exists $\tC>1$ such that for all $x,\,y\in\mfd^n$,
		\begin{equation*}
		\tC^{-1}\Lambda^{-m_{f,\cC_0}(x,y)} \leq \varrho(x,y) \leq \tC\Lambda^{-m_{f,\cC_0}(x,y)}.
	\end{equation*}
	\end{lemma}
	\begin{proof}
        Let $\{\cC_m\}_{m\in\N_0}$ be a {\celSeq} of $f$.
		Since $\varrho$ is a visual metric, there exists a {\celSeq} $\{\cD_m\}_{m\in\N_0}$ of $f$ and a constant $C>1$ satisfying $C^{-1}\Lambda^{-m_{f,\cD_0}(x,y)} \leq \varrho(x,y) \leq  C\Lambda^{-m_{f,\cD_0}(x,y)}$
		for all $x,\,y\in\mfd^n$.
		Combining this with Lemma~\ref{l: m(x,y) properties: hyperbolic}~(ii), we get a constant $k_1\in\N$ such that
		$
		\bigl(C\Lambda^{k_1}\bigr)^{-1}\Lambda^{-m_{f,\cC_0}(x,y)} \leq \varrho(x,y) \leq \bigl(C\Lambda^{k_1}\bigr)\Lambda^{-m_{f,\cC_0}(x,y)}
		$ 
        for all $x,\, y\in\mfd^n$.
		Then $\tC\=C\Lambda^{k_1}$ is the desired constant.
	\end{proof}

	Consequently, the expansion factor of a visual metric is unique.
	\begin{cor}
		Let $f\:\mfd^n\to\mfd^n$ be an expanding {\itCel} branched cover on a {\clcnmtr} $n$-manifold. Let $\varrho$ be a visual metric for $f$. Then the expansion factor of $\varrho$ is unique.
	\end{cor}
	\begin{proof}
		Suppose that $\tLambda > \Lambda>1$ are two expansion factors of $\varrho$. Let $\{\cD_m\}_{m\in\N_0}$ be a {\celSeq} of $f$ and denote $m(x,y)\=m_{f,\cD_0}(x,y)$ for all $x,\,y\in\mfd^n$. By Lemma~\ref{l: visual metric: same expansion factor}, there are constants $C>1$ and $\tC>1$ such that 
		$C^{-1}\Lambda^{-m(x,y)} \leq \varrho(x,y) \leq  C\Lambda^{-m(x,y)}$ and 
        $\tC^{-1}\tLambda^{-m(x,y)} \leq \varrho(x,y) \leq  \tC\tLambda^{-m(x,y)}$
    		for all $x,\,y\in\mfd^n$. In particular, $\tC\tLambda^{-m(x,y)} \geq C^{-1}\Lambda^{-m(x,y)}$ for all $x,\,y\in\mfd^n$. However, by choosing appropriate $x$ and $y$, we can make $m(x,y)$ so large that $\tC\tLambda^{-m(x,y)} < C^{-1}\Lambda^{-m(x,y)}$, which is a contradiction.	 
	\end{proof}

    The next lemma gives estimates of the size and distance (with respect to a visual metric) of chambers in a {\celSeq}. This lemma is applied in Section~\ref{sct: QS uniformization of visual} to show the branched quasisymmetric properties (with respect to visual metrics).
    
    \begin{lemma}\label{l: visual metric: metric properties of cells}
    	Let $f\:\mfd^n\to\mfd^n$ be an expanding {\itCel} branched cover on a {\clcnmtr} $n$-manifold. Equip $\mfd^n$ with a visual metric $\varrho$ for $f$ and let $\Lambda>1$ be the expansion factor for $\varrho$. Then for each {\celSeq} $\{\cD_m\}_{m\in\N_0}$ of $f$, there exists $C'>1$ such that for all $m\in\N_0$ and $X,\, Y \in\cD_m^{\top}$ we have
    		$(1/C')\Lambda^{-m} \leq \diam X \leq  C'\Lambda^{-m}$ and if $X\cap Y=\emptyset$ then $\dist(X,Y) \geq  (1/C')\Lambda^{-m}$.
    \end{lemma}
    \begin{proof}
    	Let $\{\cD_m\}_{m\in\N_0}$ be a {\celSeq} of $f$.
    	By Lemma~\ref{l: visual metric: same expansion factor}, let $C>1$ be such that $C^{-1}\Lambda^{-m_{f,\cD_0}(x,y)} \leq \varrho(x,y) \leq  C\Lambda^{-m_{f,\cD_0}(x,y)}$ for all $x,\,y\in\mfd^n$. By Lemma~\ref{l: expansion: J_m--->+infty}, choose $k\in\N$ for which $J\bigl(\bigcup_{i \geq  k}\cD_i,\cD_0\bigr)>2$. Set $b\=\min\bigl\{\diam X_0:X_0\in\cD_0^{\top}\bigr\}$ and $C'\=\max \bigl\{ C\Lambda^k, \, C^2/b \bigr\}$.
        
        Fix $m\in\N_0$ and $X, \, Y \in\cD_m^{\top}$. Let $x\in X$ and $y\in Y$ be arbitrary. 
        
        By (\ref{e: the constant m(x,y)}) and Lemma~\ref{l: m(x,y) property: iteration}, $m\leq m_{f,\cD_0}(x,y) \leq m_{f,\cD_0}(f^m(x),f^m(y)) + m$, and thus $\frac{\varrho(f^m(x),f^m(y))}{C^{2}\Lambda^{m}}\leq\varrho(x,y) \leq  \frac{C}{\Lambda^{m}}$. This, combined with the fact that $f^m|_X\:X\to f^m(X)$ is a bijection, implies
        $C^{-2}\Lambda^{-m}\diam(f^m(X)) \leq \diam X \leq C\Lambda^{-m}$. Thus, $(1/C')\Lambda^{-m} \leq \diam X \leq  C'\Lambda^{-m}$.
        
    	Now assume $X\cap Y=\emptyset$. Suppose that $m(x,y) >  m+k$. Then there exist $j >  k$ and $X',\,Y'\in\cD_{m+j}^{\top}$ satisfying $x\in X'$, $y\in Y'$, and $X'\cap Y'\neq\emptyset$.
    	Then $X'\cup Y'$ is a connected set that meets both $X$ and $Y$, and thus, by Corollary~\ref{c: CBC: connected set in flower}~(iii), $f^m(X')\cup f^m(Y')$ joins opposite sides of $\cD_0$, which contradicts $J\bigl(\bigcup_{i \geq  k}\cD_i,\cD_0\bigr)>2$.  Thus, $m_{f,\cD_0}(x,y) \leq  m+k$ and $\varrho(x,y) \geq  C^{-1}\Lambda^{-m-k}$. It follows that $\dist(X,Y) \geq  C^{-1}\Lambda^{-k}\Lambda^{-m}\geq  (1/C')\Lambda^{-m}$.
    \end{proof}


\end{document}